\documentclass[pdftex,oneside,11pt,reqno]{amsart}

\usepackage{amssymb,enumerate,amsmath,amsthm,bbm,mdwlist,url,multirow,appendix,mathtools,hyperref,lscape,tabularx,caption}
\usepackage[pdftex]{graphicx}
\usepackage[shortlabels]{enumitem}

\usepackage{float}
\restylefloat{table}

\makeatletter
\def\@tocline#1#2#3#4#5#6#7{\relax
  \ifnum #1>\c@tocdepth % then omit
  \else
    \par \addpenalty\@secpenalty\addvspace{#2}%
    \begingroup \hyphenpenalty\@M
    \@ifempty{#4}{%
      \@tempdima\csname r@tocindent\number#1\endcsname\relax
    }{%
      \@tempdima#4\relax
    }%
    \parindent\z@ \leftskip#3\relax \advance\leftskip\@tempdima\relax
    \rightskip\@pnumwidth plus4em \parfillskip-\@pnumwidth
    #5\leavevmode\hskip-\@tempdima
      \ifcase #1
       \or\or \hskip 1em \or \hskip 2em \else \hskip 3em \fi%
      #6\nobreak\relax
    \dotfill\hbox to\@pnumwidth{\@tocpagenum{#7}}\par
    \nobreak
    \endgroup
  \fi}
\makeatother

\theoremstyle{definition}
\newtheorem{definition}{Definition}[section]%Extra square-bracket argument achieves that the numbering is the same as for definition (single uniform counter). 
\theoremstyle{theorem}
\newtheorem{proposition}{Proposition}[section]
\newtheorem{lemma}{Lemma}[section]
\newtheorem{theorem}{Theorem}[section]
\newtheorem{assumption}{Assumption}[section]
\newtheorem{corollary}{Corollary}[section]
\theoremstyle{remark}
\newtheorem{remark}{Remark}[section]
\newtheorem{example}{Example}[section]
\numberwithin{equation}{section}
\numberwithin{definition}{section}
\def\EE{\mathsf E}
\def\RR{\mathbb R}
\def\ZZ{\mathbb Z}

\def\Zh{\mathbb{Z}_h}

\def\Cr{\mathbb{C}^{\rightarrow}}
\def\Crr{\overline{\mathbb{C}^{\rightarrow}}}

\def\Cll{\overline{\mathbb{C}^{\leftarrow}}}

\def\Zq{Z^{(q)}}
\def\Wq{W^{(q)}}
\def\sc{W}

\def\Wqprime{W^{(q)\prime}}

\def\ind{\mathbbm{1}_{[-V,0)}}
\def\indd{\mathbbm{1}_{[-1,0)}}
\def\supp{\mathrm{supp}}

\def\measure{\lambda}
\def\measurenu{\nu}
\def\diffusion{\sigma^2}
\def\drift{\mu}

\def\diffW{\Delta_{W}^{(q)}(x,h)}
\def\diffWn{\Delta_{W}^{(q)}(x,h_n)}
\def\diffZ{\Delta_{Z}^{(q)}(x,h)}
\def\diffZn{\Delta_{Z}^{(q)}(x,h_n)}

\makeatletter
\newcommand{\mytag}[2]{%
  \text{#1}%
  \@bsphack
  \begingroup
    \@onelevel@sanitize\@currentlabelname
    \edef\@currentlabelname{%
      \expandafter\strip@period\@currentlabelname\relax.\relax\@@@%
    }%
    \protected@write\@auxout{}{%
      \string\newlabel{#2}{%
        {#1}%
        {\thepage}%
        {\@currentlabelname}%
        {\@currentHref}{}%
      }%
    }%
  \endgroup
  \@esphack
}
\makeatother

\begin{document}
\title{Markov chain approximations to scale functions of L\'evy processes}

\author{Aleksandar Mijatovi\'{c}}
\address{Department of Mathematics, Imperial College London, UK}
\email{a.mijatovic@imperial.ac.uk}

\author{Matija Vidmar}
\address{Department of Statistics, University of Warwick, UK}
\email{m.vidmar@warwick.ac.uk}

\author{Saul Jacka}
\address{Department of Statistics, University of Warwick, UK}
\email{s.d.jacka@warwick.ac.uk}

\thanks{The support of the Slovene Human Resources Development and Scholarship Fund under contract number 11010-543/2011 is acknowledged.}

\begin{abstract}
We introduce a general algorithm for the computation of the
scale functions of a spectrally negative L\'evy process $X$,
based on a natural weak approximation of $X$ via upwards
skip-free continuous-time Markov chains with stationary
independent increments. The algorithm consists of evaluating a finite linear 
recursion with its (nonnegative) coefficients given explicitly in terms 
of the L\'evy triplet of $X$.
Thus it is easy to implement and numerically stable.
Our main result establishes sharp rates of convergence 
of this algorithm providing an explicit link between the semimartingale
characteristics of $X$ and its scale functions, not unlike the
one-dimensional It\^o diffusion setting, where scale functions are expressed
in terms of certain integrals of the coefficients of the governing
SDE.
\end{abstract}



\maketitle

% \setcounter{tocdepth}{3}
%\tableofcontents

\section{Introduction}\label{section:Introduction}
It is well-known that, for a spectrally negative L\'evy process $X$ \cite[Chapter~VII]{bertoin} \cite[Section~9.46]{sato}, fluctuation theory in terms of the two families of scale functions, $(\Wq)_{q\in [0,\infty)}$ and $(\Zq)_{q\in[0,\infty)}$, has been developed \cite[Section~8.2]{kyprianou}. Of particular importance is the function $\sc:=W^{(0)}$, in terms of which the others may be defined, and which features in the solution of many important problems of applied probability \cite[Section 1.2]{kuznetsovkyprianourivero}. It is central to these applications to be able to evaluate scale functions for any spectrally negative L\'evy 
process $X$.

The goal of the present paper is to define and analyse a very simple novel algorithm for computing $W$. Specifically, to compute $W(x)$ for some $x>0$, choose small $h>0$ such that $x/h$ is an integer. Then the approximation $\sc_h(x)$ to $\sc(x)$ is given by the recursion:
\begin{eqnarray}
\label{eq:LinRecursion}
\sc_h(y+h)=\sc_h(0)+\sum_{k=1}^{y/h+1}\sc_h(y+h-kh)\frac{\gamma_{-kh}}{\gamma_{h}},\qquad \sc_h(0)=(\gamma_hh)^{-1}
\end{eqnarray}
for $y=0,h,2h,\ldots,x-h$, where the coefficients $\gamma_h$ and $(\gamma_{-kh})_{k\geq 1}$ are expressible directly in terms of the L\'evy measure $\measure$,
(possibly vanishing) Gaussian component $\sigma^2$ and drift $\mu$ of the L\'evy process $X$, as follows. Let:
\footnotesize
\begin{eqnarray*}
\tilde{\sigma}^2_h :=  
\frac{1}{2h^2}\left(\diffusion+\int_{[-h/2,0)}\!\!\!\!\!\!\!\!y^2\ind(y)\measure(dy)\right),  \quad
%+\int_{[-h/2,0)}\frac{y^2}{2h^2}\mathbbm{1}_{[-V,0)}(y)\measure(dy),\quad
%\\
\tilde{\drift}^h  :=  \frac{1}{2h} 
\left(\mu +h\sum_{k\in\mathbb{N}}k\measure\left(\left[\left(-k-\frac{1}{2}\right)h,\left(-k+\frac{1}{2}\right)h\right)\cap[-V,0)\right)
%\sum_{y\in \Zh^{--}}y \measure\left(A_y^h\cap [-V,0)\right) 
%[-\min(hn+h/2,V),-hn+h/2)
\right),
%\>\measure[-hn-h/2
%\vee V,h(-n+\frac{1}{2}))\right)
\end{eqnarray*}
\normalsize
%and
%note that 
%%the sum in 
%$\tilde{\drift}^h$ is 
%non-negative for 
%$h\in(0,h_\star)$.
%always finite as 
%$0\notin\Zh^{--}$
%$[a,b)=\emptyset$ for $a>b$.
%i.e. 
where $V$ equals $0$ or
$1$ according as to whether $\measure$ is finite or infinite, and the drift $\mu$ is relative to the cut-off function $\tilde{c}(y):=y\mathbbm{1}_{[-V,0)}(y)$ (see Eq.~\ref{eq:laplace_exponent} for the Laplace exponent of $X$); remark 
$\tilde{\sigma}^2_h  =  \sigma^2/2h^2$
and
$\tilde{\mu}^h  =  \mu/2h$, when $V=0$. Then the coefficients in~\eqref{eq:LinRecursion}
are given by:
\begin{eqnarray}
\label{eq:Explicit_1}
\gamma_h:=\tilde{\sigma}^2_h+\mathbbm{1}_{(0,\infty)}(\diffusion)\tilde{\drift}^h
+\mathbbm{1}_{\{0\}}(\diffusion)2\tilde{\drift}^h,& & \quad 
\gamma_{-h}:=\tilde{\sigma}^2_h-\mathbbm{1}_{(0,\infty)}(\diffusion)\tilde{\drift}^h+\measure(-\infty,-h/2]\\
%q_k=\measure[-kh-h/2,-kh+h/2)/\gamma,\quad k\geq 2.
\gamma_{-kh}:=\measure(-\infty,-kh+h/2],& &\quad\text{where } k\geq 2.
\label{eq:Explicit_2}
\end{eqnarray}

Indeed, the algorithm just described is based on a purely probabilistic idea of weak approximation: for small positive $h$, $X$ is approximated by what is a random walk $X^h$ on a lattice with spacing $h$, skip-free to the right, and embedded into continuous time as a compound Poisson process (see Definition~\ref{def:USF_Levy_chain}). Then, in recursion~\eqref{eq:LinRecursion}, $\sc_h$ is the scale function associated to $X^h$ --- it plays a probabilistically analogous r\^ole for the process $X^h$, as does $\sc$ for the process $X$. Thus $\sc_h$ is computed as an approximation to $\sc$ (see Corollary~\ref{proposition:calculating_scale_functions}).

When it comes to existing methods for the evaluation of $\sc$, note that analytically $\sc$ is characterized via its Laplace transform $\widehat{W}$, $\widehat{W}$ in turn being a certain rational function of the Laplace exponent $\psi$ of $X$. However, already $\psi$ need not be given directly in terms of elementary/special functions, and less often still is it possible to obtain closed-form expressions for $\sc$ itself. The user is then faced with a Laplace inversion algorithm  \cite{CohenAM} \cite[Chapter~5]{kuznetsovkyprianourivero}, which (i) necessarily involves the evaluation of $\psi$, typically at complex values of its argument and requiring high-precision arithmetic due to numerical instabilities; (ii) says little about the dependence of the scale function on the L\'evy triplet of $X$ (recall that $\psi$ depends on a parametric complex integral of the L\'evy measure, making it hard to discern how a perturbation in the L\'evy measure influences the values taken by the scale function); and (iii) being a numerical approximation, fails \emph{a priori} to ensure that the computed values of the scale function are probabilistically meaningful (e.g. given an output of a numerical Laplace inversion, it is not necessary that the formulae for, say, exit probabilities, involving $W$, should yield values in the interval $[0,1]$).

By contrast, it follows from~\eqref{eq:LinRecursion} and the discussion following, that our proposed algorithm (i)~requires no evaluations of the Laplace exponent of $X$ and is numerically stable, as it operates in nonnegative real arithmetic \cite[Theorem~7]{panjer}; (ii) provides an explicit link between the deterministic semimartingale characteristics of $X$, in particular its L\'evy measure,  
and the scale function $W$; and (iii) yields probabilistically consistent outputs. Further, the values of $\sc_h$ are so computed by a simple finite linear recursion and, as a by-product of the evaluation of $W_h(x)$, values $W_h(y)$ for all the grid-points $y=0,h,2h,\ldots,x-h,x$, are obtained (see Matlab code for the algorithm in~\cite{Scale_Function_Code}), which is useful in applications (see Section~\ref{section:numerical_illustrations} below). 

Our main results will (I) show that $\sc_h$ converges to $\sc$ pointwise, and uniformly on the grid with spacing $h$ (if bounded away from $0$ and $+\infty$), for any spectrally negative L\'evy process,
and (II) establish sharp rates for this convergence under a mild assumption on the L\'evy measure.

Due to the explicit connection between the coefficients appearing in \eqref{eq:LinRecursion} and the L\'evy triplet of $X$, \eqref{eq:LinRecursion} also has the spirit of its one-dimensional It\^o diffusion analogue, wherein the computation of the scale function requires numerical evaluation of certain integrals of the coefficients of the SDE driving said diffusion (for the explicit formulae of the integrals see e.g.~\cite[Chapters~2 and~3]{BorodinSalminen}). Indeed, we express $\sc$ as a single limit, as $h\downarrow 0$, of nonnegative terms explicitly given in terms of the L\'evy triplet. This is more direct than the Laplace inversion of a rational transform of the Laplace exponent, and hence may be of purely theoretical significance (see Remark~\ref{remark:numerics:three} on how the scale functions are affected by a perturbation of the L\'evy measure, following directly from a transformed form of \eqref{eq:LinRecursion}). 

Finally, note that an algorithm, completely analogous to~\eqref{eq:LinRecursion}, for the computation of the scale functions $\Wq$, and also $\Zq$, $q\geq0$, follows from our results (see Corollary~\ref{proposition:calculating_scale_functions}, Eq.~\eqref{equation:recursion:Wq} and~\eqref{equation:recursion:Zq}) 
and presents no further difficulty for the analysis of convergence  (see Theorem~\ref{theorem:rates_for_scale_functions} below). Indeed, our discretization allows naturally to approximate other quantities involving scale functions, which arise in application: the derivatives of $\Wq$ by difference quotients of $\Wq_h$; the integrals of a continuous (locally bounded) function against $d\Wq$ by its integrals against $d\Wq_h$; expressions of the form $\int_0^xF(y,\Wq(y))dy$, where $F$ is continuous locally bounded, by the sums $\sum_{k=0}^{\lfloor x/h\rfloor -1}F(kh,\Wq_h(kh))h$ etc. (See Section~\ref{section:numerical_illustrations}  for examples.)  
%The study of the convergence properties of such functionals, however, goes beyond the desired scope of this article.

\subsection{Overview of main 
%the 
results}
\label{subsec:Main_REsults_description}
The key idea leading to the algorithm 
in~\eqref{eq:LinRecursion} 
is best described by the following two steps: 
(i) approximate the spectrally negative L\'evy process
$X$
by a
%a family 
%$(X^h)_{h\in (0,h_\star)}$,
%for some positive
%$h_\star$,
%defined for all small
%enough positive $h<h_\star$, which are 
continuous-time Markov chain (CTMC) 
$X^h$ 
with state space
$\Zh:=\{hk\!\!:k\in\ZZ\}$
($h\in (0,h_\star)$ for some $h_\star>0$),
as described in Subsection~\ref{subsection:the_approximation}; 
(ii) find an algorithm for computing the scale functions
of the chain $X^h$. 
The approximation in Subsection~\ref{subsection:the_approximation}
implies that 
$X^h$
is a compound Poisson (CP) process, which is not spectrally negative.
However, since the corresponding jump chain of $X^h$ is a
skip-free to the right
$\Zh$-valued random walk, 
%which is 
%Furthermore, the continuous-time chain 
%However,
%for every chain
%$X^h$
it is possible to introduce (right-continuous, nondecreasing) scale functions
$(\Wq_h)_{q\geq 0}$ and $(\Zq_h)_{q\geq 0}$ (with measures 
$d\Wq_h$ and $d\Zq_h$ supported in $\Zh$), in 
%complete 
analogy to the spectrally negative case.
Moreover, as described in Corollary~\ref{proposition:calculating_scale_functions}, 
%this weak approximation gives rise to a
a straightforward recursive algorithm is readily available for evaluating \emph{exactly} any function in the families 
$(\Wq_h)_{q\geq 0}$ and
$(\Zq_h)_{q\geq 0}$ at any point. More precisely, it emerges, that for each $x\in \Zh$, $\Wq_h(x)$ (resp. $\Zq_h(x)$) obtains as a \emph{finite} linear combination of the preceding values $\Wq_h(y)$ (resp. $\Zq_h(y)$) for $y\in \{0,h,\ldots,x-h\}$; with the starting value $\Wq_h(0)$ (resp. $\Zq(0)$) being known explicitly. This is in spite of the fact that the state space of the L\'evy process $X^h$ is in fact the \emph{infinite} lattice $\Zh$. 

In order to precisely describe the rates of convergence of the 
algorithm in~\eqref{eq:LinRecursion}, we introduce some notation. Fix $q\geq 0$ and define for $K,G$ bounded
subset of $(0,\infty)$: 
$$\Delta_W^K(h):=\sup_{x\in \Zh\cap
K}\left\vert \Wq_h(x-\delta^0h)-\Wq(x)\right\vert\text{ and
}\Delta_Z^G(h):=\sup_{x\in \Zh\cap G}\left\vert \Zq_h(x)-\Zq(x)\right\vert,$$
where $\delta^0$ equals $0$ 
if
$X$ has sample paths of finite
variation and
$1$ otherwise. 
We further introduce: 
$$\kappa(\delta):=\int_{[-1,-\delta)}\vert y\vert\measure(dy),\qquad\text{for any $\delta\geq 0$.}$$ 
If the jump part of 
$X$
has paths of infinite variation, 
i.e. in the case the equality $\kappa(0)=\infty$ holds, 
we assume
(throughout the paper we shall make it explicit when this assumption is in effect): 
\begin{assumption}\label{assumption:salient}
There exists $\epsilon\in (1,2)$ with:
\begin{enumerate}[(1)]
\item\label{assumption:salient:one}  $\limsup_{\delta\downarrow 0}\delta^{\epsilon}\measure(-1,-\delta)<\infty$ and
\item\label{assumption:salient:two} $\liminf_{\delta\downarrow 0}\int_{[-\delta,0)}x^2\measure(dx)/\delta^{2-\epsilon}>0$. 
\end{enumerate}
\end{assumption}
Note that this is a fairly mild condition, fulfilled if e.g. $\measure(-1,-\delta)$ ``behaves as'' $\delta^{-\epsilon}$, as $\delta\downarrow 0$; for a precise statement see Remark~\ref{remark:diffusion_positive:asymptotic}. 

Here is now our main result:

\begin{theorem}\label{theorem:rates_for_scale_functions}
Let $K$ and $G$ be bounded subsets of $(0,\infty)$, $K$ bounded away from zero when $\diffusion=0$. If $\kappa(0)=\infty$, suppose further that Assumption~\ref{assumption:salient} is fulfilled. Then the rates of convergence of the scale functions are summarized by the following table: 
\begin{center}
\begin{tabular}{|c|c|}\hline
$\measure(\mathbb{R})=0$ & $\Delta_W^K(h)=O(h^2)$ and $\Delta_Z^G(h)=O(h)$\\\hline
$0<\measure(\mathbb{R})$ \& $\kappa(0)<\infty$ & $\Delta_W^K(h)+\Delta_Z^G(h)=O(h)$\\\hline
$\kappa(0)=\infty$ & $\Delta_W^K(h)+\Delta_Z^G(h)=O(h^{2-\epsilon})$\\\hline
\end{tabular}
\end{center}
Moreover, the rates so established are sharp in the sense that for each of the
three entries in the table above, examples of spectrally negative
L\'evy processes are constructed for which the rate of convergence is no better than 
stipulated.
\end{theorem}
%Then the obtained convergence rates may be summarized in
%Table~\ref{table:rates_scale_functions}. 
\begin{remark}%\leavevmode \\
\noindent (1) The rates of convergence depend on the behaviour of the tail of the L\'evy measure at the origin; by contrast behaviour of Laplace inversion algorithms tends to be susceptible to the degree of smoothness of the scale function (for which see \cite{kyprianou:smoothness}) itself \cite{abate_unified}.\\
\noindent (2) More exhaustive and at times general statements are to be found in
Propositions~\ref{proposition:convergence:BM+drift:Wq}--%~\ref{proposition:convergence:BM+drift:Zq},~\ref{proposition:Wqconvergence:non-trivial_diffusion},~\ref{proposition:Zqconvergence:non-trivial_diffusion} and ~\ref{proposition:Wqconvergence:trivial_diffusion_FV}
\ref{proposition:convergence:diffusion_zero_infinite_variation}. In particular, the case $\diffusion>0$ and $\kappa(0)=\infty$ does not require
Assumption~\ref{assumption:salient} to be fulfilled, although the statement of
the convergence rate is more succinct under its proviso.\\
\noindent (3) The proof of Theorem~\ref{theorem:rates_for_scale_functions} consists of studying the differences of the integral representations of the scale functions. 
%The difficulty in the analysis comes from the fact 
The integrands, however, decay only according to some power law, making the analysis much more involved than was the case in \cite{vidmarmijatovicsaul}, where the corresponding decay was exponential. In particular, one cannot, in the pure-jump case, directly apply the integral triangle inequality. The structure of the proof is explained in detail in Subsection~\ref{subsection:method_preliminary_etc}. \\
\noindent (4) Since scale functions often appear in applications (for which see Section~\ref{subsections:scales:applications}
below) in the form $\Wq(x)/\Wq(y)$ ($x,y>0$, $q\geq 0$), we note that the rates from Theorem~\ref{theorem:rates_for_scale_functions}
transfer directly to such quotients, essentially because 
$\Wq_h(y)\to \Wq(y)\in (0,\infty)$, as $h\downarrow 0$, and since for all $h\in (0,h_\star)$, $\frac{1}{\Wq(y)}-\frac{1}{\Wq_h(y)}=\frac{\Wq_h(y)-\Wq(y)}{\Wq(y)\Wq_h(y)}$.\\
\noindent (5) For a result concerning the derivatives of $\Wq$  see Subsection~\ref{subsection:functionals_of_scale_fncs_conv}.
\end{remark}
%
%\begin{table}[h]
%\begin{tabular}{|c|c|}\hline
%$\measure(\mathbb{R})=0$ & $\Delta_W^K(h)=O(h^2)$ and $\Delta_Z^G(h)=O(h)$\\\hline
%$0<\measure(\mathbb{R})$ \& $\kappa(0)<\infty$ & $\Delta_W^K(h)+\Delta_Z^G(h)=O(h)$\\\hline
%$\kappa(0)=\infty$ & $\Delta_W^K(h)+\Delta_Z^G(h)=O(h^{2-\epsilon})$\\\hline
%\end{tabular}
%\caption{\textbf{Rates of convergence for scale functions}. $K$ and $G$ are bounded subsets of $(0,\infty)$; $K$ bounded away from zero, when $\diffusion=0$.}
%\label{table:rates_scale_functions}
%\end{table} For a precise account see again the statements of
%Propositions~\ref{proposition:convergence:BM+drift:Wq},~\ref{proposition:convergence:BM+drift:Zq},~\ref{proposition:Wqconvergence:non-trivial_diffusion},~\ref{proposition:Zqconvergence:non-trivial_diffusion} and ~\ref{proposition:Wqconvergence:trivial_diffusion_FV}-~\ref{proposition:convergence:diffusion_zero_infinite_variation}.

\subsection{Overview of the literature and of the applications of scale functions}\label{subsections:scales:applications}
For the general theory of spectrally negative L\'evy processes and their scale functions we refer to~\cite[Chapter~8]{kyprianou} and~\cite[Chapter~VII]{bertoin}, while an excellent account of available numerical methods for computing them can be found in~\cite[Chapter~5]{kuznetsovkyprianourivero}. %Indeed, except possibly for special subclasses of the spectrally negative family, these are one or another of the many Laplace inversion methods, which have stood the test of time. They require, thus, the evaluation of the Laplace exponent at complex, rarely only real, values of its argument. This makes our proposed approach qualitatively different from the techniques in the literature. 
Examples, few, but important, of processes when the scale functions 
\emph{can} be given analytically, appear e.g. in \cite{hubalek}; and in certain cases it is 
possible to construct them \emph{indirectly} \cite[Chapter~4]{kuznetsovkyprianourivero} (i.e. not starting from the basic datum, which we consider here to be the characteristic triplet of 
$X$). Finally, in the special case when $X$ is a positive drift minus a compound Poisson subordinator, we note that numerical schemes for (finite time) ruin/survival probabilities (expressible in terms of scale functions), based on discrete-time Markov chain approximations of one sort or another, have been proposed in the literature (see \cite{vylder,dickson,cardoso,dicksongray} and the references therein).

In terms of applications of scale functions in applied probability, there are numerous identities concerning
boundary crossing problems and related path decompositions in which scale
functions feature~\cite[p.~100]{kuznetsovkyprianourivero}. They do so either (a)~indirectly (usually as Laplace transforms of quantities which are ultimately of
interest), or even (b)~directly (then typically, but not always, as probabilities in the form of quotients $\sc(x)/\sc(y)$). For examples of the latter see the two-sided exit problem \cite[Chapter VII, Theorem 8]{bertoin}; ruin probabilities \cite[p. 217, Eq.~(8.15)]{kyprianou} and the Gerber-Shiu measure \cite[Section~5.4]{gerber_shiu} in the insurance/ruin theory context; laws of suprema of continuous-state branching processes \cite[Proposition~3.1]{bingham_branching}; L\'evy measures of limits of continuous-state branching processes with immigration (CBI processes) \cite[Eq.~(3.7)]{keller_mijatovic}; laws of branch lengths in population biology \cite[Eq.~(7)]{lambert}; the Shepp-Shiryaev optimal stopping problem (solved for the spectrally negative case in \cite[Theorem~2, Eq.~(30)]{avram_exitproblems}); \cite[Proposition~1]{loeffen} for an optimal dividend control problem. A further overview of these and other applications of scale functions (together with their derivatives and the integrals $\Zq$), e.g. in queuing theory and fragmentation processes, may be found in \cite[Section 1.2]{kuznetsovkyprianourivero}, see also the references therein. A suite of identities involving Laplace transforms of quantities pertaining to the reflected process of $X$ appears in \cite{mijatovicpistorius}.

\subsection{Organisation of the remainder of the paper}
Section~\ref{section:setting_and_notation} gives the setting and fixes general notation. Section~\ref{section:upwards_skip_free} introduces upwards skip-free L\'evy chains (they being the continuous-time analogues of random walks, which are skip-free to the right), describes their scale functions and how to compute them. In Section~\ref{section:Convergence_of_scale_functions} we demonstrate pointwise convergence of the approximating scale functions to those of the spectrally negative L\'evy process. Then Section~\ref{section:convergence_rates} establishes the rate at which this convergence transpires. Finally, Section~\ref{section:numerical_illustrations} provides some numerical illustrations and further discusses the computational side of the proposed algorithm. Appendices~\ref{appendix:technical_lemmas} and~\ref{appendix:some_asymptotic_properties_of_measures_on_R} contain the proofs of technical results from Subsection~\ref{subsection:auxiliary_technical_results}, while Appendix~\ref{appendix:further_examples} provides some additional numerical examples. 
%Sections~\ref{subsection:some_technical_estimates_and_bounds} and~\ref{subsection:some_asymptotic_properties_at_zero}. 

\section{Setting and general notation}\label{section:setting_and_notation}
Throughout this paper we let $X$ be a spectrally negative L\'evy process (i. e. $X$ has stationary independent increments, is c\`adl\`ag, $X_0=0$ a.s., the L\'evy measure $\measure$ of $X$ is concentrated on $(-\infty,0)$ and $X$ does not have a.s. monotone paths). The Laplace exponent $\psi$ of $X$, defined via $\psi(\beta):=\log \EE[e^{\beta X_1}]$ ($\beta\in\{\gamma\in \mathbb{C}:\Re\gamma\geq 0\}=:\Crr$), can be expressed as (see e.g. \cite[p. 188]{bertoin}):
\begin{equation}\label{eq:laplace_exponent}
\psi(\beta)=\frac{1}{2}\diffusion\beta^2+\drift \beta+\int_{(-\infty,0)}\left(e^{\beta y}-\beta \tilde{c}(y)-1\right)\measure(dy),\quad \beta\in \Crr.
\end{equation}
The L\'evy triplet of $X$ is thus given by $(\diffusion,\measure,\mu)_{\tilde{c}}$, $\tilde{c}:=\mathrm{id}_\mathbb{R}\mathbbm{1}_{[-V,0)}$ with $V$ equal to either $0$ or $1$, the former only if $\int_{[-1,0)}\vert x\vert\measure(dx)<\infty$ (where $\mathrm{id}_\mathbb{R}$ is the identity on $\mathbb{R}$). Further, when the L\'evy measure satisfies $\int_{[-1,0)}\vert x\vert\measure(dx)<\infty$, we may always express $\psi$ in the form $\psi(\beta)=\frac{1}{2}\diffusion\beta^2+\drift_0 \beta+\int_{(-\infty,0)}\left(e^{\beta y}-1\right)\measure(dy)$ for $\beta\in\Crr$. If in addition $\diffusion=0$, then necessarily the drift $\mu_0$ must be strictly positive, $\mu_0>0$ \cite[p. 212]{kyprianou}. 

\subsection{The approximation}\label{subsection:the_approximation}
We now recall from \cite{vidmarmijatovicsaul}, specializing to the spectrally negative setting, the spatial discretisation of $X$ by the family of CTMCs $(X^h)_{h\in (0,h_\star)}$  (where $h_\star\in (0,+\infty]$). This family weakly approximates $X$ as $h\downarrow 0$. As in  \cite{vidmarmijatovicsaul} we will use two approximating schemes, scheme 1 and 2, according as $\diffusion>0$ or $\diffusion=0$. Recall that two different schemes are introduced since the case $\diffusion>0$ allows for a better (i.e. a faster converging) discretization of the drift term, but the case $\diffusion=0$ (in general) does not \cite[Paragraph~2.2.1]{vidmarmijatovicsaul}. Let also $V=0$, if $\measure$ is finite and $V=1$, if $\measure$ is infinite. % (Table~\ref{table:reference}). 
Notation-wise, define for $h>0$,
$c^h_{y}:=\measure(A_y^h)$ with $A_y^h:=[y-h/2,y+h/2)$ ($y\in
\Zh^{--}:=\Zh\cap (-\infty,0)$); $A_0^h:=[-h/2,0)$;
$$c^h_0:=\int_{A_0^h}y^2\ind(y)\measure(dy)\quad\text{ and }\quad\drift^h:=\sum_{y\in
\Zh^{--}}y\int_{A_y^h}\ind(z)\measure(dz).$$ 

We now specify the law of the approximating chain $X^h$ by insisting that (i) $X^h$ is a compound Poisson (CP)
process, with $X^h_0=0$, a.s., and whose positive jumps do not exceed $h$ -- hence admits a Laplace exponent $\psi^h(\beta):=\log\EE[e^{\beta X^h_1}]$ ($\beta\in \Crr$) --; and (ii) specifying $\psi^h$ under scheme 1, as:
\begin{equation}\label{eq:laplace_exponent:scheme1}
\psi^h(\beta)=(\drift-\drift^h) \frac{e^{\beta h}-e^{-\beta h}}{2h}+(\diffusion+c_0^h)\frac{e^{\beta  h}+e^{-\beta h}-2}{2h^2}+\sum_{y\in \Zh^{--}}c_y^h\left(e^{\beta y}-1\right),
\end{equation}
and under scheme 2, as:
\begin{equation}
\psi^h(\beta)=(\drift-\drift^h)\frac{e^{\beta h}-1}{h} + c_0^h\frac{e^{\beta h}+e^{-\beta h}-2}{2h^2} +\sum_{y\in \Zh^{--}}c_y^h\left(e^{\beta y}-1\right).
\label{eq:laplace_exponent:scheme2}
\end{equation}

This is consistent with the approximation of \cite{vidmarmijatovicsaul}: the above Laplace exponents follow from the forms of the characteristic exponents \cite[Eq.~(3.1)
and~(3.2)]{vidmarmijatovicsaul} via analytic
continuation, and to properly appreciate where the different terms appearing in \eqref{eq:laplace_exponent:scheme1}-\eqref{eq:laplace_exponent:scheme2} come from, we refer the reader to our paper \cite{vidmarmijatovicsaul}, especially Section~2.1 therein. 

\begin{table}[!hbt]
\caption{The usage of schemes 1 and 2 and of $V$ depends on the nature of $\diffusion$ and $\measure$.
%Note that the case $\measure(\mathbb{R})<\infty$, $\diffusion=0$ shall fall naturally under the scope of this investigation (by contrast to \cite{vidmarmijatovicsaul}).
 }\label{table:reference}
\begin{center}
\begin{tabular}{|c|c|c|}\hline
 L\'evy measure/diffusion part& $\diffusion>0$ & $\diffusion=0$  \\\hline
$\measure(\mathbb{R})<\infty$ & $V=0$, scheme 1 & $V=0$, scheme 2\\\hline
$\measure(\mathbb{R})=\infty$  & $V=1$, scheme 1 & $V=1$, scheme 2\\\hline
\end{tabular}
\end{center}
\end{table}

Indeed, note that, starting directly from \cite[Eq.~(3.2)]{vidmarmijatovicsaul}, the term $(\drift-\drift^h)\frac{e^{\beta h}-1}{h}$ in \eqref{eq:laplace_exponent:scheme2} should actually read as: $$(\mu-\mu^h)\left(\frac{e^{\beta h}-1}{h}\mathbbm{1}_{[0,\infty)}(\mu-\mu^h)+\frac{1-e^{-\beta h}}{h}\mathbbm{1}_{(-\infty,0]}(\mu-\mu^h)\right).$$ However, when $X$ is a spectrally negative L\'evy process with $\diffusion=0$, we have $\mu-\mu^h\geq 0$, at least for all sufficiently small $h$. Indeed, if $\int_{[-1,0)}\vert y\vert\measure(dy)<\infty$, then $\mu_0>0$ and by dominated convergence $\mu-\mu^h\to\mu_0$ as $h\downarrow 0$. On the other hand, if $\int_{[-1,0)}\vert y\vert\measure(dy)=\infty$, then we deduce by monotone convergence $-\mu^h\geq \frac{1}{2}\int_{ [-1,-h/2)}\vert y\vert\measure(dy)\to\infty$ as $h\downarrow 0$. We shall assume throughout that $h_\star$ is already chosen small enough, so that $\mu-\mu^h\geq 0$ holds for all $h\in (0,h_\star)$.

In summary, then, $h_\star$ is chosen so small as to guarantee that, for all $h\in (0,h_\star)$: (i) $\mu-\mu^h\geq 0$ and (ii) $\psi^h$ is the Laplace exponent of some CP process $X^h$, which is also a CTMC with state space $\Zh$ (note that in \cite[Proposition 3.9]{vidmarmijatovicsaul} it is shown $h_\star$ can indeed be so chosen, viz. point (ii)). Eq.~\eqref{eq:laplace_exponent:scheme1} and \eqref{eq:laplace_exponent:scheme2} then determine the \emph{weak} approximation $(X^h)_{h\in (0,h_\star)}$ precisely.
Finally, for $h\in (0,h_\star)$, let $\measure^h$ denote the L\'evy measure of
$X^h$. In particular, $\psi^h(\beta)=\int\left(e^{\beta
y}-1\right)\measure^h(dy)$, $\beta\in\Crr$, $h\in (0,h_\star)$, so that the jump intensities, equivalently the L\'evy measure, of $X^h$ can be read off directly from \eqref{eq:laplace_exponent:scheme1}-\eqref{eq:laplace_exponent:scheme2}. It will also be
convenient to define $\psi^0:=\psi$. 

\subsection{Connection with integro-differential equations}\label{subsection:connection_to_IDE}
An alternative form of \eqref{eq:LinRecursion} (as generalized to the case of arbitrary $q\geq 0$; see \eqref{equation:recursion:Wq_basic} of Proposition~\ref{proposition:calculating_scale_functions:basic}) is the analogue of the relation $(L-q)\Wq=0$ in the spectrally negative case, the latter holding true under sufficient regularity conditions on $\Wq$ (see e.g. \cite[Eq.~(12)]{biffis}). Here $L$ is the infinitesimal generator of $X$ \cite[p. 208, Theorem 31.5]{sato}: $$Lf(x)=\frac{\diffusion}{2}f''(x)+\mu f'(x)+\int_{(-\infty,0)}\left(f(x+y)-f(x)-yf'(x)\mathbbm{1}_{[-V,0)}(y)\right)\measure(dy)$$ ($f\in C^2_0(\mathbb{R})$, $x\in \mathbb{R}$). This suggests there might be a link between our probabilistic approximation and solutions to integro-differential equations. 

Indeed, one can check, for each $q\geq 0$, by taking Laplace transforms (using \eqref{eq:laplace_exponent}, the expression for the Laplace transform of $\Wq$ \cite[p. 100, Eq.~(4)]{kuznetsovkyprianourivero}, nondecreasingness of $\Wq$, and the theorem of Fubini), that the function $\Wq$ satisfies the following integro-differential equation (cf. \cite[Corollary IV.3.3]{asmussen} for the case $\diffusion=0$, $\measure(\mathbb{R})<+\infty$, and survival probabilities): 
\footnotesize
\begin{equation}\label{eq:integro-differential}
\frac{1}{2}\diffusion \frac{d\Wq}{dx}=1-\mu\Wq(x)+\int_0^\infty\left(\Wq(x-y)(\measure(-\infty,-y)+q)-\Wq(x)\measure[-V,-y)\mathbbm{1}_{(0,V]}(y)\right)dy
\end{equation}
\normalsize
(for the value of $\Wq(0)$ see \cite[p. 127, Lemma 3.1]{kuznetsovkyprianourivero}). Note that in the last integral of \eqref{eq:integro-differential}, the two terms appearing in its integrand cannot be separated when $\kappa(0)=+\infty$. 

Furthermore, \eqref{eq:LinRecursion} as generalized to arbitrary $q\geq 0$ (see Corollary~\ref{proposition:calculating_scale_functions}, Eq.~\ref{equation:recursion:Wq}) can be rewritten, when $\diffusion=0$, as (for $x\in \Zh^{++}$): 
\footnotesize
\begin{equation}\label{eq:integrodiff:scheme2}
\frac{1}{2}c_0^h\frac{\Wq_h(x)-\Wq_h(x-h)}{h}=1+(\mu^h-\mu)\Wq_h(x)+h\sum_{k=1}^{x/h}\Wq_h(x-kh)\left(\measure(-\infty,-(k-1/2)h)+q\right),
\end{equation}
\normalsize
with $\Wq_h(0)=\frac{1}{\frac{1}{2}c_0^h/h+\mu-\mu^h}$, and when $\diffusion>0$, as (again for $x\in \Zh^{++}$):
\footnotesize
\begin{eqnarray}
\nonumber \frac{1}{2}\left(\diffusion+c_0^h\right)\frac{\Wq_h(x)-\Wq_h(x-h)}{h}&=&1+(\mu^h-\mu)\frac{\Wq_h(x)+\Wq_h(x-h)}{2}\\
&&+h\sum_{k=1}^{x/h}\Wq_h(x-kh)\left(\measure(-\infty,-(k-1/2)h)+q\right),
\label{eq:integrodiff:scheme1}
\end{eqnarray}
\normalsize
with $\Wq_h(0)=\frac{2h}{\diffusion+c_0^h+(\mu-\mu^h)h}$. 

Thus \eqref{eq:integrodiff:scheme2} and \eqref{eq:integrodiff:scheme1} can be seen as (simple) approximation schemes for the integro-differential equation \eqref{eq:integro-differential}.\footnote{We are grateful to an anonymous referee for pointing out this connection, on an earlier draft of this paper.} However, from this viewpoint alone, it would be very difficult indeed to ``guess'' the correct discretization, which would also yield meaningful generalized scale functions of approximating chains --- the latter being our starting point and precisely the aspect of our schemes, which we wish to emphasize. Indeed, higher-order schemes for  \eqref{eq:integro-differential}, if and should they exist, would (likely) no longer be connected with L\'evy chains. 

\subsection{General notation}\label{subsection:further_general_notation}
With regard to miscellaneous notation, we let $\mathbb{R}_+$ (respectively $\mathbb{R}^+$) be the nonnegative (respectively strictly positive) real numbers;  $\mathbb{Z}_+=\mathbb{N}_0$ (respectively $\mathbb{Z}^+$, $\mathbb{Z}_-$, $\mathbb{Z}^-$) the nonnegative, (respectively strictly positive, nonpositive, strictly negative) integers; $\Zh^+$ (respectively $\Zh^{++}$, $\Zh^-$, $\Zh^{--}$) the nonnegative (respectively strictly positive, nonpositive, strictly negative) elements of $\Zh:=\{hk:k\in\ZZ\}$; $\mathbb{C}^{\leftarrow}:=\{z\in\mathbb{C}\!\!:\Re z<0\}$ (respectively $\mathbb{C}^{\rightarrow}:=\{z\in\mathbb{C}\!\!:\Re z>0\}$) with $\overline{\mathbb{C}^{\leftarrow}}$ (respectively $\overline{\mathbb{C}^{\rightarrow}}$) denoting the closure of this set (note that the arrow notation is suggestive of which halfplane is being considered). To a nondecreasing right-continuous function $F:\mathbb{R}\to\mathbb{R}$ a measure $dF$ may be associated in the Lebesgue-Stieltjes sense; stating that $F$ is of $M$-exponential order means $\vert F\vert(x)\leq C e^{M x}$ for all $x\geq 0$, for some $C<\infty$. On the other hand, a real-valued function $f$, defined on some $(a,\infty)$, $a\in \mathbb{R}$ is said to grow asymptotically $M$-exponentially, $M\in (0,\infty)$, if $\{\liminf_{x\to \infty}f(x)e^{-Mx},\limsup_{x\to\infty} f(x)e^{-Mx}\}\subset (0,\infty)$. Further, for functions $g\geq 0$ and $h>0$ defined on some right neighborhood of $0$, $g\sim h$ (resp. $g=O(h)$, $g=o(h)$) means $\lim_{0+}g/h\in (0,\infty)$ (resp. $\limsup_{0+}g/h<\infty$, $\lim_{0+}g/h=0$). Next, the Laplace transform of a measurable function $f:\mathbb{R}\to\mathbb{R}$ of $M$-exponential order, with $f\vert_{(-\infty,0)}$ constant (respectively measure $\mu$ on $\mathbb{R}$, concentrated on $[0,\infty)$) is denoted $\hat{f}$ (respectively $\hat{\mu}$): $\hat{f}(\beta)=\int_0^\infty e^{-\beta x}f(x)dx$ for $\Re\beta>M$ (respectively $\hat{\mu}(\beta)=\int_{[0,\infty)}e^{-\beta x}\mu(dx)$ for all $\beta\geq 0$ such that this integral is finite). For $\{x,y\}\subset\mathbb{R}$, $\lfloor x\rfloor:=\max\{k\in \ZZ:k\leq x\}$, $x\lor y=\max\{x,y\}$ and $x\land y=\min\{x,y\}$. A sequence $(h_n)_{n\in\mathbb{N}}$ of non-zero real numbers is said to be nested, if $h_n/h_{n+1}\in \mathbb{N}$ for all $n\in \mathbb{N}$. $\delta_x$ denotes the Dirac measure at $x\in \mathbb{R}$. Finally, increasing will mean strictly increasing; DCT stands for the Dominated Convergence Theorem; and we interpret $\pm a/0=\pm \infty$ for $a>0$. 

\section{Upwards skip-free L\'evy chains and their scale functions}\label{section:upwards_skip_free}
%That random walks, which are skip-free to the right (i.e. right-continuous), are the discrete-time, discrete-space analogue of spectrally negative processes is clear and also not new \cite[p. 99, Section 9.3]{doney}. 
In the sequel, we will require a fluctuation theory (and, in particular, a
theory of scale functions) for random walks, which are skip-free to the right,
once these have been embedded into continuous-time as CP processes (see next
definition and remark). Indeed, this theory has been developed in full detail
in \cite{vidmar:fluctuation_theory} and we recall here for the readers
convenience the pertinent results. 
%To begin with, then, define the relevant subclass of L\'evy processes, as follows (this notion was introduced in~\cite{vidmar}):

\begin{definition} 
\label{def:USF_Levy_chain}
A L\'evy process $Y$ with L\'evy measure $\measurenu$ is
said to be an \emph{upwards skip-free L\'evy chain}, if it is compound Poisson
and the L\'evy measure $\measurenu$ satisfies
$\supp(\measurenu)\subset
\Zh$ and $\supp(\measurenu\vert_{\mathcal{B}((0,\infty))})=\{h\}$ 
for some $h>0$.
\end{definition}

\begin{remark}\label{remark:ours_are_usf}
For all $h\in (0,h_\star)$, $X^h$ is an upwards skip-free L\'evy chain.
\end{remark}
For the remainder of this section we let $Y$ be an upwards skip-free L\'evy chain with L\'evy measure $\measurenu$, such that $\measurenu(\{h\})>0$ ($h>0$).

The following is either clear or else can be found in \cite[Subsection 3.1]{vidmar:fluctuation_theory}:
\begin{enumerate}[(i)]
\item One can introduce the Laplace exponent $\varphi:\Crr\to\mathbb{C}$,
given by $\varphi(\beta):=\int_{\mathbb{R}}(e^{\beta x}-1)\measurenu(dx)$
($\beta\in \Crr$), for which: $\EE[e^{\beta Y_t}]=\exp\{t\varphi(\beta)\}$
($\beta\in \Crr$, $t\geq 0)$. $\varphi$ is continuous in $\Crr$, analytic in
$\Cr$, $\lim_{+\infty}\varphi\vert_{[0,\infty)}=+\infty$ with
$\varphi\vert_{[0,\infty)}$ strictly convex.  
 \item Let
$\Phi(0)\in [0,\infty)$ be the largest root of $\varphi$ on $[0,\infty)$.
Then $\varphi\vert_{[\Phi(0),\infty)}:[\Phi(0),\infty)\to [0,\infty)$ is an
increasing bijection and we let
$\Phi:=(\varphi\vert_{[\Phi(0),\infty)})^{-1}:[0,\infty)\to [\Phi(0),\infty)$
be its inverse.
\end{enumerate}

We introduce in the next proposition two families of scale functions for $Y$,
which play analogous roles in the solution of exit problems, as they do in the
case of spectrally negative L\'evy processes, see \cite[Subsections
4.1-4-3]{vidmar:fluctuation_theory}: 

\begin{proposition}[Scale functions]\label{proposition:scale_functions_for_USF}
There exists a family of functions $W^{(q)}:\mathbb{R}\to [0,\infty)$ and
$$Z^{(q)}(x)=1+q\int_0^{\lfloor x/h\rfloor h}W^{(q)}(y)dy,\quad
x\in\mathbb{R}$$ defined for each $q\geq 0$ such that for any $q\geq 0$, we
have $\Wq(x)=0$ for $x<0$ and $\Wq$ is characterised on $[0,\infty)$ as the
unique right-continuous and piecewise constant function of exponential order
whose Laplace transform satisfies: $$\widehat{\Wq}(\beta)=\frac{e^{\beta
h}-1}{\beta h(\varphi(\beta)-q)}\text{ for }\beta>\Phi(q).$$
\end{proposition}

\begin{remark}\label{remark:scalefnsUSF}
%\leavevmode
\begin{enumerate}[(i)]
\item\label{remark:scalefnsUSF:one} The functions $\Wq$ are nondecreasing and the corresponding measures $d\Wq$ are supported in $\Zh$ for each $q\geq 0$. 
\item\label{remark:scalefnsUSF:three} The Laplace transform of the functions $\Zq$ is given by: $$\widehat{\Zq}(\beta)=\frac{1}{\beta}\left(1+\frac{q}{\varphi(\beta)-q}\right)\text{ for }\beta>\Phi(q),\, q\geq 0.$$
\item\label{remark:scalefnsUSF:four} For all $q\geq 0$: $\Wq(0)=1/(h\measurenu(\{h\}))$. 
\end{enumerate}
\end{remark}
Finally, the following proposition, whose corollary will gives rise to a method for
calculating the values of the scale functions associated to $Y$, follows from
the strong Markov property of $Y$ (see~\cite[Subsection~4.4]{vidmar:fluctuation_theory}). 

\begin{proposition}\label{proposition:calculating_scale_functions:basic}
Let $P$ be the transition matrix of the jump chain of the CTMC $Y$ and let $q\geq 0$, $\nu_q:=1+q/\measurenu(\mathbb{R})$. Assume (for ease of notation and without loss of generality) that $h=1$. Then, seen as vectors, $\Wq:=(\Wq(k))_{k\in\ZZ}$ and $\Zq:=(\Zq(k))_{k\in\ZZ}$ satisfy, entry-by-entry: $$(P\Wq)\vert_{\ZZ_+}=\nu_q\Wq\vert_{\ZZ_+}\text{ and }(P\Zq)\vert_{\ZZ_+}=\nu_q\Zq\vert_{\ZZ_+},$$ i.e.  with $\gamma:=\measurenu(\RR)$, $p:=\measurenu(\{1\})/\gamma$ and $q_n:=\measurenu(\{-n\})/\gamma$ ($n\in\mathbb{N}$), the recursive relations (for $n\in \mathbb{N}\cup \{0\}$): 
\begin{equation}\label{equation:recursion:Wq_basic}
p\Wq(n+1)=\left(1+\frac{q}{\gamma}\right)\Wq(n)-\sum_{k=1}^nq_k\Wq(n-k),
\end{equation}
 and 
\begin{equation}\label{equation:recursion:Zq:basic}
p\Zq(n+1)+\left(1-p-\!\!\sum_{k=1}^{n-1}q_k\right)=\left(1+\frac{q}{\gamma}\right)\Zq(n)-\!\!\sum_{k=1}^{n-1}q_k\Zq(n-k)
\end{equation}
hold true.  Additionally $\Wq\vert_{\mathbb{Z}^-}=0$ with $\Wq(0)=1/\measurenu(\{1\})$, whereas $\Zq\vert_{\mathbb{Z}_-}=1$.
\end{proposition}

\begin{corollary}\label{proposition:calculating_scale_functions}
Continue to assume $h=1$. We have for all $n\in \mathbb{N}\cup \{0\}$:
\begin{equation}\label{equation:recursion:Wq}
\Wq(n+1)=\Wq(0)+\sum_{k=1} ^{n+1}\Wq(n+1-k)\frac{q+\measurenu(-\infty,-k]}{\measurenu(\{1\})},\quad \Wq(0)=1/\measurenu(\{1\}),
\end{equation}
and for $\widetilde{\Zq}:=\Zq-1$,
\begin{equation}\label{equation:recursion:Zq}
\widetilde{\Zq}(n+1)=(n+1)\frac{q}{\measurenu\{1\}}+\sum_{k=1}^n\widetilde{\Zq}(n+1-k)\frac{q+\measurenu(-\infty,-k]}{\measurenu(\{1\})},\quad \widetilde{\Zq}(0)=0.
\end{equation}
\end{corollary}

\begin{remark}\label{remark:numerics:three} 
Whilst we have based our algorithm on Eq.~\eqref{equation:recursion:Wq} and~\eqref{equation:recursion:Zq}, nevertheless Eq.~\eqref{equation:recursion:Wq_basic} and \eqref{equation:recursion:Zq:basic} should not be discounted entirely. For example, they allow %together with \eqref{eq:laplace_exponent:scheme1}-\eqref{eq:laplace_exponent:scheme2}, 
to make the following observation. 

Recall $\sc_h$ is given by~\eqref{equation:recursion:Wq}, as applied to the process $Y=X^h/h$, equivalently then it may be obtained from \eqref{equation:recursion:Wq_basic}, again as applied to the process $Y=X^h/h$. Assume next (by scaling, without loss of generality) $V\leq x\in \Zh$. 

Suppose the L\'evy measure $\measure$ of $X$ is modified below (possibly including) the level $-x$ in such a way that $\measure(-\infty,-x]$ is preserved, whilst $ \mathbbm{1}_{(-x,0)}\cdot\measure$, $\diffusion$ and $\mu$ are kept fixed. Then \eqref{equation:recursion:Wq_basic}, as applied to $X^h/h$ for computing $W_h$ on the interval $[0,x]$, remains unaffected. This is because this recursion up to level $x$ depends solely on the probabilities of the jump-sizes with modulus at most $x-h$  and on the total L\'evy mass of the approximating L\'evy chain. The latter, however, do not change by said transformation of the L\'evy measure (see \eqref{eq:laplace_exponent:scheme1}-\eqref{eq:laplace_exponent:scheme2}). Consequently, since, as we shall see, $\sc_h$ converges to $\sc$, as $h\downarrow 0$, this also means that  the scale function $\sc$ itself on the interval $[0,x]$ is invariant under such a transformation. Analogously for $\Wq$ and $\Zq$. 

%In summary, under the described perturbation of the L\'evy measure of $X$, the values of the scale functions $\Wq$ and $\Zq$ on the interval $[0,x]$ remain unchanged. 
\end{remark}

\section{Convergence of scale functions}\label{section:Convergence_of_scale_functions}
First we fix some notation. Pursuant to \cite[Subsections 8.1 \& 8.2]{kyprianou} (respectively Section~\ref{section:upwards_skip_free}) we associate henceforth with $X$ (respectively $X^h$) two families of scale functions $(\Wq)_{q\geq 0}$ and $(\Zq)_{q\geq 0}$ (respectively  $(\Wq_h)_{q\geq 0}$ and $(\Zq_h)_{q\geq 0}$, $h\in (0,h_\star)$). Note that these functions are defined on the whole of $\mathbb{R}$, are nondecreasing, c\`adl\`ag, with $\Wq(x)=\Wq_h(x)=0$ and $\Zq(x)=\Zq_h(x)=1$ for $x\in (-\infty,0$). We also let $\Phi(0)$ (respectively $\Phi^h(0)$) be the largest root of $\psi\vert_{[0,\infty)}$ (respectively $\psi^h\vert_{[0,\infty)}$) and denote by $\Phi$ (respectively $\Phi^h$) the inverse of $\psi\vert_{[\Phi(0),\infty)}$ (respectively $\psi^h\vert_{[\Phi^h(0),\infty)}$, $h\in (0,h_\star)$). As usual $\sc$ (resp. $\sc_h$) denotes $\sc^{(0)}$ (resp. $\sc^{(0)}_h$, $h\in (0,h_\star)$). 

Next, for $q\geq 0$, recall the Laplace transforms of the functions $\Wq$ and $\Zq$ \cite[p. 214, Theorem~8.1]{kyprianou} (for $\beta>\Phi(q)$): $\int_0^\infty e^{-\beta x}\Wq(x)dx=1/(\psi(\beta)-q)$ and $\int_0^\infty e^{-\beta x}\Zq(x)dx=\frac{1}{\beta}\left(1+\frac{q}{\psi(\beta)-q}\right)$ (where the latter formula follows using e.g. integration by parts). The Laplace transforms of $\Wq_h$ and $\Zq_h$, $h\in (0,h_\star)$, follow from Proposition~\ref{proposition:scale_functions_for_USF} and Remark~\ref{remark:scalefnsUSF}~\ref{remark:scalefnsUSF:three}. 

\begin{proposition}[Pointwise convergence]\label{remark:pointwise_convergence}\label{proposition:scales:pointwise_convergence}
Suppose $\psi^h\to \psi$ and $\Phi^h\to\Phi$ pointwise as $h\downarrow 0$. Then, for each $q\geq 0$, $\Wq_{h}\to \Wq$ and $\Zq_{h}\to \Zq$ pointwise, as $h\downarrow 0$. 
\end{proposition}
\begin{remark}\label{remark:convergencePhi}
We will see in \ref{delta1} of Subsection~\ref{subsection:the_difference} that, in fact, $\psi^h\to \psi$ locally uniformly in $[0,\infty)$ as $h\downarrow 0$, which implies that $\Phi^h\to\Phi$ pointwise as $h\downarrow 0$. In particular, given any $q\geq 0$, $\gamma>\Phi(q)$ implies $\gamma>\Phi^h(q)$ for all $h\in (0,h_0)$, for some $h_0>0$.
\end{remark}
\begin{proof}
Since $\Phi^h(q)\to \Phi(q)$ as $h\downarrow 0$, it follows via integration by parts ($\int_{[0,\infty)}e^{-\beta x}dF(x)=\beta\int_{(0,\infty)}e^{-\beta x}F(x)dx$ for any $\beta\geq 0$ and any nondecreasing right-continuous $F:\mathbb{R}\to\mathbb{R}$ vanishing on $(-\infty,0)$  \cite[Chapter 0, (4.5) Proposition]{revuzyor}) that, for some $h_0>0$, the Laplace transforms of $d\Wq$, $d\Zq$, $(d\Wq_h)_{h\in (0,h_0)}$ and $(d\Zq_h)_{h\in (0,h_0)}$, are defined (i.e. finite) on a common halfline. These measures are furthermore concentrated on $[0,\infty)$  and since $\psi^h\to\psi$ pointwise as $h\downarrow 0$, then $\widehat{d\Wq_h}\to \widehat{d\Wq}$ and $\widehat{d\Zq_h}\to \widehat{d\Zq}$ pointwise as $h\downarrow 0$. By \cite[p. 110, Theorem 8.5]{bhattacharya}, it follows that $d\Wq_{h_n}\to d\Wq$ and $d\Zq_{h_n}\to d\Zq$ vaguely as $n\to\infty$, for any sequence $(h_n)_{n\geq 1}\downarrow 0$. This implies that, as $h\downarrow 0$, $\Wq_{h}\to \Wq$ (respectively $\Zq_{h}\to \Zq$) pointwise at all points of continuity of $\Wq$ (respectively $\Zq$). Now, the functions $\Zq$ are continuous everywhere, whereas $\Wq$ is continuous on $\mathbb{R}\backslash \{0\}$ and has a jump at $0$, if and only if $X$ has sample paths of finite variation \cite[p. 222, Lemma 8.6]{kyprianou}. In the latter case, however, we necessarily have $\diffusion=0$ and $\int_{[-1,0)}\vert y\vert\measure(dy)<\infty$ \cite[p. 140, Theorem 21.9]{sato} and the jump size is $\Wq(0)=1/\mu_0$ (see Section~\ref{section:setting_and_notation} for definition of $\mu_0$). By Remark~\ref{remark:scalefnsUSF}~\ref{remark:scalefnsUSF:four} and \eqref{eq:laplace_exponent:scheme2}, $\Wq_h(0)=1/(h\measure^h(\{h\}))=1/(\mu-\mu^h+c_0^h/h)$. The latter quotient, however, converges to $1/\mu_0$, as $h\downarrow 0$, by the DCT  (since $\int_{[-1,0)}\vert y\vert\measure(dy)<\infty$ in this case, and  $c_0^h\leq (h/2)\int_{[-V\land (h/2),0)}\vert y\vert\lambda(dy)$).
\end{proof}

\section{Rates of convergence}\label{section:convergence_rates}
In this section we establish our main result, Theorem~\ref{theorem:rates_for_scale_functions} from the Introduction. Subsection~\ref{subsection:method_preliminary_etc} describes the general method of proof and establishes some preliminary observations and notation. Subsection~\ref{subsection:auxiliary_technical_results}
%{subsection:some_technical_estimates_and_bounds} and~\ref{subsection:some_asymptotic_properties_at_zero} 
contains technical results, notationally and otherwise independent from rest of the text, which are applied time and again in the sequel. Then Subsections~\ref{subsection:the_difference}--\ref{subsection:sclaes:diffusion0infvariation} establish a series of convergence results, which together imply Theorem~\ref{theorem:rates_for_scale_functions}. Finally, Subsection~\ref{subsection:functionals_of_scale_fncs_conv} contains a convergence result for the derivatives $\Wqprime$. 

\subsection{Method of proof, preliminary observations and notation}\label{subsection:method_preliminary_etc}
%We let $X$ be either a spectrally negative L\'evy process or an upwards
%skip-free (USF) L\'evy chain (if the latter, then supported by $\Zh$) with
%characteristic triplet $(\diffusion,\measure,\drift)$ relative to some cut-off
%function $\tilde{c}$. There corresponds to it the Laplace exponent
%$\psi:\Crr\to\mathbb{R}$ (continuous in $\Crr$, analytic in $\Cr$); $\Phi(0)$,
%it being the largest root of $\psi$; the inverse $\Phi$ of
%$\psi\vert_{[\Phi(0),\infty)}$ and a collection of nondecreasing scale
%functions $\Wq$ and $\Zq$, given for each $q\geq 0$. 
The key step in the proof of Theorem~\ref{theorem:rates_for_scale_functions}
consists of a detailed analysis of the relevant differences arising in the
integral representations (see
Paragraph~\ref{subsubsection:integral_representations}) of the scale functions,
see Paragraph~\ref{subsubsection:precise_statement}.
A more detailed explanation of the method of proof
will be given in 
Paragraph~\ref{subsubsection:method_for_obtaining_the_rates_of_convergence}.

\begin{remark}
With reference to Subsection~\ref{subsection:connection_to_IDE}, there is of course extensive literature on numerical solutions to integro-differential equations (IDE) of the relevant (Volterra) type (viz. Eq.~\eqref{eq:integro-differential}). This literature will, however, typically assume at least the continuity of the kernel appearing in the integral of the IDE, to even pose the problem, and obtain rates of convergence under additional smoothness conditions thereon (and the solution to the IDE) \cite[Chapters~2 and~3]{brunner} \cite[Chapters~7 and~11]{linz}. In our case the kernel appearing in \eqref{eq:integro-differential} is of course not (necessarily) even continuous (let alone possessing higher degrees of smoothness). Further, discounting for a moment the continuity requirement on the kernel, which may appear technical, some relevant \emph{general} results on convergence do exist, e.g. \cite[p. 102, Theorem 7.2]{linz} for the case $\diffusion=0$ \& $\measure(\mathbb{R})<+\infty$, but are not really (directly) applicable, since one would need to \emph{a priori} establish (at least) a rate of convergence for the difference between the integral appearing in \eqref{eq:integro-differential} and its discretization (local consistency error; see \cite[p. 101, Eq.~(7.12)]{linz}). This does not appear possible in general without a knowledge of the (sufficient) smoothness properties of the target function $\Wq$ (the latter not always being clear; see \cite{kyprianou:smoothness}) and indeed, it would seem, those of the tail $(y \mapsto \measure(-\infty,-y))$. Such an error analysis would be further complicated when $\diffusion>0$ (respectively $\kappa(0)=\infty$), since then we are dealing with the discretization also of the derivative of $\Wq$ (respectively the integral in \eqref{eq:integro-differential} cannot be split up as the difference of the integrals of each individual term of the integrand). It is then not very likely that looking at this problem from the integro-differential perspective alone would allow us to obtain, moreover sharp, rates of convergence (at least not in general). 

By contrast, the method for obtaining the sharp rates of convergence that we shall use, based on the integral representations of the scale functions and their approximations, will allow us to handle all the cases within a single framework.
\end{remark}

\subsubsection{Integral representations of scale
functions}\label{subsubsection:integral_representations}

\begin{proposition}\label{proposition:well-posedness}
Let $q\geq 0$. For all $\beta\in\mathbb{C}$ with $\Re\beta>\Phi(q)$, $\psi(\beta)-q\ne 0$ (respectively $\psi^h(\beta)-q\ne 0$) and one has $(\psi(\beta)-q)\widehat{\Wq}(\beta)=1$ (respectively $\beta h(\psi^h(\beta)-q)\widehat{\Wq_h}(\beta)=e^{\beta h}-1$), and  $\beta \widehat{\Zq}(\beta)=1+\frac{q}{\psi(\beta)-q}$ (respectively $\beta \widehat{\Zq_h}(\beta)=1+\frac{q}{\psi^h(\beta)-q}$) for the scale functions of $X$ (respectively $X^h$, $h\in (0,h_\star)$). 
\end{proposition}
\begin{proof}
The stipulated equalities extend from $\beta>\Phi(q)$ real, to complex $\beta$ with  $\Re\beta>\Phi(q)$, via analytic continuation, using expressions for the Laplace transforms of the scale functions (the latter having been noted in Section~\ref{section:Convergence_of_scale_functions}). In particular, so extended, they then imply $\psi(\beta)-q\ne 0$ for the range of $\beta$ as given.
\end{proof}

\begin{corollary}[Integral representation of scale functions]\label{corollary:scale_fncs_laplace-transforms}
Let $q\geq 0$. For any $\gamma>\Phi(q)$, we have, for all $x>0$ (with $\beta:=\gamma+is$): 
\begin{equation}\label{eq:integral_represantation_Wq}
\Wq(x)=\frac{1}{2\pi}\lim_{T\to\infty}\int_{-T}^T\frac{e^{\beta x}}{\psi(\beta)-q}ds
\end{equation}
and 
\begin{equation}
\Zq(x)=\frac{1}{2\pi}\lim_{T\to\infty}\int_{-T}^T\frac{e^{\beta x}}{\beta}\left(1+\frac{q}{\psi(\beta)-q}\right)ds.
\end{equation}
Likewise, for any $h\in (0,h_\star)$ and then any $\gamma>\Phi^h(q)$, we have, for all $x\in \Zh^+$ (again with $\beta:=\gamma+is$):
\begin{equation}\label{eq:integral_represantation_Wqh}
\Wq_h(x)=\frac{1}{2\pi}\int_{-\pi/h}^{\pi/h}\frac{e^{\beta (x+h)}}{\psi^h(\beta)-q}ds
\end{equation}
and 
\begin{equation}
\Zq_h(x)=\frac{1}{2\pi}\int_{-\pi/h}^{\pi/h}\frac{e^{\beta x}}{\beta}\frac{\beta h}{1-e^{-\beta h}}\left(1+\frac{q}{\psi^h(\beta)-q}\right)ds.
\end{equation}
\end{corollary}
\begin{proof}
First note that  $\Wq$ and $\Zq$ (respectively $\Wq_h$ and $\Zq_h$) are of $\gamma$-exponential order for all $\gamma>\Phi(q)$ (respectively $\gamma>\Phi^h(q)$, $h\in (0,h_\star)$). Then use the inverse Laplace \cite[Section 3.3]{davies} (respectively $Z$ \cite[p. 11]{jury}) transform.
\end{proof}

\subsubsection{The differences $\Delta_W^{(q)}$ and $\Delta_Z^{(q)}$}\label{subsubsection:precise_statement}
For $x\geq 0$, let $T_x$ (resp. $T^h_x$) denote the first entrance time of $X$ (resp. $X^h$) to $[x,\infty)$, let $\underline{X}_t:=\inf \{X_s:s\in [0,t]\}$  (resp. $\underline{X^h}_t:=\inf \{X_s^h:s\in [0,t]\}$), $t\geq 0$, be the running infimum process and $\underline{X}_{\infty}:=\inf\{X_s:s\in [0,\infty)\}$ (resp. $\underline{X^h}_\infty:=\inf\{X_s^h:s\in [0,\infty)\}$) the overall infimum of $X$ (resp. $X^h$, $h\in (0,h_\star)$). 

In the case of the spectrally negative process $X$, it follows from \cite[Theorem 8.1~(iii)]{kyprianou}, regularity of $0$ for $(0,\infty)$ \cite[p. 212]{kyprianou}, dominated convergence and continuity of $\Wq\vert_{(0,\infty)}$ that, for $q\geq 0$ and $\{x,y\}\subset \mathbb{R}^+$: 
\begin{equation}\label{eq:comparison_issue_X}
\EE[e^{-q T_y}\mathbbm{1}(\underline{X}_{T_y}\geq -x)]=\frac{\Wq(x)}{\Wq(x+y)}\text{ and }\EE[e^{-q T_y}\mathbbm{1}(\underline{X}_{T_y}>-x)]=\frac{\Wq(x)}{\Wq(x+y)}.
\end{equation}
On the other hand, we find \cite[Theorem 4.6]{vidmar:fluctuation_theory} the direct analogues of these two formulae in the case of the approximating processes $X^h$, $h\in (0,h_\star)$ as being ($q\geq 0$, $\{x,y\}\subset \Zh^{++}$):
\begin{equation}\label{eq:comparison_issue_Xh}
\EE[e^{-q T^h_y}\mathbbm{1}(\underline{X^h}_{T_y^h}\geq -x)]=\frac{\Wq_h(x)}{\Wq_h(x+y)}\text{ and }\EE[e^{-q T_y^h}\mathbbm{1}(\underline{X}_{T_y^h}>-x)]=\frac{\Wq_h(x-h)}{\Wq_h(x-h+y)}.
\end{equation}

We conclude by comparing \eqref{eq:comparison_issue_X} with \eqref{eq:comparison_issue_Xh} that there is \emph{no a priori probabilistic reason} to favour either $\Wq_h$ or $\Wq_h(\cdot-h)$ in the choice of which of these two quantities to compare to $\Wq$. Nevertheless, this choice is not completely arbitrary:

\noindent (a) In view of \eqref{eq:integral_represantation_Wq} and \eqref{eq:integral_represantation_Wqh}, the quantity $\Wq_h(\cdot-h)$ seems more favourable (cf. also the findings of Proposition~\ref{proposition:convergence:BM+drift:Wq}, especially when $q=\vert \mu\vert=0$). In addition, when $X$ has sample paths of infinite variation, a.s., $\Wq(0)$ is equal to zero \cite[p. 33, Lemma 3.1]{kuznetsovkyprianourivero} and so is $\Wq_h(-h)$, whereas $\Wq_h(0)$ is always strictly positive ($h\in (0,h_\star)$). 

\noindent (b) On the other hand, when $X$ has sample paths of finite variation, a.s., then $\Wq(0)=1/\mu_0>0$ \cite[p. 33, Lemma 3.1]{kuznetsovkyprianourivero} and if in addition the L\'evy measure is finite, then in fact also $\Wq_h(0)=1/\mu_0$ for all $h\in (0,h_\star)$. 

\begin{remark}\label{remark:basic}
% \leavevmode
\begin{enumerate}[(i)]
\item\label{remark:basic:i} It follows from the above discussion that it is reasonable to approximate $\Wq$ by $\Wq_h(\cdot-h)$ (resp. $\Wq_h$), when $X$ has sample paths of infinite (resp. finite) variation (a.s.). Indeed, in the Brownian motion with drift case, approximating $\Wq$ by $\Wq_h$, or even the average $(\Wq_h+\Wq_h(\cdot-h))/2$, rather than by $\Wq_h(\cdot-h)$, would lower the order of convergence from quadratic to linear (see Proposition~\ref{proposition:convergence:BM+drift:Wq}).  
\item When $q=0$ or $x=0$, $\Zq(x)=\Zq_h(x)=1$, $h\in (0,h_\star)$. Thus, when comparing these functions, we shall always assume $q\land x>0$, for the only interesting case.
\end{enumerate}
\end{remark}

In view of Remark~\ref{remark:basic}~\ref{remark:basic:i} we define $\delta^0$ to be equal to $0$ or $1$ according as the sample paths of $X$ are of finite or infinite variation (a.s.). Fix $q\geq 0$. For $h\in (0,h_\star)$ we then define the differences:
\begin{equation}\label{eq:diffW}
\Delta_{W}^{(q)}(x,h):=\Wq(x)-\Wq_h(x- \delta^0h),\quad x\in \Zh^{++}\cup \{\delta_0h\}
\end{equation}
and 
\begin{equation}\label{eq:diffZ}
\Delta_Z^{(q)}(x,h):=\Zq(x)-\Zq_h(x),\quad x\in\Zh^{++}.
\end{equation}
Fix further any $\gamma>\Phi(q)$. Let $h\in (0,h_\star)$ be such that also $\gamma>\Phi^h(q)$, Then Corollary~\ref{corollary:scale_fncs_laplace-transforms} implies,  for any $x\in \Zh^{++}\cup \{\delta_0h\}$ (we always let, here and in the sequel, $\beta:=\gamma+is$ to shorten notation): 
\footnotesize
\begin{eqnarray}\label{eq:difference:Wqs} \nonumber
 e^{-\gamma x}2\pi \diffW&=&\underbrace{\lim_{T\to\infty}\int_{(-T,T)\backslash (-\pi/h,\pi/h)}e^{isx}\frac{ds}{\psi(\beta)-q}}_{\mytag{(a)}{W:a}}+\underbrace{\int_{[-\pi/h,\pi/h]}e^{isx}\left[\frac{\psi^h-\psi}{(\psi-q)(\psi^h-q)}\right](\beta)ds}_{\mytag{(b)}{W:b}}+\\
&&\underbrace{(1-\delta^0)\int_{[-\pi/h,\pi/h]}e^{isx}\left(1-e^{ish}\right)\frac{ds}{\psi^h(\beta)-q}}_{\mytag{(c)}{W:c}}
\end{eqnarray}
\normalsize
whereas for $x\in\Zh^{++}$: 
\footnotesize
\begin{eqnarray}\label{eq:difference:Zqs}
\nonumber e^{-\gamma x}2\pi\diffZ&=&\underbrace{\lim_{T\to\infty}\!\int_{(-T,T)\backslash (-\pi/h,\pi/h)}\frac{e^{isx}}{\beta}\left(\frac{q}{\psi(\beta)-q}\right)ds}_{\mytag{(a)}{Z:a}}\\\nonumber
&+&\underbrace{\int_{[-\pi/h,\pi/h]}\frac{e^{isx}}{\beta}\left(1-\frac{\beta h}{1-e^{-\beta h}}\right)\left(\frac{q}{\psi^h(\beta)-q}\right)ds}_{\mytag{(b)}{Z:b}}\\
&+&\underbrace{q\int_{[-\pi/h,\pi/h]}\frac{e^{isx}}{\beta}\left[\frac{\psi^h-\psi}{(\psi^h-q)(\psi-q)}\right](\beta)ds}_{\mytag{(c)}{Z:c}}.
\end{eqnarray}
\normalsize
Note that in \eqref{eq:difference:Zqs} we have taken into account that the difference between the inverse Laplace and inverse $Z$ transform, for the function, which is identically equal to $1$, vanishes identically. 

\begin{remark}\label{remark:approx_by_mean}
Notice that if we were to approximate $W^{(q)}(x)$ by the average $\frac{1}{2}(W^{(q)}_h(x)+W^{(q)}_h(x-h))$ instead of $W^{(q)}_h(x-\delta^0h)$, and were to adapt accordingly the definition in \eqref{eq:diffW}, the resulting change to  \eqref{eq:difference:Wqs} would be, that  term \ref{W:c}
 would always be present, with $(1-\delta^0)$ replaced by $1/2$. 

Now, when the sample paths are of finite variation (hence when we have $\delta^0=0$), none of the arguments would change, and the same theoretical rates of convergence would obtain.  Indeed as we will see in the proof of Proposition~\ref{proposition:Wqconvergence:trivial_diffusion_FV}, in this case \eqref{eq:difference:Wqs}\ref{W:c} admits an estimate that yields a linear order of convergence ($O(h)/x$), and one gets linear order terms such as this from terms \ref{W:a} and  \ref{W:b}, also. 

However, when the sample paths are of infinite variation (hence when we have $\delta^0=1$), then in the estimate of the difference $\Delta_W^{(q)}$ we would have to \emph{add} to the error also an estimate of \eqref{eq:difference:Wqs}\ref{W:c} (with $1-\delta^0$ replaced by $1/2$ therein), which in general would then worsen the theoretical order of convergence. (Cf. also Remark~\ref{remark:basic}~\ref{remark:basic:i}.)
\end{remark}

\subsubsection{Method for obtaining the rates of convergence in \eqref{eq:difference:Wqs} and \eqref{eq:difference:Zqs}}\label{subsubsection:method_for_obtaining_the_rates_of_convergence}
Apart from the Brownian motion with drift case, which is treated explicitly, the method for obtaining the rates of convergence for the differences \eqref{eq:difference:Wqs} and \eqref{eq:difference:Zqs} is as follows (recall $\beta=\gamma+is$):
\begin{enumerate}
\item \label{method:scales:one}First we estimate $\vert \psi^h-\psi\vert (\beta)$ to control the numerators. In particular, we are able to conclude $\psi^h\to\psi$, uniformly in bounded subsets of $\Crr$. See Subsection~\ref{subsection:the_difference}.
\item \label{method:scales:two} Then we show $\vert\psi ^h-q\vert(\beta)$ is suitably bounded from below on $s\in (-\pi/h,\pi/h)$, uniformly in $h\in [0, h_0)$, for some $h_0>0$. This property, referred to as \emph{coercivity}, controls the denominators. See Subsection~\ref{subsection:coercivity}.
 \item \label{method:scales:three} Finally, using \eqref{method:scales:one} and \eqref{method:scales:two}, one can estimate the integrals appearing in \eqref{eq:difference:Wqs} and \eqref{eq:difference:Zqs} either by a direct $\vert\int\cdot ds\vert\leq \int\vert\cdot\vert ds$ argument, or else by first applying a combination of integrations by parts (see \eqref{eq:general_scheme_per_partes} below) and Fubini's Theorem. In the latter case, the estimates of $\vert \frac{d(\psi(\beta)-\psi^h(\beta))}{ds}\vert$ and the growth in $s$, as $\vert s\vert\to\infty$, of $\frac{d(\psi^h-q)(\beta)}{ds}$, $h\in [0,h_\star)$, also become relevant, and we provide these in Subsection~\ref{subsection:_the_difference_of_derivatives}. 
\end{enumerate}

\begin{remark}
%\leavevmode

\noindent (i)  Note that the integral representation of the scale functions is crucial for our programme to yield results. The formulae \eqref{eq:difference:Wqs} and \eqref{eq:difference:Zqs} suffice to give a precise rate 
%although they \emph{fail} to give a good error bound in terms of the argument of the scale functions (viz. the exponential factor $e^{\gamma x}$).
%this for the precise reasons that a direct Laplace inversion is not a good way to calculate the scale functions numerically: ``The problem here is the exponential factor $e^{\gamma x}$, which can be very large and would amplify the errors present in the numerical evaluation of the integral"\cite[p. 66]{kuznetsovkyprianourivero}. 
%From the point of view of a convergence rate, this is of course inessential. In a similar vein, we remark that convergence is established as 
locally uniformly in $x\in (0,\infty)$. 

\noindent (ii) The integration by parts in \eqref{method:scales:three} is applied according to the general scheme ($f$ differentiable, $x>0$): 
\begin{equation}\label{eq:general_scheme_per_partes}
\frac{1}{x}\frac{d}{ds}\left( e^{isx}f(s)\right)=ie^{isx}f(s)+\frac{1}{x}e^{isx}f'(s),
\end{equation}
and with the integral $\int e^{isx}f(s) ds$ (over a relevant domain) in mind. Then, upon integration against $ds$, the left-hand side and the second term on the right-hand side of \eqref{eq:general_scheme_per_partes} admit for an estimate, which could not be made for $\int e^{isx}f(s) ds$ directly, but in turn a factor of $1/x$ emerges, implying (as we will see) that the final bound is locally uniform in $(0,\infty)$ (in the estimates there is always also present a factor of $e^{\gamma x}$ which from the perspective of the relative error, and in view of the growth properties of $\Wq$ and $\Zq$ at $+\infty$ \cite[p.  129, Lemma~3.3]{kuznetsovkyprianourivero}, is perhaps not so bad). In fact, the convergence rate obtained via \eqref{method:scales:one}-\eqref{method:scales:three} is uniform in bounded subsets of $(0,\infty)$, if \eqref{method:scales:three} does not involve integration by parts. 

%uniform in bounded subsets of $(0,\infty)$, according as integration by parts is or is not used in the derivations. Indeed, 
\noindent (iii) Now, it is usually the case that the estimates from \eqref{method:scales:three} may be made by a direct application of the integral triangle inequality. We were not able to avoid integration by parts, however, in the case of the convergence for the functions $\Wq$, $q\geq 0$, and even then only if $\diffusion=0$ (see Subsections~\ref{subsection:scales:FV} and~\ref{subsection:scales:diffusion=0:infinite_variation}). Particularly delicate is the case when furthermore the sample paths of $X$ are of finite variation (a.s.). In the latter case a key lemma is Lemma~\ref{lemma:FV}, which itself depends crucially on the findings of Proposition~\ref{appendix:proposition:integrability:main}. 

\noindent (iv) Even when $\diffusion=0$, however, numerical experiments (see Section~\ref{section:numerical_illustrations} and Appendix~\ref{appendix:further_examples}) seem to suggest that, at least in some further subcases, one should be able to establish convergence for the functions $\Wq$, $q\geq 0$, which is uniform in bounded (rather than just compact) subsets of $(0,\infty)$. This remains open for future research. 
\end{remark}

%\begin{remark}
%While we are thus not always able to avoid using integration by parts in order to obtain a rate of convergence, numerical experiments (see Section~\ref{section:numerical_illustrations}) seem to suggest that, at least in some further subcases, one should be able to establish convergence which is uniform in bounded (rather than just compact) subsets of $(0,\infty)$, and this remains open to future research. 
%\end{remark}

%Sharpness of the rates is obtained by \emph{ad hoc} methods, which share some common points (in particular, so-called \emph{reduction by domination}, given explicitly below), but are eventually disparate. 
\begin{remark}
Sharpness of the rates is obtained by constructing specific examples of L\'evy processes, for which convergence is no better than stipulated (cf. the statement of Theorem~\ref{theorem:rates_for_scale_functions}). The key observation here is the following principle of \emph{reduction by domination}: 
\begin{quote}
Suppose we seek to prove that $f\geq 0$ converges to $0$ no faster than $g>0$, i.e. that $\limsup_{h\downarrow 0}f(h)/g(h)\geq C>0$ for some $C$. If one can show $f(h)\geq A(h)-B(h)$ and $B=o(g)$, then to show $\limsup_{h\downarrow 0}f(h)/g(h)\geq C$, it is sufficient to establish $\limsup_{h\downarrow 0}A(h)/g(h)\geq C$. 
\end{quote}
(This principle was also applied in \cite{vidmarmijatovicsaul} to establish sharpness of the stated rates of convergence there.)
\end{remark}
We will use the basic, but very useful, principle of reduction by domination without explicit reference in the sequel. 

\subsubsection{Further notation}
Notation-wise, we let  (where $\delta\in [0,1]$):
%(and, in part, recall) 
$$\xi(\delta):=\int_{[-\delta,0)}u^2\measure(du),\ \kappa(\delta):=\int_{[-1,-\delta)}\vert y\vert\measure(dy),\ \zeta(\delta):=\delta\kappa(\delta)\text{ and }\gamma(\delta):=\delta^2\measure([-1,-\delta))$$
and remark that, by the findings of \cite[Lemma 3.8]{vidmarmijatovicsaul}, $\gamma(\delta)+\zeta(\delta)+\xi(\delta)\to 0$ as $\delta\downarrow 0$.

Finally, note that, unless otherwise indicated, we consider henceforth as having fixed:
\begin{equation*}
X,\quad (X^h)_{h\in (0,h_\star)}, \quad q\geq 0\quad \text{and a}\quad \gamma>\Phi(q).
\end{equation*} 
We insist that the dependence on $x$ of the error estimates will be kept explicit throughout, whereas the dependence on the L\'evy triplet, $q$  and $\gamma$ will be subsumed in the capital (or small) $O$ ($o$) notation. In particular, the notation $f(x,h)=g(x,h)+l(x)O(h)$, $x\in A$, means that $l(x)>0$ for $x\in A$ and: $$\sup_{x\in A}\vert (f(x,h)-g(x,h))/l(x)\vert=O(h)$$ and analogously when $O(h)$ is replaced by $o(h)$ etc. Further, we shall sometimes resort to the notation: $$A(s):=A^0(s):=\psi(\gamma+is)-q\ (\text{for }s\in \mathbb{R})\text{ and }A^h(s):=\psi^h(\gamma+is)-q\ (\text{for }s\in [-\pi/h,\pi/h]),$$ $h\in (0,h_\star)$, where reference to $q$ and $\gamma$ has been suppressed. We stress that in Subsections~\ref{subsection:the_difference}-\ref{subsection:functionals_of_scale_fncs_conv} we shall have throughout: 
\begin{equation*}
\beta:=\gamma+is.
\end{equation*}
\noindent The remainder of our analysis in this section will proceed as follows. First we list in Subsection~\ref{subsection:auxiliary_technical_results}, for the readers convenience, a number of auxiliary technical results.
%, see Subsections~\ref{subsection:some_technical_estimates_and_bounds} and~\ref{subsection:some_asymptotic_properties_at_zero}. 
Their proofs, which are independent of the analysis in Section~\ref{section:convergence_rates}, are relegated to Appendices~\ref{appendix:technical_lemmas} and~\ref{appendix:some_asymptotic_properties_of_measures_on_R}. Then Subsection~\ref{subsection:the_difference} estimates the absolute difference $\vert \psi^h-\psi\vert$, Subsection~\ref{subsection:_the_difference_of_derivatives} analyzes the derivatives $A^{h\prime}$ and the difference $\vert A^{h\prime}-A'\vert$, while in Subsection~\ref{subsection:coercivity} we prove suitable coercivity of $\vert \psi^h-q\vert$. Subsections~\ref{subsection:scales:BMdrift}-\ref{subsection:sclaes:diffusion0infvariation} deal with the various cases of convergence for the scale functions. Subsection~\ref{subsection:functionals_of_scale_fncs_conv}  establishes a convergence result for the derivatives of $\Wq$ in the case when $\diffusion>0$. 

\subsection{Auxiliary technical results}\label{subsection:auxiliary_technical_results}
[Apart from the notation of Subsection~\ref{subsection:further_general_notation}, the contents of this subsection is notationally and otherwise independent from the remainder of the text. For proofs see Appendices~\ref{appendix:technical_lemmas} and~\ref{appendix:some_asymptotic_properties_of_measures_on_R}.]

\subsubsection{Some estimates and bounds}\label{subsection:some_technical_estimates_and_bounds}

\begin{lemma}\label{lemma:technical:estimates}
For every $\gamma_\star\in\mathbb{R}_+$ and $h_\star\in \mathbb{R}^+$, there is an $A_0\in (0,\infty)$ such that for all $\gamma\in [-\gamma_\star,\gamma_\star]$, $h\in (0,h_\star)$ and then all $s\in [-\pi/h,\pi/h]$ (with $\beta=\gamma+is$):
\begin{enumerate}[(i)]
\item\label{lemma:technical:estimates:i} $\left\vert\frac{1}{2h}(e^{\beta h}-e^{-\beta h})\right\vert\leq A_0\vert \beta\vert$. 
\item\label{lemma:technical:estimates:ii} $\vert e^{\beta h}-1\vert\leq A_0\vert \beta \vert h$.
\item\label{lemma:technical:estimates:iii} $\left\vert \frac{1}{h}(e^{\beta h}-1)-\beta\right\vert\leq A_0h \vert\beta \vert^2$.
\item\label{lemma:technical:estimates:iv} $\left\vert\frac{1}{2h^2}(e^{\beta h}+e^{-\beta h}-2)\right\vert\leq A_0\vert \beta\vert^2$. 
\item\label{lemma:technical:estimates:rough:I} $\left\vert\frac{1}{2h^2}(e^{\beta h}+e^{-\beta h}-2-(\beta h)^2)\right\vert\leq A_0 h^2\vert \beta\vert^4.$
\item \label{lemma:technical:estimates:rough:II} $\left\vert \frac{1}{2h}(e^{\beta h}-e^{-\beta h}-2\beta h)\right\vert\leq A_0 h^2\vert \beta\vert^3$.
\end{enumerate}

Further:
\begin{enumerate}[(a)]
\item \label{lemma:technical:coercivity:xx} For any $\xi \in [0,2\pi)$,
% the mapping 
$\left(u\mapsto (1-\cos(u))/u^2\right)$ is bounded away from $0$ for $u\in [-\xi,\xi]$. 
\item \label{lemma:technical:coercivity:a} For any $\xi\in [0,2\pi)$,
% the mapping 
$\left(u\mapsto \frac{1}{u^2}\left(\cosh(u)-1\right)\right)$ is bounded away from $0$ for $\Im u\in [-\xi,\xi]$. 
\item \label{lemma:technical:coercivity:b} For any $\xi\in [0,2\pi)$ and $L\in\mathbb{R}$,
% the mapping 
$\left(u\mapsto \frac{1}{u}\left(e^u-1\right)\right)$ is bounded away from $0$ for $\Im u\in [-\xi,\xi]$ and $\Re u\geq L$. 
\end{enumerate}
\end{lemma}

\begin{remark}
In \ref{lemma:technical:coercivity:xx}-\ref{lemma:technical:coercivity:b}, at $u=0$, the relevant limit (which exists) is taken, in order to make the mappings well-defined in this single point.
\end{remark}

\begin{lemma}\label{lemma:convergence:a_uniformly_bounded_family}
The family of functions $f_{\gamma_0}:[-\pi,\pi]\to\mathbb{R}$, defined by: $f_{\gamma_0}(s_0):=\left(1-\frac{\gamma_0+is_0}{1-e^{-\gamma_0-is_0}}\right)\frac{1}{\gamma_0+is_0}$, $s_0\in [-\pi,\pi]$, is uniformly bounded for $\gamma_0$ belonging to bounded subsets of $\mathbb{R}\backslash \{0\}$. 
\end{lemma}

\begin{lemma}\label{lemma:fundamental:inequalities}
\begin{enumerate}[(i)]
\item\label{lemma:technical:3rdorder:i} For any $z\in \Cll$: $\vert e^z-1\vert\leq \vert z\vert$.
\item\label{lemma:technical:3rdorder:ii} There exists $C\in (0,\sqrt{5/2}]$, such that for any $z\in \Cll$: $\vert e^z-z-1\vert\leq C\vert z\vert^2$.
\end{enumerate}
\label{lemma:technical:3rdorder}
\end{lemma}

\begin{lemma}\label{lemma:technical:2ndorder}
\begin{enumerate}[(i)]
\item\label{lemma:technical:2ndorder:i} Let $\{x,y\}\subset \mathbb{R}_-$, $\vert x-y\vert\leq h/2\leq \vert y\vert$. Then for any $\alpha\in\Crr$ we have: $\left\vert e^{\alpha x}-\alpha x-(e^{\alpha y}-\alpha y)\right\vert\leq 2 h \vert\alpha\vert^2\vert y\vert$.
\item\label{lemma:technical:2ndorder:ii} There exists $C\in (0,\infty)$ such that whenever $\{x,y\}\subset \mathbb{R}_-$, $\alpha\in\Crr$ and $\vert x-y\vert\leq h/2\leq \vert y\vert$, we have: $\left\vert e^{\alpha x}-\alpha x-\frac{\alpha^2x^2}{2}-\left(e^{\alpha y}-\alpha y-\frac{\alpha^2y^2}{2}\right)\right\vert\leq C hy^2\vert \alpha\vert^3$.
\end{enumerate}
\end{lemma}

\subsubsection{Some asymptotic properties at $0$ of measures on $\mathbb{R}$}\label{subsection:some_asymptotic_properties_at_zero}
%[Again, apart from the notation of Subsection~\ref{subsection:further_general_notation}, the contents of this subsection is notationally and otherwise independent from the remainder of the text. For proofs see Appendix~\ref{appendix:some_asymptotic_properties_of_measures_on_R}.]

Let $\nu$ be a measure on $\mathbb{R}$. 
\begin{proposition}\label{appendix:proposition:integrability:main}
If $\nu$ is compactly supported and locally finite in $\mathbb{R}\backslash \{0\}$, then: $\int\vert x\vert \nu(dx)<\infty$, precisely when $\int_1^\infty \frac{ds}{s^2}\int \nu(dx)(1-\cos(sx))<\infty.$
\end{proposition}

\begin{lemma}\label{lemma:fubini}
Let $r\geq 0$. Then:
\begin{enumerate}[(i)]
\item\label{fubini:one} $\int_{[0,r]} x \nu(dx)=\int_{[0,r]}\nu((t,r])dt$.
\item\label{fubini:two} $\int_{[-r,r]}\vert x\vert \nu(dx)=\int_{[0,r]}\nu([-r,r]\backslash [-t,t])dt$.
\item\label{fubini:three} $\int_{[-1,1]\backslash [-r,r]}\vert x\vert \nu(dx)=rg(r)+\int_{(r,1]}g(t)dt$ whenever $r\leq 1$ and with $g(r):=\nu([-1,1]\backslash [-r,r])$.
\item\label{fubini:four} $\int_{[0,r]}x^2\nu(dx)=2\int_{[0,r]}t\nu((t,r])dt$.
\item\label{fubini:five} $\int_{[-r,r]}x^2\nu(dx)=2\int_{[0,r]}t\nu([-r,r]\backslash [-t,t])dt$.
\end{enumerate}
\end{lemma}

\begin{proposition}\label{proposition:generally_on_measure_asymptotics}
Suppose $g(\delta):=\nu([-1,1]\backslash [-\delta,\delta])\sim 1/\delta^{1+\alpha}$ as $\delta\downarrow 0$, so that in particular $g(\delta)$ is finite for all $0<\delta\leq 1$ and necessarily $\alpha\geq -1$. Then:
\begin{enumerate}[(a)]
\item\label{asymptoptic:a} $\gamma(\delta):=\delta^2\nu([-1,1]\backslash [-\delta,\delta])\sim \delta^{1-\alpha}$ as $\delta\downarrow 0$. 
\item\label{asymptoptic:b} $\int_{[-1,1]} x^2\nu(dx)<\infty$, iff $\alpha<1$. If $\alpha\in (-1,1)$, then $\int_{[-\delta,\delta]} x^2\nu(dx)\sim \delta^{1-\alpha}$ as $\delta\downarrow 0$. 
\item\label{asymptoptic:c} $\nu$ is a L\'evy measure, iff $\nu(\mathbb{R}\backslash [-1,1])<\infty$, $\nu(\{0\})=0$ and $\alpha<1$. 
\item\label{asymptoptic:d} $\int_{[-1,1]} \vert x\vert \nu(dx)=\infty$, iff $0\leq \alpha$. 
\item\label{asymptoptic:e} Finally, as $\delta\downarrow 0$, if $\alpha>0$, $\int_{(\delta,1]}g(t)dt\sim \delta^{-\alpha}$ and if $\alpha=0$, then $\int_{(\delta,1]}g(t)dt\sim \vert \log\delta\vert$. In particular, $\zeta(\delta):=\delta \int_{[-1,1]\backslash [-\delta,\delta]}\vert x\vert \nu(dx)\sim \delta^{1-\alpha}$, when $\alpha>0$, respectively $\zeta(\delta)\sim \delta\vert \log\delta\vert$, when $\alpha=0$.
\end{enumerate}
\end{proposition}

\begin{proposition}\label{proposition:generally_on_measure_asymptotics_bis}
Define $\lambda(dx)=\mathbbm{1}_{(-1,1)}(x)\vert x\vert\nu(dx)$. Furthermore, let $\alpha\in (0,1)$, with $\limsup_{\delta\downarrow 0}\nu((-1,1)\backslash [-\delta,\delta])\delta^{1+\alpha}<\infty$. Then each of the quantities $\limsup_{\delta\downarrow 0}\lambda((-1,1)\backslash [-\delta,\delta])\delta^{\alpha}$, $\limsup_{\delta\downarrow 0}\int_{[-\delta,\delta]}x^2\nu(dx)\delta^{\alpha-1}$ and $\sup_{s\in \mathbb{R}\backslash \{0\}}\frac{1}{\vert s\vert^\alpha}\vert \int (e^{isy}-1)\lambda(dy)\vert$ is finite.
\end{proposition}

\subsection{Estimating the absolute difference $\vert \psi^h-\psi\vert$}\label{subsection:the_difference}
Recall that $\beta$ is defined to be $\gamma+is$ throughout. We establish in this subsection two key properties of the difference $\psi^h-\psi$: 
\begin{enumerate}[label=($\Delta_\arabic{*}$),ref=($\Delta_\arabic{*}$)]
\item\label{delta1} $\psi^h\to \psi$ as $h\downarrow 0$, uniformly in bounded subsets of $\Crr$. 
\item\label{delta2} There exists $A_0\in (0,\infty)$, such that for all $h\in (0,h_\star\land 2)$ and then all $s\in [-\pi/h,\pi/h]$, the following holds (see Table~\ref{table:reference} for values of the parameter $V$):
\begin{enumerate}[(i)]
\item When $\diffusion>0$: $$\vert \psi^h-\psi\vert(\beta)\leq A_0\left[h^2\vert \beta\vert^4+h\xi(h/2)\vert \beta\vert^3+h\vert \beta \vert+V\zeta(h/2)\vert \beta\vert^2\right].$$ In particular, if in addition $\kappa(0)<\infty$, we have: $$\vert \psi^h-\psi\vert(\beta)\leq A_0[ h^2\vert \beta\vert^4+h\vert \beta\vert^2].$$ If, moreover, $\measure(\RR)<\infty$, then $$\vert \psi^h-\psi\vert(\beta)\leq A_0[h^2\vert \beta\vert^4+h\vert \beta\vert].$$
\item When $\diffusion=0$: $$\vert \psi^h-\psi\vert(\beta)\leq A_0\left[h\xi(h/2)\vert \beta\vert^3+(h+\zeta(h/2))\vert \beta\vert^2\right].$$ If in addition $\kappa(0)<\infty$, then: $$\vert \psi^h-\psi\vert(\beta)\leq A_0h\vert \beta\vert^2.$$ 
\end{enumerate}
\end{enumerate}
\noindent \emph{Proof of \ref{delta1} and \ref{delta2}.}
Indeed, suppose $\diffusion>0$ (respectively $\diffusion=0$), so that we are working under scheme 1 (respectively scheme 2). We decompose, referring to \eqref{eq:laplace_exponent}, \eqref{eq:laplace_exponent:scheme1} and \eqref{eq:laplace_exponent:scheme2}, the difference $\psi^h-\psi$ into terms, which allow for straightforward estimates. To wit, for any $h_0\in (0,2]$ and $\rho_0>0$, there exists $A_0\in (0,\infty)$, such that for all $h\in (0,h_0)$ and then all $s\in [-\pi/h,\pi/h]$, as well as all $\rho\in [0,\rho_0]$  (with $\alpha=\rho+is$): 

\begin{enumerate}[(1)]
\item\label{item:diffusion>0:first} $\left\vert \diffusion\left(\frac{e^{\alpha h}+e^{-\alpha h}-2}{2h^2}-\frac{\alpha^2}{2}\right)\right\vert\leq A_0h^2\vert \alpha\vert^4$ by \ref{lemma:technical:estimates:rough:I}  of Lemma~\ref{lemma:technical:estimates} (respectively this is void).
\item\label{item:diffusion>0:2}  $\left\vert c_0^h\left(\frac{e^{\alpha h}+e^{-\alpha h}-2}{2h^2}-\frac{\alpha^2}{2}\right)\right\vert\leq A_0h^2\xi(h/2)\vert\alpha\vert^4$ by \ref{lemma:technical:estimates:rough:I}  of Lemma~\ref{lemma:technical:estimates}. 
\item\label{item:diffusion>0:3} By a direct Taylor expansion:
\begin{eqnarray*}
&\phantom{\leq}&\left\vert V\int_{[-h/2,0)}y^2\frac{\alpha^2}{2}\measure(dy)-\int_{[-h/2,0)}\left(e^{\alpha y}-V\alpha y-1\right)\measure(dy)\right\vert\\
&\leq& A_0\left(V\vert \alpha\vert^3h\xi(h/2)+(1-V)\vert \alpha\vert h\right).
\end{eqnarray*} 
\item\label{item:diffusion>0:4} $\left\vert \mu\left(\frac{e^{\alpha h}-e^{-\alpha h}}{2h}-\alpha\right)\right\vert\leq A_0h^2\vert \alpha\vert^3$ by \ref{lemma:technical:estimates:rough:II} (respectively $\left\vert \mu\left(\frac{e^{\alpha h}-1}{h}-\alpha\right)\right\vert\leq A_0h\vert \alpha\vert^2$ by \ref{lemma:technical:estimates:iii} of Lemma~\ref{lemma:technical:estimates}).
\item\label{item:diffusion>0:5} $\left\vert \left(-\mu^h\right)\left(\frac{e^{\alpha h}-e^{-\alpha h}}{2h}-\alpha\right)\right\vert\leq A_0Vh^2\kappa(h/2)\vert \alpha\vert^3$ (respectively $\left\vert \left(-\mu^h\right)\left(\frac{e^{\alpha h}-1}{h}-\alpha\right)\right\vert\leq  A_0Vh\kappa(h/2)\vert \alpha\vert^2$), by the same token, since in fact:
\begin{equation*}
-\mu^h =- \sum_{y\in\Zh^{--}}y\int_{A_y^h}\ind(z)\measure(dz)\leq -2V\int_{[-1,-h/2)}z\measure(dz)= 2V\kappa(h/2).
\end{equation*}
\item\label{item:diffusion>0:last} Finally: 
\begin{eqnarray*}
&\phantom{\leq}&\left\vert\sum_{y\in\Zh^{--}}c_y^h\left(e^{ \alpha y}-1\right)-\alpha \mu^h-\int_{(-\infty,-h/2]}\left(e^{\alpha z}-\alpha z\ind(z)-1\right)\measure(dz)\right\vert\\
&\leq& \left\vert\!\sum_{y\in \Zh^{--}}\int_{A^h_y\cap (-\infty,-V)}\left(e^{\alpha y}-e^{\alpha z}\right)\measure(dz)\right\vert+\left\vert\sum_{y\in \Zh^{--}}\!\int_{A^h_y\cap [-V,0)}\left(e^{\alpha y}-e^{\alpha z}-V\alpha (y-z) \right)\measure(dz)\right\vert\\
&\leq&A_0\left(\vert\alpha \vert h+V\vert \alpha\vert^2h\kappa(h/2)\right),
\end{eqnarray*}
by \ref{lemma:technical:estimates:ii} of Lemma~\ref{lemma:technical:estimates} (since $\vert e^{\alpha y}-e^{\alpha z}\vert\leq \vert 1-e^{\alpha (y-z)}\vert$) and \ref{lemma:technical:2ndorder:i} of Lemma~\ref{lemma:technical:2ndorder}.
\end{enumerate}
From the estimates \ref{item:diffusion>0:first}-\ref{item:diffusion>0:last}, \ref{delta1} follows, since any compact subset of $\Crr$ is contained in the rectangle $[0,\rho_0]\times [-\pi/h_,\pi/h]$, for all $h\in (0,h_0)$, so long as $\rho_0$ is chosen large enough, and $h_0$ small enough. On the other hand \ref{delta2} follows by taking $h_0=h_\star\land 2$ and $\rho_0=\rho=\gamma$, so that $\alpha=\beta$.\qed

\begin{remark}
Pursuant to \ref{delta1} above and Remark~\ref{remark:convergencePhi}, we assume henceforth that $h_\star$ has already been chosen small enough, so that in addition $\gamma>\Phi^h(q)$ for all $h\in (0,h_\star)$
\end{remark}

\subsection{Estimating the absolute difference $\vert A^{h\prime}-A'\vert$ and growth of $A^{h\prime}$ at infinity}\label{subsection:_the_difference_of_derivatives}
We establish here the following two properties pertaining to the derivatives $A'$ and $A^{h\prime}$, $h\in (0,h_\star)$:

\begin{enumerate}[label=($\Delta^{\prime}_\arabic{*}$),ref=($\Delta^{\prime}_\arabic{*}$)]
\item\label{delta1D} For any finite $h_0\in  (0,h_\star]$, there exists an $A_0\in (0,\infty)$ such that for all $h\in [0,h_0)$ and then all $s\in (-\pi/h,\pi/h)$: $$\vert A^{h\prime}(s)\vert\leq A_0\vert \beta\vert^{\epsilon-1},$$ where $\epsilon=2$, if $\diffusion>0$; $\epsilon=1$, if $\diffusion=0$ and $\kappa(0)<\infty$; finally, if $\diffusion=0$ and $\kappa(0)=\infty$, then $\epsilon$ must satisfy \ref{assumption:salient:one} of Assumption~\ref{assumption:salient} from the Introduction. 
\item\label{delta2D} There is an $A_0\in (0,\infty)$, such that for all $h\in (0,h_\star\land 2)$ and then all $s\in [-\pi/h,\pi/h]$, the following holds:
\begin{enumerate}[(i)]
\item When $\diffusion>0$: $$\vert A^{h\prime}(s)-A'(s)\vert\leq A_0(h^2\vert \beta\vert^3+h\xi(h/2)\vert \beta\vert^2+(h+\zeta(h/2))\vert \beta\vert).$$ 
\item When $\diffusion=0$: $$\vert A'(s)-A^{h\prime}(s)\vert\leq A_0\left[h+\zeta(h/2)+\xi(h/2)\right]\vert \beta\vert.$$ If in addition $\kappa(0)<\infty$, then: $$\vert \psi^h-\psi\vert(\beta)\leq A_0h\vert \beta\vert.$$ 
\end{enumerate}
\end{enumerate}
\noindent \emph{Proof of \ref{delta1D} and \ref{delta2D}.}
Indeed, we have, using differentiation under the integral sign, for $s\in \mathbb{R}$: 
\begin{equation}\label{eq:FV:derivative_of_A}
A'(s)=i\diffusion \beta+i\drift+i\int_{(-\infty,0)}z\left(e^{\beta z}-\ind(z)\right)\measure(dz).
\end{equation}
Suppose now first that $\diffusion>0$. Then, for $h\in (0,h_\star)$ and $s\in [-\pi/h,\pi/h]$:
\footnotesize
\begin{equation*}
A^{h\prime}(s)=i\drift\frac{e^{\beta h}+e^{-\beta h}}{2}-i\drift^h\left(\frac{e^{\beta h}+e^{-\beta h}}{2}-1\right)+(\diffusion+c_0^h)i\frac{e^{\beta h}-e^{-\beta h}}{2h}+i\sum_{y\in \Zh^{--}}\int_{A_y^h}y\left(e^{\beta y}-\ind(z)\right)\measure(dz).
\end{equation*}
\normalsize
From these expressions it follows readily, using \ref{lemma:technical:3rdorder:i} of Lemma~\ref{lemma:fundamental:inequalities}, \ref{lemma:technical:estimates:i} and \ref{lemma:technical:estimates:iv} of Lemma~\ref{lemma:technical:estimates}, $h\vert \mu^h\vert\leq 2h\kappa(h/2)\leq 4\int_{[-1,0)}y^2\measure(dy)$  and $\vert y\vert \leq 2\vert z\vert$, $e^{\gamma y}\leq e^{\gamma(z+h/2)}$ for $z\in A_y^h$, $y\in \Zh^{--}$, that $A'$ and $A^{h\prime}$ are both bounded by an affine function of $\vert s\vert$ on $s\in (-\pi/h,\pi/h)$, uniformly in $h\in (0,h_0)$ for any finite $h_0\in (0,h_\star]$. 

On the other hand, when $\diffusion=0$, we have for $h\in (0,h_\star)$ and then $s\in [-\pi/h,\pi/h]$:
\begin{eqnarray}\label{eq:FV:derivative_of_Ah}
A^{h\prime}(s)&=&\frac{ic_0^h}{2h}\left(e^{\beta h}-e^{-\beta h}\right)+i(\mu-\mu^h)e^{\beta h}+i\sum_{y\in \Zh^{--}}c_y^hye^{\beta y}\\
&=&\frac{ic_0^h}{2h}\left(e^{\beta h}-e^{-\beta h}\right)+i\mu e^{\beta h}-\mu^h(e^{\beta h}-1)+i\sum_{y\in \Zh^{--}}\int_{A_y^h}y\left(e^{\beta y}-\ind(z)\right)\measure(dz).\nonumber
\end{eqnarray}
Now, if $\kappa(0)<\infty$, it follows readily from $c_0^h\leq h\kappa(0)$, $-\drift^h\leq 2\kappa(0)$ and $\int_{(-\infty,0)}\vert y\vert e^{\gamma y}\measure(dy)<\infty$ that $A'$ and $A^{h\prime}$ are bounded, uniformly in $h\in (0,h_0)$ for any finite $h_0\in (0,h_\star]$. If, however, $\kappa(0)=\infty$ and then under Assumption~\ref{assumption:salient}, the desired conclusion of \ref{delta1D} follows from the estimates of Lemma~\ref{lemma:technical:estimates} and Proposition~\ref{proposition:generally_on_measure_asymptotics_bis} using \ref{assumption:salient:one} of Assumption~\ref{assumption:salient}. 

Finally, from the above expressions for the derivatives $A^\prime$ and $A^{h\prime}$, \ref{delta2D} follows using Lemma~\ref{lemma:technical:estimates} and a decomposition similar to the one in Subsection~\ref{subsection:the_difference}, which allowed to establish \ref{delta2}. For example, when $\diffusion=0$, we have the following decomposition of $A^{h\prime}(s)-A'(s)$ into three summands, each of which is then easily estimated ($h\in (0,h_\star\land 2)$):
\begin{enumerate}[(1)]
\item $\sum_{y\in \Zh^{--}}\int_{A_y^h}\left[y(e^{\beta y}-\indd(z))-z(e^{\beta z}-\indd(z))\right]\measure(dz)$;
\item $i(\mu-\mu^h)(e^{\beta h}-1)$;
\item $ic_0^h\left[\frac{e^{\beta h}-e^{-\beta h}}{2h}\right]-i\int_{[-h/2,0)}y(e^{\beta y}-1)\measure(dy)$. 
\end{enumerate}
\qed

\begin{remark}\label{remark:A'forFV}
Note that if $\diffusion=0$, $\kappa(0)<\infty$, then also $A'(s)=i\drift_0+i\int_{(-\infty,0)}ye^{\beta y}\measure(dy)$, $s\in \mathbb{R}$.
\end{remark}

\subsection{Coercivity of $\vert\psi^h-q\vert$}\label{subsection:coercivity}
In this subsection we establish the following coercivity property:
\begin{enumerate}[label=(C),ref=(C)]
\item\label{(C)} There exists an $h_0\in (0,h_\star]$ and a $B_0\in (0,\infty)$, such that for all $h\in [0,h_0)$ and then all $s\in (-\pi/h,\pi/h)$, the following holds (recall $\psi^0=\psi$, $\beta=\gamma+is$): $$\vert \psi^h(\beta)-q\vert\geq B_0\vert \beta\vert^\epsilon,$$  where $\epsilon=2$, if $\diffusion>0$; $\epsilon=1$, if $\diffusion=0$ and $\kappa(0)<\infty$; finally, if $\diffusion=0$ and $\kappa(0)=\infty$, then $\epsilon$ must satisfy \ref{assumption:salient:two} of Assumption~\ref{assumption:salient} from the Introduction.
\end{enumerate}
\noindent \emph{Proof of \ref{(C)}.}
(In the argument which follows, once again we refer the reader to expressions \eqref{eq:laplace_exponent}, \eqref{eq:laplace_exponent:scheme1} and \eqref{eq:laplace_exponent:scheme2}.) 

Suppose first $\diffusion>0$, so that we work under scheme 1. Consider $\psi(\beta)$. The diffusion term is certainly quadratic in $s$. The drift term (viewed as a function of $s$) is bounded by an affine function of $\vert s\vert$, and the L\'evy measure integral has subquadratic growth in $s$, as can be seen immediately by the DCT and Lemma~\ref{lemma:technical:3rdorder}: 
\begin{equation}\label{equation:subquadratic}
\lim_{R\to\infty}\sup_{\alpha\in \Crr,\vert \alpha\vert\geq R}\frac{1}{\vert \alpha\vert^2}\left\vert\int_{(-\infty,0)}\left(e^{\alpha y}-\alpha y\ind(y)-1\right)\measure(dy)\right\vert=0.
\end{equation} 
In addition $(s\mapsto (\psi-q)(\beta))$ is bounded away from zero on bounded subsets of $\RR$, by continuity and Proposition~\ref{proposition:well-posedness}. This establishes the claim for $h=0$. 

To establish coercivity for $\psi^h(\beta)-q$, $h>0$, we proceed as follows. First, by \ref{lemma:technical:estimates:i} of Lemma~\ref{lemma:technical:estimates}, for any finite $h_0\in (0,h_\star]$, there exists a $B_0\in (0,\infty)$, such that for all $h\in (0,h_0)$ and then all $s\in (-\pi/h,\pi/h)$: $$\left\vert\frac{1}{2h}(e^{\beta h}-e^{-\beta h})\right\vert\leq B_0\vert \beta\vert.$$ This controls the term involving $\mu$. Next, by \ref{lemma:technical:estimates:rough:II} of Lemma~\ref{lemma:technical:estimates}, again for any finite $h_0\in (0,h_\star]$, there are $\{A_1,A_2\}\subset (0,\infty)$, such that for all $h\in (0,h_0)$ and then all $s\in (-\pi/h,\pi/h)$: 
\begin{equation*}
\left\vert - \mu^h\left(\frac{e^{\beta h}-e^{-\beta h}}{2h}-\beta\right)\right\vert\leq A_1h^2\vert \beta\vert^3\vert \mu^h\vert\leq A_2 \vert \beta\vert^2\zeta(h/2)
\end{equation*}
with $\zeta(h/2)\to 0$ as $h\downarrow 0$. Further, just as in \eqref{equation:subquadratic}:

\begin{equation*}
\lim_{R\to\infty}\sup_{\alpha\in \Crr,\vert \alpha\vert\geq R}\sup_{h>0}\frac{1}{\vert \alpha\vert^2}\left\vert\sum_{y\in\Zh^{--}}\int_{A^h_y}\left(e^{\alpha y}-\alpha y\ind(z)-1\right)\measure(dz)\right\vert= 0,
\end{equation*} 
where, additionally, one should note that for $y\in\Zh^{--}$ and $z\in A_y^h$, $\vert y\vert\leq 2\vert z\vert$.

This, coupled with $\diffusion>0$ and \ref{lemma:technical:coercivity:a} of Lemma~\ref{lemma:technical:estimates}, implies that there exist $\{B_0,C_0\}\subset (0,\infty)$ and an $h_0\in (0,h_\star]$, such that for all $h\in (0,h_0)$ and then all $s\in (-\pi/h,\pi/h) \backslash (-C_0,C_0)$: $\vert \psi^h(\beta)-q\vert\geq B_0 s^2$. Finally, since, as $h\downarrow 0$, $\psi^h(\beta)-q\to\psi(\beta)-q$ uniformly in $s$ belonging to bounded sets, and $\psi(\beta)-q$ is bounded away from $0$ on such sets, we obtain the asserted result.

Now suppose $\diffusion=0$ (so that scheme 2 is in effect) and consider first the case when $\kappa(0)<\infty$. With regard to $\psi(\beta)$, note that $\mu_0\beta$ is linear in $s$, whereas:
\begin{equation}\label{eq:linear_coercivity}
\lim_{R\to\infty}\sup_{\alpha\in \Crr,\vert \alpha\vert\geq R}\frac{1}{\vert \alpha\vert}\left\vert\int_{(-\infty,0)}\left(e^{\alpha y}-1\right)\measure(dy)\right\vert=0,
\end{equation}
by \ref{lemma:technical:3rdorder:i} of Lemma~\ref{lemma:technical:3rdorder} and the DCT. The asserted coercivity follows immediately in the case $h=0$. 

To handle $h>0$, it will be observed first that $\mu-\mu^h\to \mu_0>0$ as $h\downarrow 0$, e.g. by the DCT. Also by the DCT, \ref{lemma:technical:3rdorder:i} of Lemma~\ref{lemma:technical:3rdorder}, and the fact that $\kappa(0)<\infty$: $$\lim_{R\to\infty}\sup_{\alpha\in \Crr,\vert \alpha\vert\geq R}\sup_{h>0}\frac{1}{\vert\alpha\vert}\left\vert \sum_{y\in \Zh^{--}} c_y^h(e^{\alpha y}-1)\right\vert=0.$$ Moreover, by \ref{lemma:technical:estimates:iv} of Lemma~\ref{lemma:technical:estimates}, for any finite $h_0\in (0,h_\star]$, there exists $A_0\in (0,\infty)$ such that for all $h\in (0,h_0)$ and then all $s\in (-\pi/h,\pi/h)$, $$\left \vert c_0^h\left(\frac{e^{\beta h}+e^{-\beta h}-2}{2h^2}\right)\right\vert\leq A_0\vert \beta\vert \int_{[-h/2,0)}\vert y\vert\measure(dy)$$ with $\int_{[-h/2,0)}\vert y\vert\measure(dy)\to 0$ as $h\downarrow 0$, since $\kappa(0)<\infty$. Coupled with \ref{lemma:technical:coercivity:b} of Lemma~\ref{lemma:technical:estimates}, the asserted coercivity follows. 

In the last instance, let $\diffusion=0$ and $\kappa(0)=\infty$. Necessarily, $V=1$ and Assumption~\ref{assumption:salient} is in effect. We control first $\Re\psi(\beta)$. Clearly $\gamma \mu$; $\int_{(-\infty,-1)}(e^{\gamma  y}\cos(sy)-1)\measure(dy)$ and $\int_{[-1,0)}(e^{\gamma y}-\gamma y-1)\measure(dy)$ are bounded in $s$, whereas (by \ref{lemma:technical:coercivity:xx} of Lemma~\ref{lemma:technical:estimates} and \ref{assumption:salient:two} of Assumption~\ref{assumption:salient}): $$\left\vert \int_{[-1,0)}e^{\gamma y}(\cos(sy)-1)\measure(dy)\right \vert\geq B_1s^2\int_{[-\pi/\vert s\vert,0)}y^2\measure(dy)\geq B_0\vert s\vert^{\epsilon},$$ for all $s\in\mathbb{R}$ with $\vert s \vert\geq K_0$, and some $\{K_0,B_0,B_1\}\subset (0,\infty)$. Coercivity for $\psi(\beta)-q$ follows. 

Now we turn our attention to $\psi^h(\beta)$ and again we control $\Re\psi^h(\beta)$. First observe that:
\begin{equation}
\underbrace{\frac{c_0^h}{2h^2}\left(e^{\gamma h}+e^{-\gamma h}-2\right)\cos(sh)}_{\text{bounded in }s}+\underbrace{\frac{c_0^h}{h^2}\left(\cos(sh)-1\right)}_{=:\mytag{$(I)$}{single_star}\leq 0}.
\end{equation}
Next, with respect to the term involving the drift $\mu$, we refer to \ref{lemma:technical:estimates:ii} of Lemma~\ref{lemma:technical:estimates} to obtain a linear bound in $s$. On the other hand we have:
\begin{equation*}
\Re\left\{\left(\frac{e^{\beta h}-1}{h}-\beta\right)\left(-\mu^h\right)\right\}=\underbrace{\frac{e^{\gamma h}(\cos(sh)-1)}{h}(-\mu^h)}_{\leq 0}+\underbrace{\frac{e^{\gamma h}-1-\gamma h}{h}(-\mu^h)}_{\text{bounded in }s},
\end{equation*}
since $-\mu^h \leq 2\kappa(h/2)$ and $\zeta(h/2) \to 0$ as $h\downarrow 0$. Finally, we consider the term: 
\begin{equation}\label{eq:coercivity:a_final_term}
\sum_{y\in\Zh^{--}}c_y^h\left(e^{\gamma y}\cos(s y)-1\right)-\sum_{y\in\Zh^{--}}\gamma y\int_{A_y^h}\indd(z)\measure(dz).
\end{equation} Certainly the part of \eqref{eq:coercivity:a_final_term} corresponding to $\mathbbm{1}_{(-\infty,-1)}\cdot\measure$ is bounded in $s$. The part of \eqref{eq:coercivity:a_final_term} corresponding to $\mathbbm{1}_{[-1,0)}\cdot \measure$ is:
\begin{eqnarray}
\nonumber &\phantom{=}&\sum_{y\in\Zh^{--}}\int_{A_y^h\cap [-1,0)}\measure(dz)\left[e^{\gamma y}\cos(s y)-\gamma y-1\right]\\
&=&\underbrace{\sum_{y\in \Zh^{--}}\int_{A_y^h\cap [-1,0)}\measure(dz)\left[e^{\gamma y}-\gamma y-1\right]}_{\text{bounded in }s}+\underbrace{\sum_{y\in\Zh^{--}}\int_{A_y^h\cap [-1,0)}\measure(dz)e^{\gamma y}\left(\cos(sy)-1\right)}_{=:\mytag{$(II)$}{double_star}\leq 0}.
\end{eqnarray}
Combining \ref{single_star} and \ref{double_star}, we have, via \ref{lemma:technical:coercivity:xx} of Lemma~\ref{lemma:technical:estimates}, for some $\{A_0,\beta_0,K_0,\alpha\}\subset (0,\infty)$, $h_0\in (0,h_\star]$, for all $h\in (0,h_0)$, and then all $s\in [-\pi/h,\pi/h]\backslash [-K_0,K_0]$: 
\begin{equation*}
\vert \Re\psi^h(\beta)\vert\geq \beta_0s^2\left(c_0^h+\sum_{y\in\Zh^{--},-y\leq \pi/\vert s\vert}y^2c_y^h\right)-A_0\vert s\vert.
\end{equation*}
Also: 
\begin{equation*}
c_0^h+\sum_{y\in\Zh^{--},-y\leq \pi/\vert s\vert}y^2c_y^h\geq\frac{4}{9}\int_{\left(-\left(\left(\frac{\pi}{\vert s\vert}-\frac{h}{2}\right)\lor \frac{h}{2}\right),0\right)}u^2\measure(du)\geq\frac{4}{9}\int_{(-\frac{1}{2}\frac{\pi}{\vert s\vert},0)}u^2\measure(du),
\end{equation*}
since either $\pi/\vert s\vert\geq h$, in which case $(\pi/\vert s\vert)-(h/2)\geq \frac{1}{2}\frac{\pi}{\vert s\vert}$, or $\pi/\vert s\vert\leq h$, in which case $\frac{h}{2}\geq   \frac{1}{2}\frac{\pi}{\vert s\vert}$. 

Using now item \ref{assumption:salient:two} of Assumption~\ref{assumption:salient}, the required coercivity follows at once. \qed

\subsection{Brownian motion with drift ($\diffusion>0=\measure(\RR)$)}\label{subsection:scales:BMdrift}
The scale functions can be calculated explicitly here, by using the recursive relations of Proposition~\ref{proposition:calculating_scale_functions:basic}. Then the following two proposition follow readily (essentially by Taylor expansions; recall also the notation from \eqref{eq:diffW} and \eqref{eq:diffZ}):

\begin{proposition}[$\diffusion>0=\measure(\mathbb{R})$ ($\Wq$ convergence)]\label{proposition:convergence:BM+drift:Wq}
Suppose $\diffusion>0=\measure(\mathbb{R})$ and let $q\geq 0$. If $q\lor \vert \mu\vert=0$, then for all $h\in (0,h_ \star)$ and all $x\in \Zh^{++}$: $\diffW=0$. If, however, $q\lor \vert \mu\vert>0$, then:
\begin{enumerate}[(i)]
\item There exist $\{A_0,h_0\}\subset (0,\infty)$ such that for all $h\in (0,h_0)$ and then all $x\in \Zh^{++}$ with $xh^2\leq 1$: $$\left\vert\diffW\right\vert\leq A_0h^2(1+x)e^{\alpha_+ x}.$$
\item For any nested sequence $(h_n)_{n\geq 1}\downarrow 0$ and then any $x\in \cup_{n\geq 1} \mathbb{Z}_{h_n}^{++}$: $$\lim_{n\to\infty}\frac{\diffWn}{h_n^2}=\frac{q^2}{2(\drift^2+2\diffusion q)}\Wq(x)+\frac{x}{\sqrt{\drift^2+2\diffusion q}}\left(e^{\alpha_+ x}\theta_+-e^{\alpha_- x}\theta_-\right).$$ (In particular, when $q=0$, this limit is $-\frac{2}{3}\frac{\drift^2 x}{(\diffusion)^3}e^{-2 \mu x/\diffusion}$.)
\end{enumerate}
Here:
\begin{eqnarray*}
\alpha_\pm&:=&\frac{-\drift\pm\sqrt{\drift^2+2 q\diffusion}}{\diffusion}\\
\theta_\pm&:=&\frac{\drift^3\sqrt{2q\diffusion+ \drift^2}\pm(\frac{1}{2}q^2(\diffusion)^2-\drift^4-\drift^2\diffusion q)}{3(\diffusion)^3\sqrt{2q\diffusion+ \drift^2}}.
\end{eqnarray*}
\end{proposition}

\begin{remark}
We note that for all $x\geq 0$: $$\Wq(x)=\frac{1}{\sqrt{\drift^2+2\diffusion q}}\left(e^{\alpha_+ x}-e^{\alpha_-x}\right)$$ (when $q\land \vert \mu\vert>0$) and $\Wq(x)
=x$ (otherwise). Observe also that, unless $q=0$, $\alpha_\pm\in\pm (0,\infty)$. 
\end{remark}

\begin{proposition}[$\diffusion>0=\measure(\mathbb{R})$ ($\Zq$ convergence)]\label{proposition:convergence:BM+drift:Zq}
Suppose $\diffusion>0=\measure(\mathbb{R})$, let $q>0$. 
\begin{enumerate}[(i)]
\item There exist $\{A_0,h_0\}\subset (0,\infty)$ such that for all $h\in (0,h_0)$ and then all $x\in \Zh^{++}$ with $xh^2\leq 1$: $$\left\vert \diffZ\right\vert\leq A_0\left[h^2(1+x)e^{\alpha_+x}+h(e^{\alpha_+ x}-e^{\alpha_- x})\right].$$
\item For any nested sequence $(h_n)_{n\geq 1}\downarrow 0$ and then any $x\in \cup_{n\geq 1} \mathbb{Z}_{h_n}^{++}$: $$\lim_{n\to\infty}\frac{\diffZn}{h_n}=-\frac{1}{2}\frac{q}{\sqrt{\drift^2+2\diffusion q}}\left(e^{\alpha_+x}-e^{\alpha_-x}\right).$$
\end{enumerate}
\end{proposition}

\subsection{Non-trivial diffusion component}
We consider the convergence when $\diffusion>0$. This case is relatively straightforward, as coercivity is very strong (namely, quadratic). Note that $\delta^0=1$ and we work under scheme 1.

\begin{proposition}[$\diffusion>0$ ($\Wq$ convergence)] \label{proposition:Wqconvergence:non-trivial_diffusion}
Suppose $\diffusion>0$ and let $q\geq 0$. 

\begin{enumerate}[(i)]
\item\label{Wqs:diffusion_positibe:i} For any $\gamma>\Phi(q)$, there are $\{A_0,h_0\}\subset (0,\infty)$ such that for all $h\in (0,h_0)$ and then all $x\in \Zh^{++}$: $$\left\vert\diffW\right\vert\leq A_0 \left(h+\zeta(h/2)+\xi(h/2)h\log(1/h)\right) e^{\gamma x}.$$ In particular, if $\kappa(0)<\infty$, then $\left\vert\diffW\right\vert\leq A_0 he^{\gamma x}$ and under Assumption~\ref{assumption:salient}, $\left\vert\diffW\right\vert\leq A_0 h^{2-\epsilon}e^{\gamma x}$. 
\item \label{Wqs:diffusion_positibe:ii} There exist:
\begin{enumerate}[(a)]
\item\label{Wqs:diffusion_postive:ii:a} a L\'evy triplet $(\diffusion,\measure,\drift)$ with $\diffusion\ne 0$ and $0<\kappa(0)<\infty$;
%\item\label{Wqs:diffusion_postive:ii:b}   a L\'evy triplet $(\diffusion,\measure,\drift)$ with $\diffusion\ne 0$ and $\kappa(0)<\infty=\measure(\mathbb{R})$;
\item\label{Wqs:diffusion_postive:ii:c}  for each $\epsilon\in (1,2)$ a L\'evy triplet $(\diffusion,\measure,\drift)$ with $\diffusion\ne 0$ and $\lambda(-1,-\delta)\sim 1/\delta^\epsilon$ as $\delta\downarrow 0$;
\end{enumerate}
and then in each of the cases \ref{Wqs:diffusion_postive:ii:a}-\ref{Wqs:diffusion_postive:ii:c} a nested sequence $(h_n)_{n\geq 1}\downarrow 0$ such that for each $q\geq 0$ there is an $x\in \cup_{n\geq 1}\mathbb{Z}_{h_n}^{++}$ with: $$\liminf_{n\to\infty}\frac{\left\vert\diffWn\right\vert}{ h_n\lor \zeta(h_n)}>0,$$ where $h_n\lor \zeta(h_n)\sim h_n$, if $\kappa(0)<\infty$ and $\sim h_n^{2-\epsilon}$, if $\kappa(0)=\infty$ as $n\to \infty$. 
\end{enumerate}
\end{proposition}

\begin{remark}\label{remark:diffusion_positive:asymptotic}
Note that if $\lambda(-1,\delta)\sim 1/\delta^\epsilon$ as $\delta\downarrow 0$, with $\epsilon\in (1,2)$, then (as $h\downarrow 0$) $\xi(h/2)\sim h^{2-\epsilon}$, and $\zeta(h/2)\sim h^{2-\epsilon}$, so that $h\xi(h/2)\log(1/h)=o(h\kappa(h/2))$. See Proposition~\ref{proposition:generally_on_measure_asymptotics}. More generally, Assumption~\ref{assumption:salient} is fulfilled if $\measure(-1,-\delta)\sim \delta^{-\epsilon}l(\delta)$ where $0<\liminf_{0+}l<\limsup_{0+}l<+\infty$ (see Lemma~\ref{lemma:fubini}). 
%It is not too difficult to see that the latter condition is neither weaker nor stronger than demanding $l$ to be slowly varying at $0+$.
\end{remark}

\begin{remark}\label{remark:under_assumption_salient_etc}
Under Assumption~\ref{assumption:salient}, it follows that $\zeta(h/2)+\xi(h/2)=O(h^{2-\epsilon})$ as $h\downarrow 0$ (see again Lemma~\ref{lemma:fubini}).
\end{remark}

\begin{proof}
First, with respect to~\ref{Wqs:diffusion_positibe:i}, we have as follows. \ref{W:a} of \eqref{eq:difference:Wqs} is seen immediately to be of order $O(h)$ by coercivity~\ref{(C)}; whereas~\ref{W:b} of \eqref{eq:difference:Wqs} is of order $O(h+h\xi(h/2)\log(1/h)+V\zeta(h/2))$ by coercivity~\ref{(C)} and the estimate of the absolute difference $\vert \psi^h-\psi\vert$~\ref{delta2}. Since $\delta^0=1$, \ref{W:c} of \eqref{eq:difference:Wqs} is void. 

Second we prove~\ref{Wqs:diffusion_positibe:ii}.
\begin{itemize}
\item We consider first \ref{Wqs:diffusion_postive:ii:a}. Take $\measure=\delta_{-1/2}$, $h_n=1/3^n$ ($n\geq 1$), $\drift=0$, $\diffusion=1$, $x\in \cup_{n \geq 1}\mathbb{Z}_{h_n}^{++}$ ($x$ is now fixed!). The goal is to establish no better than linear convergence in this case. 
\end{itemize}

Now,~\ref{W:a} of \eqref{eq:difference:Wqs} is actually of order $O(h^2)$. Indeed, in the difference to replacing $\psi(\beta)-q$ by $-\frac{1}{2}\sigma^2 s^2$, this is seen immediately to be even of order $O(h^3)$, by coercivity \ref{(C)} and a simple $\vert \int \cdot\vert\leq \int \vert\cdot\vert$ argument. On the other hand: $$\lim_{T\to\infty}\int_{(-T,T)\backslash (-\pi/h,\pi/h)}\frac{e^{isx}}{s^2}ds=O(h^2).$$ This is so by an integration by parts argument, writing: $$\frac{d}{ds}\left(\frac{e^{isx}}{s^2}\right)=\frac{ix e^{isx}}{s^2}-\frac{2e^{isx}}{s^3}.$$ We can thus focus on~\ref{W:b} of \eqref{eq:difference:Wqs}. Consider there the difference: 
\begin{equation}
(\psi^h-\psi)(\beta)=\overbrace{\frac{\diffusion}{2h^2}\left(e^{\beta h}+e^{-\beta h}-2\right)-\frac{1}{2}\diffusion \beta^2}^{\mytag{(b.1)}{diffusion:b.1}}+\overbrace{\int_{(-\infty,0)}(e^{\beta y}-1)\measure(dy)-\sum_{y\in \Zh^{--}}c_y^h(e^{\beta y}-1)}^{\mytag{(b.2)}{diffusion:b.2}}.
\end{equation} The part of \ref{W:b} in \eqref{eq:difference:Wqs} corresponding to \ref{diffusion:b.1} is, in the difference to the analogous term for Brownian motion without drift, bounded (up to a non-zero multiplicative constant) by: $$\int_{[-\pi/h,\pi/h]}ds \frac{h^2\vert \beta\vert^4 \vert \beta\vert^2}{\vert\beta\vert^8}=O(h^2)$$ (this follows by \ref{lemma:technical:estimates:iv} and \ref{lemma:technical:estimates:rough:I} of Lemma~\ref{lemma:technical:estimates}, the fact that $e^{\beta y}-1$ is uniformly bounded by $2$, and by coercivity \ref{(C)}). Since the term corresponding to just Brownian motion is shown to be of order $O(h^2)$ itself (see Proposition~\ref{proposition:convergence:BM+drift:Wq}), we can thus focus on \ref{diffusion:b.2}. The latter is $e^{-\beta/2}-e^{-(1-h_n)\beta/2}=e^{-\beta/2}(1-e^{\beta h_n/2})$. In the difference to replacing $1-e^{\beta h_n/2}$ by $-\beta h_n/2$, a term of order $O(h_n^2)$ emerges in $\diffWn$, this by \ref{lemma:technical:estimates:iii} of Lemma~\ref{lemma:technical:estimates}, and coercivity \ref{(C)}. Hence it is sufficient to study: $$\frac{1}{2\pi i}\int_{[-\pi/h_n,\pi/h_n]}e^{\beta x}\frac{e^{-\beta/2 }\beta}{(\psi-q)(\beta)(\psi^{h_n}-q)(\beta)}ds$$ which we would like bounded away from $0$, as $n\to\infty$. Now, by coercivity \ref{(C)}, and the DCT, this expression in fact converges to: $$\frac{1}{2\pi i}\int_{-\infty}^\infty \frac{e^{\beta (x-1/2)}\beta }{(\psi-q)^2(\beta)}ds=:g(x).$$ Note that $g$ is continuous in its parameter $x\in [0,\infty)$ by the DCT. Moreover, $g$ cannot vanish identically on $\cup_{n\geq 1}\mathbb{Z}_{h_n}^{++}$, since then it would do so on $\mathbb{R}_+$ by continuity. But this cannot be. Naively, since in $g$ we are looking at the inverse Laplace transform of a non-vanishing function $T$. Formally, one performs a Laplace transform of $g$, and concludes, via Fubini and Cauchy's Residue Theorem (recalling the quadratic behaviour of $\psi(\beta)$ as $\vert \beta\vert\to\infty$ over $\Crr$, see \eqref{equation:subquadratic}) that $T=\hat{g}$, where $T(\alpha):=\frac{e^{-\alpha /2}\alpha}{(\psi-q)^2(\alpha)}$ ($\alpha>\gamma$). Then $g$ vanishing implies the same of $T$, which is a clear contradiction.

\begin{itemize}
\item Consider now \ref{Wqs:diffusion_postive:ii:c}. We are seeking to establish strictly worse than linear convergence here, since $\kappa(0)=\infty$. 
\end{itemize}
For sure \ref{W:a} in \eqref{eq:difference:Wqs} is of order $O(h)$. When it comes to~\ref{W:b}, consider its decomposition, in the numerator of the integrand according to items \ref{item:diffusion>0:first}-\ref{item:diffusion>0:last} from Subsection~\ref{subsection:the_difference}. Now, \ref{item:diffusion>0:first} thus yields in~\ref{W:b} a term of order $O(h)$; \ref{item:diffusion>0:2} one of order $o(h)$; with respect to \ref{item:diffusion>0:3} we will choose a $\measure$ which falls under the scope of Remark~\ref{remark:diffusion_positive:asymptotic} and hence this will contribute a term of order $o(\zeta(h/2))$; \ref{item:diffusion>0:4} gives a term of order $o(h)$; \ref{item:diffusion>0:5} contributes as $o(h\kappa(h/2))$; whereas finally \ref{item:diffusion>0:last} yields a term of order $O(h)$ in the part corresponding to $\mathbbm{1}_{(-\infty,1)}\cdot\measure$ and the part corresponding to $\mathbbm{1}_{[-1,0)}\cdot\measure$ is where we will get sharpness of the rate from. 

So we take $\diffusion=1$, $\drift=0$, $h_n=1/3^n$ ($n\geq 1$), $\measure=\sum_{k=1}^\infty w_k\delta_{-x_k}$,  $x_k=\frac{3}{2}h_k$ and $w_k=1/x_k^\epsilon$ ($k\geq 1$). Clearly $\kappa(0)=\infty$ and by checking it on the decreasing sequence $(h_n)_{n\geq 1}$ it is clear that $\lambda(-1,-\delta)\sim \delta^{-\epsilon}$ as $\delta\downarrow 0$. Moreover, $$\left\vert\sum_{y\in \Zh^{--}}\!\int_{A^h_y\cap [-1,0)}\left(e^{\beta y}-\beta y-(e^{\beta z}-\beta z)-\frac{1}{2}\beta^2(y^2-z^2)\right)\measure(dz)\right\vert$$
yields in \eqref{eq:difference:Wqs}, by \ref{lemma:technical:2ndorder:ii} of Lemma~\ref{lemma:technical:2ndorder} and coercivity \ref{(C)}, a term of order $O(h\log(1/h))=o(h^{2-\epsilon})$. Therefore it is sufficient to study: $$\sum_{y\in \Zh^{--}}\!\int_{A^h_y\cap [-1,0)}\beta^2(y^2-z^2)\measure(dz)=\sigma_2-\sigma_1,$$ where:
\begin{equation*}
\sigma_1:=\int_{[-1,-h_n/2)}u^2\measure(du)=\sum_{k=1}^nx_k^2w_k,\quad \sigma_2:=\sum_{k=1}^n(x_k-h_n/2)^2w_k
\end{equation*}
and hence $\sigma_1-\sigma_2=2\zeta(h_n/2)-\gamma(h_n/2)\geq \zeta(h_n/2)$. Moreover, $$\int_{[-\pi/h_n,\pi/h_n]}e^{isx}\frac{\beta^2}{ [(\psi-q)(\psi^{h_n}-q)](\beta)]}ds\to\int_{\mathbb{R}}e^{isx}\frac{\beta^2}{(\psi-q)^2(\beta)}ds$$ as $n\to\infty$ by the DCT. By the usual arguments, this integral does not vanish simultaneously in all $x\in \cup_{n\geq 1}\mathbb{Z}_{h_n}^{++ }$, whence tightness obtains.
\end{proof}

\begin{proposition}[$\diffusion>0$ ($\Zq$ convergence)]\label{proposition:Zqconvergence:non-trivial_diffusion}
Suppose $\diffusion>0$, let $q>0$. 

\begin{enumerate}[(i)]
\item \label{diffusion>0:Zq:i} For any $\gamma>\Phi(q)$, there are $\{A_0,h_0\}\subset (0,\infty)$ such that for all $h\in (0,h_0)$ and then all $x\in \Zh^{++}$: $$\left\vert\diffZ\right\vert\leq A_0 \left(h+\zeta(h/2)\right) e^{\gamma x}.$$ In particular, if $\kappa(0)<\infty$, then $\left\vert\diffZ\right\vert\leq A_0 he^{\gamma x}$ and otherwise $\left\vert\diffZ\right\vert\leq A_0 \zeta(h/2)e^{\gamma x}$. 
\item \label{diffusion>0:Zq:ii}
\begin{enumerate}[(a)]
\item\label{Zqs:diffusion_postive:ii:a} There exists a nested sequence $(h_n)_{n\geq 1}\downarrow 0$, such that for any $q>0$ and any $x\in \cup_{n\geq 1}\mathbb{Z}_{h_n}^{++}$, there exists a L\'evy triplet $(\diffusion,\drift,\measure)$ with $\diffusion>0$ and $0<\kappa(0)<\infty$, and such that:
$$\liminf_{n\to\infty}\frac{\left\vert\diffZn\right\vert}{ h_n}>0.$$ 

%\item\label{Zqs:diffusion_postive:ii:b}   There exists a nested sequence $(h_n)_{n\geq 1}\downarrow 0$, such that for any $q>0$ and any $x\in \cup_{n\geq 1}\mathbb{Z}_{h_n}^{++}$, there exists a L\'evy triplet $(\diffusion,\drift,\measure)$ with $\diffusion>0$ and $\kappa(0)<\infty=\measure(\mathbb{R})$, and such that:
\item\label{Zqs:diffusion_postive:ii:c}  There exists for each $\epsilon\in (1,2)$ a L\'evy triplet $(\diffusion,\measure,\drift)$ with $\diffusion>0$ and $\lambda(-1,\delta)\sim 1/\delta^\epsilon$ as $\delta\downarrow 0$, and a nested sequence $(h_n)_{n\geq 1}\downarrow 0$ such that for each $q>0$, there is an $x\in \cup_{n\geq 1}\mathbb{Z}_{h_n}^{++}$ with:
\end{enumerate}
$$\liminf_{n\to\infty}\frac{\left\vert\diffZn\right\vert}{\zeta(h_n)}>0.$$ 
\end{enumerate}
\end{proposition}
\begin{proof}
With respect to~\ref{diffusion>0:Zq:i}, we have as follows. First, \ref{Z:a} of \eqref{eq:difference:Zqs} is $O(h^2)$ and \ref{Z:b} $O(h)$ by coercivity \ref{(C)} (and by Lemma~\ref{lemma:convergence:a_uniformly_bounded_family} in the case of~\ref{Z:b}). Second, \ref{Z:c} is $O(h+\zeta(h/2))$ by coercivity~\ref{(C)} and the estimate \ref{delta2}. 

Next we show~\ref{diffusion>0:Zq:ii}. 
\begin{itemize}
\item We consider first \ref{Zqs:diffusion_postive:ii:a}. Take $\diffusion=1$, $\drift=0$ and $\measure=\alpha \delta_{-1}$, where we fix $\alpha>0$. The idea is to note that convergence is `tightly linear' in the Brownian motion case (see Proposition~\ref{proposition:convergence:BM+drift:Zq}), and then to show that by taking $\alpha$ small enough, we do not spoil this.
\end{itemize}
Now, remark that:
\begin{itemize}
\item[---] As $\alpha\downarrow 0$, $\psi$ is nondecreasing, hence $\Phi(q)$ is nonincreasing, and so a $\gamma>\Phi(q)$ can be chosen, uniformly in all $\alpha$ bounded.
\item[---] Moreover, the presence of $\alpha$ does not affect coercivity, which is in addition uniform in all $\alpha$ small enough. Indeed, just take $\gamma>\Phi^{BM}(q)$, where $\Phi^{BM}$ corresponds to the Brownian motion part of this L\'evy process. Then $\vert \psi^h_{BM}(\beta)-q\vert\geq B_0\vert \beta\vert^2$, for all $s\in (-\pi/h,\pi/h)$, for all $h\in [0,h_0)$, for some $\{B_0,h_0\}\subset (0,\infty)$. The part of $\psi^h(\beta)$ corresponding to the CP term is bounded uniformly in $s$ and all $h$ small enough (including $0$), and moreover, scales with $\alpha$. Hence there are $\{B_0, h_0,\alpha_0\}\subset (0,\infty)$ such that for all $\alpha\in (0,\alpha_0)$: $\vert \psi^h(\beta)-q\vert\geq B_0\vert \beta\vert^2$, for all $s\in (-\pi/h,\pi/h)$, for all $h\in [0,h_0)$. 
\end{itemize}
Take now $h_n=1/2^n$ ($n\geq 1$), and a fixed $x\in \cup_{n\geq 1}\mathbb{Z}_{h_n}^{++}$. We show that \eqref{eq:difference:Zqs}, when looked in the difference to the analogous expression for the Brownian motion part, is of order $o(h)+\alpha O(h)$ (i.e. bounded in absolute value by terms either decaying faster than linear, or else with a coefficient that scales with $\alpha$). Indeed, the difference of~\ref{Z:b}s in \eqref{eq:difference:Zqs} follows readily as being $\alpha O(h)$. In addition,~\ref{Z:a} of \eqref{eq:difference:Zqs} is of order $O(h^2)$. Also~\ref{Z:c} is of order $o(h)$ except in the part corresponding to the CP term, which is itself $\alpha O(h)$. Thus, choosing $\alpha$ small enough, the desired sharpness obtains. However, at least in principle, the choice of $\alpha$ depends on $q$ and $x$, hence note the formulation of the proposition.
\begin{itemize}
\item Consider now \ref{Wqs:diffusion_postive:ii:c}. Clearly here the same example works as for the functions $\Wq$: the presence of the extra $1/\beta$ in the integrand of~\ref{Z:c} of \eqref{eq:difference:Zqs} is of no consequence (if anything, beneficial), and the rest of the terms are of order $O(h)$ anyhow.
\end{itemize}
\end{proof}

\subsection{$\diffusion=0$, finite variation paths}\label{subsection:scales:FV}
In this subsection we study the convergence when $\diffusion=0$ and $\kappa(0)<\infty$. Note that in this case necessarily $\mu_0>0$, whereas $\delta^0=0$, and we work under scheme 2. The estimates are delicate here, since coercivity is weak (namely, linear). 

%We prepare the following derivatives (for $s\in\mathbb{R}$): 
%\begin{equation*}
%A'(s)=i\mu_0+i\int_{(-\infty,0)}ye^{\beta y}\measure(dy)
%\end{equation*}
%and (for $s\in [-\pi/h,\pi/h]$):
%\begin{equation*}
%A^{h\prime}(s)=i(\mu-\mu^h)e^{\beta h}+i\sum_{y\in\Zh^{--}}yc_y^he^{\beta y}+c_0^hi\frac{e^{\beta h}-e^{-\beta h}}{2h},
%\end{equation*}
%remarking that $A'$, as well as $A^{h\prime}$ are bounded, uniformly in $h$ bounded. Indeed, this follows from the facts that $\int_{(-\infty,0)}\vert y\vert e^{\gamma y}\measure(dy)<\infty$, $c_0^h=O(h)$, \ref{lemma:technical:estimates:i} of Lemma~\ref{lemma:technical:estimates} and by the elementary estimate $\vert\sum_{y\in\Zh^{--}}yc_y^he^{\beta y}\vert\leq 2e^{\gamma h/2}\int \vert y\vert e^{\gamma y}\measure(dy)$. 

We make the following key observation. By decomposing: 
\begin{equation}\label{eq:FVpaths_decomposition}
A(s)=\psi(\beta)-q=\overbrace{\mu_0\gamma+\int(e^{\gamma y}\cos(sy)-1)\measure(dy)-q}^{=:A_e(s)}+\overbrace{is\mu_0+i\int e^{\gamma y}\sin(sy)\measure(dy)}^{=:A_o(s)}
\end{equation} into its even and odd part, it can be shown that, crucially, 
\begin{equation}\label{eq:scales:crucial}
\int_1^\infty \frac{ds}{s^2}\vert A_e(s)\vert<\infty.
\end{equation} 
See Proposition~\ref{appendix:proposition:integrability:main}. On the other hand, a similar argument to \eqref{eq:linear_coercivity}, shows that $\vert A_o(s)\vert\geq B_0 \vert s\vert$ for all $s\notin [a_0,a_0]$, for some $\{a_0,B_0\}\subset \mathbb{R}$. We shall refer to the latter property as ``coercivity of $A_o$''. 

%Furthermore, by rewriting:
%\begin{equation}\label{eq:FV:derivative_of_A}
%A'(s)=i\mu+i\int_{(-\infty,0)}(e^{\beta z}-\mathbbm{1}_{[-1,0)}(z))z\measure(dz)
%\end{equation}
%and 
%
%\begin{equation}\label{eq:FV:derivative_of_Ah}
%A^{h\prime}(s)=i(\mu-\mu^h)e^{\beta h}+i\mu^h+i\sum_{y\in\Zh^{--}}\int_{A_y^h}(e^{\beta y}-\mathbbm{1}_{[-1,0)}(z))y\measure(dz)+c_0^hi\frac{e^{\beta h}-e^{-\beta h}}{2h}
%\end{equation}
%we see immediately, and in a standard fashion, that there is an $A_0 \in (0,\infty)$, such that for all $h\in (0,h_\star\land 2)$ and then all $s\in [-\pi/h,\pi/h]$: $\vert A^{h\prime}(s)-A'(s)\vert\leq A_0h\vert \beta\vert$. 

Note also that by the DCT, $\mu-\mu^h\to \mu_0$ as $h\downarrow 0$. 

We next prove a key  lemma. While \eqref{lemma:FV:one} and \eqref{lemma:FV:three} thereof will both be used explicitly in the sequel, the same cannot be said of \eqref{lemma:FV:two}. Nevertheless, 
%even 
the proof of the latter is instructive of the techniques which we will be using, and so (also for the sake of 
%/relative/ 
completeness) we choose to keep it. 

\begin{lemma}\label{lemma:FV}
Suppose $\diffusion=0$, $\kappa(0)<\infty$. Let $\{l,a,b,M\}\subset \mathbb{N}_0$ and let $h_0\in (0,h_\star]$ be given by the coercivity condition \ref{(C)}.

\begin{enumerate}
\item\label{lemma:FV:one} If $a+b+l\geq M+1$, then: $$\sup_{(h,z)\in (0,h_0)\times \mathbb{R}}\left\vert \int_{[-\frac{\pi}{h},\frac{\pi}{h}]}e^{isz}\frac{(h\land 1)^{l-l\land (M+1)}h^{l\land (M+1)}s^M}{A(s)^{a}A^h(s)^{b}}ds\right\vert<\infty.$$
\item\label{lemma:FV:two} If only $a+b+l\geq M$, then: $$\sup_{(h,z)\in (0,h_0\land K)\times \mathbb{R}}\left\vert z \int_{[-\frac{\pi}{h},\frac{\pi}{h}]}e^{isz}\frac{h^ls^M}{A(s)^{a}A^h(s)^{b}}ds\right\vert<\infty$$ for any $K\in (0,\infty)$.
\item\label{lemma:FV:three} If even $a+b+l\geq M+2$, then: $$\sup_{(h,x,z)\in (0,h_0)\times \mathbb{R}\times (\mathbb{R}\backslash \{0\})}\left\vert\frac{1}{z } \int_{[-\frac{\pi}{h},\frac{\pi}{h}]}e^{isx}\frac{(e^{isz}-1)(h\land 1)^{l-l\land (M+2)}h^{l\land (M+2)}s^M}{A(s)^{a}A^h(s)^{b}}ds\right\vert<\infty.$$ 
\end{enumerate}
\end{lemma}

\begin{remark}
Suppose $l=b=0$ and $a=M+1$ (respectively $a=M$, $a=M+2$) for simplicity (indeed the proof of \eqref{lemma:FV:one} and \eqref{lemma:FV:three} will be reduced to this case using \ref{(C)} and \ref{delta2}, whereas \eqref{lemma:FV:two} will (essentially) follow by an application of \eqref{lemma:FV:one}). Then for large $\vert s\vert$, the integrand in \eqref{lemma:FV:one} (respectively \eqref{lemma:FV:two}, \eqref{lemma:FV:three}) behaves as $\sim e^{isz}/s$ (respectively $e^{isz}$, $e^{isx}(e^{isz}-1)/s^2$) in the variable $s$. It is then not surprising that the proof of the claims is essentially a modification of the argument 
%which would have to be made to show 
implying that (in the sense of Cauchy's principal values, as appropriate) $\int_{[-\pi/h,\pi/h]}(e^{isz}/s)ds$, $\int_{[-\pi/h,\pi/h]} z e^{isz}ds$ and $\int_{[-\pi/h,\pi/h]}(e^{isx}(e^{isz}-1)/s^2)ds/z$ are bounded in the relevant suprema (as they are).
\end{remark}
\begin{proof}
We use in the proof, without explicit reference, as has already often been the case, the observation that  $\vert sh\vert\leq\pi$ within the integration domain and the basic $\vert \int\cdot ds\vert\leq\int\vert\cdot\vert ds$-argument; it will be clear whenever these are being applied. Further, we sometimes employ, but do not always refer to, some elementary trigonometric inequalities, specifically $\vert \sin(v)\vert\leq \vert v\vert$, $1-\cos(v)\leq v^2/2$ ($v\in \mathbb{R}$), as well as \ref{lemma:technical:3rdorder:i} of Lemma~\ref{lemma:fundamental:inequalities}. 
Finally, note that the integrands in the formulation of this lemma are certainly locally bounded, by coercivity \ref{(C)}, and (hence) the integrals well-defined. Indeed, by the same token, it is only non-trivial to show the finiteness of the suprema in \eqref{lemma:FV:one} and \eqref{lemma:FV:three} for $h\in (0,h_0\land 1)$, a restriction which we therefore make outright. In this respect, note that in \eqref{lemma:FV:two} the restriction to bounded $h$ is made a priori.

Consider now first \eqref{lemma:FV:one}. By coercivity \ref{(C)}, it is assumed without loss of generality that $l\leq M$, then $a+b=M-l+1$ and finally, $l=0$. 

Next, it will be sufficient to consider the case when $b=0$, since $A(s)$ can then be successively replaced in the denominator by $A^h(s)$ modulo a quantity, which, using coercivity \ref{(C)} and the estimates \ref{delta2}, remains bounded (in the supremum over $(h,z)\in (0,h_0\land 1)\times \mathbb{R}$). 

Recall now the decomposition $A=A_e+A_o$ of \eqref{eq:FVpaths_decomposition}. With $a_0$ as above, it will furthermore be sufficient to establish \eqref{lemma:FV:one} with the integration region $[-\pi/h,\pi/h]\backslash [-a_0,a_0]$ in place of $[-\pi/h,\pi/h]$, the integrand being locally bounded in the supremum over $(h,z)\in (0,h_0\land 1)\times \mathbb{R}$ (again by \ref{(C)}). Moreover, one can then successively replace $A(s)$ by $A_o(s)$, using coercivity of $A_o$ and \ref{(C)}, as well as \eqref{eq:scales:crucial}. Again this is done modulo a term which remains bounded in the supremum over $(h,z)\in (0,h_0\land 1)\times \mathbb{R}$. Hence we need only establish the finiteness of the quantity: $$\sup_{(h,z)\in (0,h_0\land 1)\times \mathbb{R}}\left\vert \int_{[-\frac{\pi}{h},\frac{\pi}{h}]\backslash [-a_0,a_0]}e^{isz}\frac{s^{M}}{A_o(s)^{M+1}}ds\right\vert.$$ Owing to the fact that the quotient in the integrand is odd in $s$, we may clearly restrict the supremum to $z\in \mathbb{R}\backslash \{0\}$, replacing also $e^{isz}$ by $\sin(sz)$ therein. A change of variables $u=s\vert z\vert$ then leads us to consider: 
\begin{equation}\label{eq:scales:FV:fund}
\int_{[-\frac{\pi}{h}\vert z\vert,\frac{\pi}{h}\vert z\vert]\backslash [-a_0\vert z\vert,a_0\vert z\vert]}\sin(u)\frac{u^M}{(\vert z \vert A_o(u/\vert z\vert))^{M+1}}du,
\end{equation}
whose finiteness in the supremum over $(h,z)\in (0,h_0\land 1)\times \mathbb{R}\backslash \{0\}$ we seek to establish. Let $\mathcal{A}:=[-\frac{\pi}{h}\vert z\vert,\frac{\pi}{h}\vert z\vert]\backslash [-a_0\vert z\vert,a_0\vert z\vert]$. By coercivity of $A_o$, and since $\vert \sin(u) \vert\leq \vert u\vert$ ($u\in  \mathbb{R}$), we do indeed get a finite quantity for the integral (in \eqref{eq:scales:FV:fund}) over $\mathcal{A}\cap [-1,1]$. On the other hand, to handle the rest of the domain, $\mathcal{A}\backslash [-1,1]$, we resort to integration by parts; 
\begin{eqnarray*}
\frac{d}{du}\left(\cos(u)\frac{u^M}{(\vert z \vert A_o(u/\vert z\vert))^{M+1}}\right)&=&-\sin(u)\frac{u^M}{(\vert z \vert A_o(u/\vert z\vert))^{M+1}}+\\
&+&\cos(u)\left(\frac{Mu^{M-1}}{(\vert z \vert A_o(u/\vert z\vert))^{M+1}}-\frac{u^M(M+1)A_o'(u/\vert z\vert)}{(\vert z\vert A_o(u/ \vert z\vert))^{M+2}}\right).
\end{eqnarray*} 
Now, once integration over $\mathcal{A}\backslash [-1,1]$ has been performed, on the left-hand side, a bounded quantity obtains, by coercivity of $A_o$. On the right-hand side we obtain from the first term the desired quantity (modulo the sign), whereas what emerges from the second term is bounded by coercivity of $A_o$, and the boundedness of $A_o'$ (see \ref{delta1D} in Subsection~\ref{subsection:_the_difference_of_derivatives}). 

We next consider \eqref{lemma:FV:two}. Here an integration by parts must be done outright, thus: 
\footnotesize
\begin{equation*}
\frac{d}{ds}\left( e^{isz} \frac{h^ls^M}{A(s)^{a}A^h(s)^{b}}\right)=iz e^{isz}\frac{h^ls^M}{A(s)^{a}A^h(s)^{b}}+e^{isz}\left(\frac{h^lMs^{M-1}}{A(s)^{a}A^h(s)^{b}}\right)-e^{isz}\left(\frac{h^ls^MaA'(s)}{A(s)^{a+1}A^h(s)^{b}}+\frac{h^ls^MbA^{h\prime}(s)}{A(s)^{a}A^h(s)^{b+1}}\right).
\end{equation*}
\normalsize
Further, once integration over $[-\pi/h,\pi/h]$ has been performed in this last equality, on the left-hand side a bounded quantity obtains by coercivity \ref{(C)}. On the right-hand side, the first term yields the desired quantity (modulo a non-zero multiplicative constant), and the second is bounded by part \eqref{lemma:FV:one}. Now, using \eqref{eq:FV:derivative_of_A} and \eqref{eq:FV:derivative_of_Ah}, via Fubini's Theorem, part \eqref{lemma:FV:one} again, and by elementary estimates such as $e^{\gamma h}$ being bounded for $h$ bounded, $\vert\sum_{y\in\Zh^{--}}yc_y^he^{\beta y}\vert\leq 2e^{\gamma h/2}\int \vert y\vert e^{\gamma y}\measure(dy)$, $c_0^h=O(h)$ and \ref{lemma:technical:estimates:i} of Lemma~\ref{lemma:technical:estimates}, the claim obtains. 

Finally, we are left to consider \eqref{lemma:FV:three}.  Again by coercivity \ref{(C)}, it is assumed without loss of generality that $l\leq M$, then $a+b=M-l+2$ and finally, $l=0$. Moreover, by the same argument as for \eqref{lemma:FV:one}, we may further insist on $b=0$, replace the integration region by $[-\pi/h,\pi/h]\backslash [-a_0,a_0]$ and finally $A$ by $A_o$. Thus we are left to analyse: 
\begin{equation}\label{eq:FV:lemma:two_integrals}
\frac{1}{z } \int_{[-\frac{\pi}{h},\frac{\pi}{h}]\backslash [-a_0,a_0]}e^{isx}\frac{(\cos(sz)-1)s^M}{A_o(s)^{M+2}}ds\quad\text{ and }\quad\frac{1}{z } \int_{[-\frac{\pi}{h},\frac{\pi}{h}]\backslash [-a_0,a_0]}e^{isx}\frac{\sin(sz)s^M}{A_o(s)^{M+2}}ds,
\end{equation}
which we require both to be bounded in the relevant supremum. 

In the first integral of \eqref{eq:FV:lemma:two_integrals} make the substitution $v=s\vert z\vert$ to obtain: $$\mathrm{sgn}(z) \int_{[-\frac{\pi}{h}\vert z\vert,\frac{\pi}{h}\vert z\vert]\backslash [-a_0\vert z\vert,a_0\vert z\vert]}e^{i vx/\vert z\vert}\frac{(\cos(v)-1)v^M}{(\vert z\vert A_o(v/\vert z\vert))^{M+2}}dv.$$ Letting, as usual, $\mathcal{A}:=[-\frac{\pi}{h}\vert z\vert,\frac{\pi}{h}\vert z\vert]\backslash [-a_0\vert z\vert,a_0\vert z\vert]$, the integral over $\mathcal{A}\backslash [-1,1]$ (respectively $\mathcal{A}\cap  [-1,1]$) is bounded by coercivity of $A_o$ (respectively the latter and since $1-\cos(v)\leq v^2/2$). 

On the other hand, in the second integral of  \eqref{eq:FV:lemma:two_integrals}, note that by the oddness of $A_o$ only $\sin(sx)$ makes a non-zero contribution. Then we may assume $x\ne 0$ and make the substitution $u=s\vert x\vert$ to arrive at: $$\frac{x}{z}\int_{[-\frac{\pi}{h}\vert x\vert,\frac{\pi}{h}\vert x\vert]\backslash [-a_0\vert x\vert,a_0\vert x\vert]}\frac{\sin(u)\sin(uz/\vert x\vert)u^M}{(\vert x\vert A_o(u/\vert x\vert))^{M+2}}du.$$ Let again $\mathcal{A}:=[-\frac{\pi}{h}\vert x\vert,\frac{\pi}{h}\vert x\vert]\backslash [-a_0\vert \vert,a_0\vert x\vert]$ be the domain of integration. It is clear that the integral over $\mathcal{A}\cap [-1,1]$ is finite, using coercivity of $A_o$ and twice $\vert \sin(w)\vert\leq \vert w\vert$ ($w\in \mathbb{R}$). To handle the remainder of the domain, $\mathcal{A}\backslash [-1,1]$, we use one last time integration by parts, thus;
\footnotesize
\begin{eqnarray*}
 \frac{x}{z}\frac{d}{du}\left(\frac{\cos(u)\sin(uz/\vert x\vert)u^M}{(\vert x\vert A_o(u/\vert x\vert))^{M+2}}\right)&=&-\frac{x}{z}\frac{\sin(u)\sin(uz/\vert x\vert)u^M}{(\vert x\vert A_o(u/\vert x\vert))^{M+2}}+\mathrm{sgn}(x)\frac{\cos(u)\cos(uz/\vert x\vert)u^M}{(\vert x\vert A_o(u/\vert x\vert))^{M+2}}+\\
&&\frac{x}{z}\frac{\cos(u)\sin(uz/\vert x\vert)Mu^{M-1}}{(\vert x\vert A_o(u/\vert x\vert))^{M+2}}-\frac{x}{z}\frac{\cos(u)\sin(uz/\vert x\vert)u^M(M+2)A_o'(u/\vert x\vert)}{(\vert x\vert A_o(u/\vert x\vert))^{M+3}}.
\end{eqnarray*}
\normalsize
The claim now obtains by coercivity of $A_o$, boundedness of $A_o'$ \ref{delta1D} and by using the elementary estimate $\vert \sin(w)\vert\leq \vert w\vert$ ($w\in \mathbb{R}$), as appropriate.  
\end{proof}

\begin{proposition}[$\diffusion=0$, $\kappa(0)<\infty$  ($\Wq$ convergence)]\label{proposition:Wqconvergence:trivial_diffusion_FV}
Suppose $\diffusion=0$ and $\kappa(0)<\infty$. Let $q\geq 0$. 

\begin{enumerate}[(i)]
\item\label{diffusion_zero_finite_variation:i} For any $\gamma>\Phi(q)$  there are $\{A_0,h_0\}\subset (0,\infty)$ such that for all $h\in (0,h_0)$ and then all $x\in \Zh^{+}$: $$\left\vert\diffW\right\vert\leq A_0\frac{h}{x}e^{\gamma x}.$$
\item\label{diffusion_zero_finite_variation:ii} For the L\'evy triplet $(0,\delta_{-1},1)$ and the nested sequence $(h_n:=1/2^n)_{n\geq 1}\downarrow 0$, for each $q\geq 0$ and any $x\in \cup_{n\geq 1}\mathbb{Z}_{h_n}^{+}\cap [0,1)$: $$\lim_{n\to\infty}\frac{\diffWn}{h_n}=e^{x(1+q)}\frac{1}{2}(1+q)^2x.$$
\end{enumerate}

\end{proposition}
\begin{proof}
With respect to~\ref{diffusion_zero_finite_variation:i}, we estimate the three terms appearing on the right-hand side of \eqref{eq:difference:Wqs} one by one. 

First,~\ref{W:a} in \eqref{eq:difference:Wqs} is easily seen to be of order $\frac{1}{x}O(h)$ by an obvious integration by parts argument, using coercivity and the fact that $A'$ is bounded. 

Second, when it comes to~\ref{W:b} in \eqref{eq:difference:Wqs}, an integration by parts is also performed immediately: $$\frac{d}{ds}\left(e^{isx}\left(\frac{1}{A(s)}-\frac{1}{A^h(s)}\right)\right)=ix e^{isx}\left(\frac{1}{A(s)}-\frac{1}{A^h(s)}\right)+e^{isx}\left(-\frac{A'(s)}{A(s)^2}+\frac{A^{h\prime}(s)}{A^h(s)^2}\right).$$ Upon integration on $[-\pi/h,\pi/h]$, by coercivity, the left-hand side is of order $O(h)$ and hence will contribute $\frac{1}{x}O(h)$ to the right-hand side of \eqref{eq:difference:Wqs}. Write: 
\begin{equation}\label{eq:finite_variation:zero}
-\frac{A'(s)}{A(s)^2}+\frac{A^{h\prime}(s)}{A^h(s)^2}=\underbrace{\frac{A^{h\prime}(s)-A'(s)}{A(s)^2}}_{\mytag{(1)}{FV:(1)}}+\underbrace{A^{h\prime}(s)\frac{(A(s)-A^h(s))(A(s)+A^h(s))}{A(s)^2A^h(s)^2}}_{\mytag{(2)}{FV:(2)}}.
\end{equation}

We focus on each term one at a time. For \ref{FV:(1)}, there corresponds to it, modulo non-zero multiplicative constants, and by Fubini (using \eqref{eq:FV:derivative_of_A} and \eqref{eq:FV:derivative_of_Ah}): 
\begin{eqnarray}\label{eq:finite_variation:one}
\nonumber&\phantom{=}&(\mu-\mu^h)\int_{[-\pi/h,\pi/h]}\frac{e^{\beta h}-1}{A(s)^2}e^{isx}ds+c_0^h\int_{[-\pi/h,\pi/h]}\left(\frac{e^{\beta h}-e^{-\beta h}}{2h}\right)\frac{1}{A(s)^2} e^{isx}ds\\
&+&\sum_{y\in \Zh^-}\int_{A_y^h}\measure(dz)\int_{[-\pi/h,\pi/h]}\frac{y(e^{\beta y}-\mathbbm{1}_{[-V,0)}(z))-z(e^{\beta z}-\mathbbm{1}_{[-V,0)}(z))}{A(s)^2}e^{isx}ds.
\end{eqnarray} 
There are three summands in \eqref{eq:finite_variation:one}. The first is $O(h)$ by employing the decomposition $e^{\beta h}-1=(e^{\beta h}-\beta h-1)+\beta h$ and then using \ref{lemma:technical:estimates:iii} of Lemma~\ref{lemma:technical:estimates} and coercivity \ref{(C)} (respectively \eqref{lemma:FV:one} of Lemma~\ref{lemma:FV}) for the first (respectively second) term. The same is true of the second summand, noting that, for sure, $c_0^h=O(h)$, employing the decomposition $\frac{e^{\beta h}-e^{-\beta h}}{2h}=\left(\frac{e^{\beta h}-e^{-\beta h}}{2h}-\beta\right)+\beta$ and then using \ref{lemma:technical:estimates:rough:II} of Lemma~\ref{lemma:technical:estimates} and again coercivity \ref{(C)} (respectively \eqref{lemma:FV:one} of Lemma~\ref{lemma:FV}) for the first (respectively second) term.  As for the third summand of \eqref{eq:finite_variation:one}, write when $z\notin [-V,0)$: $$ye^{\beta y}-ze^{\beta z}=\overbrace{(y-z)e^{\beta y}}^{\mytag{(I)}{FV:I}}+ze^{\beta z}(\overbrace{e^{\beta(y-z)}-\beta(y-z)-1}^{\mytag{(II)}{FV:II}}+\overbrace{\beta(y-z)}^{\mytag{(III)}{FV:III}}).$$ By the findings of Lemma~\ref{lemma:technical:estimates}, the fact that $\measure(-\infty,-V)<\infty$ and coercivity \ref{(C)}, it is clear that \ref{FV:I} and \ref{FV:II} will contribute a term of order $\frac{1}{x}O(h)$ to the right-hand side of \eqref{eq:difference:Wqs}. On the other hand, for \ref{FV:III}, the same follows by \eqref{lemma:FV:one} of Lemma~\ref{lemma:FV}. Next, when $z\in [-V,0)$ we write: 
\begin{equation}
ye^{\beta y}-ze^{\beta z}-(y-z)=\overbrace{e^{\beta z}y(e^{\beta (y-z)}-1)}^{\mytag{(I)}{FV_I}}+\overbrace{(y-z)(e^{\beta z}-1)}^{\mytag{(II)}{FV_II}}.
\end{equation}
When it comes to \ref{FV_I}, it is dealt with precisely as it was for $z\notin [-V,0)$ (but note also that $\vert y\vert\leq 2\vert z\vert$). With regard to \ref{FV_II}, apply \eqref{lemma:FV:three} of Lemma~\ref{lemma:FV}. 

To handle \ref{FV:(2)} of \eqref{eq:finite_variation:zero}, i.e.: $$A^{h\prime}(s)\frac{(A(s)-A^h(s))(A(s)+A^h(s))}{A(s)^2A^h(s)^2},$$ notice that in the difference to replacing this with $$2A'(s)\frac{(A(s)-A^h(s))}{A(s)^3}$$ a term of $\frac{1}{x}O(h)$ is contributed to \eqref{eq:difference:Wqs} (just make successive replacements $A^h\to A$ and study the difference by taking advantage of coercivity \ref{(C)}, \ref{delta2} and \ref{delta2D}. So it is in fact sufficient to study: $$\int_{[-\pi/h,\pi/h]}e^{isx}A'(s)\frac{(A(s)-A^h(s))}{A(s)^3}ds.$$ Now we do first Fubini for $A'$ (via Remark~\ref{remark:A'forFV}), and get, beyond a factor of $i$: $$\mu_0 \int_{[-\pi/h,\pi/h]}e^{isx}\frac{(A(s)-A^h(s))}{A(s)^3}ds+\int \measure(dy)ye^{\gamma y}\int_{[-\pi/h,\pi/h]}e^{is(x+y)}\frac{(A(s)-A^h(s))}{A(s)^3}ds.$$ So what we would really like, is to show that:
\footnotesize
\begin{eqnarray}
\nonumber&& \int_{[-\pi/h,\pi/h]}e^{isz}\frac{(A^h(s)-A(s))}{A(s)^3}ds\\\nonumber
&=&(\mu-\mu^h)\int_{[-\pi/h,\pi/h]}e^{isz}\frac{\frac{e^{\beta h}-1}{h}-\beta}{A(s)^3}ds+c_0^h\int_{[-\pi/h,\pi/h]}e^{isz}\left(\frac{e^{\beta h}+e^{-\beta h}-2}{2h^2}\right)\frac{1}{A(s)^3}ds\\
&+&\sum_{y\in \Zh^-}\int_{A_y^h}\measure(du)\int_{[-\pi/h,\pi/h]}e^{isz}\frac{e^{\beta y}-\beta y\mathbbm{1}_{[-V,0)}(u)-(e^{\beta u}-\beta u\mathbbm{1}_{[-V,0)}(u))}{A(s)^3}ds \label{eq:FV:Wqs:another_one}
\end{eqnarray}
\normalsize
is bounded by a (constant times $h$), uniformly in $z\in\mathbb{R}$ (it will then follow immediately that a term of $\frac{1}{x}O(h)$ is being contributed to \eqref{eq:difference:Wqs}). 
\begin{itemize}
\item Now, the part corresponding to $\mathbbm{1}_{(-\infty,-V)}\cdot \measure$, namely: $$\sum_{y\in \Zh^-}\int_{A_y^h\cap (-\infty,-V)}\measure(du)\int_{[-\pi/h,\pi/h]}e^{isz}\frac{e^{\beta y}-e^{\beta u}}{A(s)^3}ds,$$ is clearly so. 
\item With respect to the term involving $c_0^h=O(h)$, make the decomposition: $$\frac{e^{\beta h}+e^{-\beta h}-2}{2h^2}=\left(\frac{e^{\beta h}+e^{-\beta h}-2}{2h^2}-\frac{\beta^2}{2}\right)+\frac{\beta^2}{2}.$$ Then use coercivity \ref{(C)} and \ref{lemma:technical:estimates:rough:I} of Lemma~\ref{lemma:technical:estimates} (respectively \eqref{lemma:FV:one} of Lemma~\ref{lemma:FV}) for the first (respectively second) term. 
\item As regards: $$(\mu-\mu^h)\int_{[-\pi/h,\pi/h]}e^{isz}\frac{\frac{1}{h}(e^{\beta h}-1)-\beta}{A(s)^3}ds$$ write:
\begin{equation}
\frac{1}{h}\left(e^{\beta h}-1\right)-\beta=\frac{1}{h}\left(\overbrace{e^{\beta h}-\frac{\beta^2 h^2}{2}-\beta h-1}^{\mytag{(I)}{FVI}}+\overbrace{\frac{\beta^2 h^2}{2}}^{\mytag{(II)}{FVII}}\right).
\end{equation}
By an expansion into a series, which converges absolutely and locally uniformly, and coercivity \ref{(C)}, it is clear that \ref{FVI} has the desired property, whereas \eqref{lemma:FV:one} of Lemma~\ref{lemma:FV} may be applied to \ref{FVII}. 
\item Finally it will be sufficient to consider: $$\sum_{y\in\Zh^{-}}\int_{A_y^h\cap [-V,0)}\measure(du)\int_{[-\pi/h,\pi/h]}e^{isz}\frac{e^{\beta u}-\beta u-(e^{\beta  y}-\beta y)}{A(s)^3}ds,$$ which we need bounded by a (constant times $h$) uniformly in $z\in \RR$. For this to work it is sufficient that the innermost integral produces ($\vert u\vert$ times $h$ times a constant). Moreover, it is enough to produce $\vert y\vert$ (or, a fortiori, $\vert y-u\vert$) in place of $\vert u\vert$. Now write: 
%\footnotesize
%\begin{equation*}
%e^{\beta u}-\beta u-(e^{\beta  y}-\beta y)=(e^{\beta y}-1)((e^{\beta (u-y)}-\beta(u-y)-1)+\beta(u-y))+\left(e^{\beta(u-y)}-\frac{\beta^2(u-y)^2}{2}-\beta(u-y)-1\right)+\frac{\beta^2(u-y)^2}{2}.
%\end{equation*}
%\normalsize
\begin{eqnarray*}
e^{\beta u}-\beta u-(e^{\beta  y}-\beta y)&=&(e^{\beta y}-1)\left((e^{\beta (u-y)}-\beta(u-y)-1)+\beta(u-y)\right)+\\
&&\left(e^{\beta(u-y)}-\frac{\beta^2(u-y)^2}{2}-\beta(u-y)-1\right)+\frac{\beta^2(u-y)^2}{2}.
\end{eqnarray*}
These terms can now be dealt with in part straightforwardly and in part by employing \eqref{lemma:FV:three} and \eqref{lemma:FV:one} of Lemma~\ref{lemma:FV}.
\end{itemize}

Third, with respect to~\ref{W:c} of \eqref{eq:difference:Wqs}, again an integration by parts is made outright, thus: 
$$\frac{d}{ds}\left(e^{isx}\frac{1-e^{ish}}{A^h(s)}\right)=ixe^{isx}\frac{1-e^{ish}}{A^h(s)}+e^{isx}\frac{-ihe^{ish}}{A^h(s)}-e^{isx}\frac{(1-e^{ish})A^{h\prime}(s)}{A^h(s)^2}.$$ Now the left-hand side is handled using coercivity \ref{(C)}. On the right-hand side, we apply \eqref{lemma:FV:one} of Lemma~\ref{lemma:FV} to the second term. Finally, in the third term on the right-hand side we may replace $A^{h\prime}(s)$ by $A'(s)$, followed by Fubini for $A'$ and an application of \eqref{lemma:FV:three} of Lemma~\ref{lemma:FV}. All in all, a term of order $\frac{1}{x}O(h)$ thus emerges on the right-hand side of \eqref{eq:difference:Wqs}. 

Part~\ref{diffusion_zero_finite_variation:ii} can be obtained by explicit computation, and is elementary. 
\end{proof}

\begin{proposition}[$\diffusion=0$, $\kappa(0)<\infty$  ($\Zq$ convergence)]
Suppose $\diffusion=0$ and $\kappa(0)<\infty$. Let $q>0$. 

\begin{enumerate}[(i)]
\item\label{diffusion_zero_finite_variation:Zqs:i} For any $\gamma>\Phi(q)$, there are $\{A_0,h_0\}\subset (0,\infty)$ such that for all $h\in (0,h_0)$ and then all $x\in \Zh^{++}$: $$\left\vert \diffZ\right\vert\leq A_0he^{\gamma x}.$$
\item\label{diffusion_zero_finite_variation:Zqs:ii} For the L\'evy triplet $(0,\delta_{-1},1)$ and the nested sequence $(h_n:=1/2^n)_{n\geq 1}\downarrow 0$, for each $q> 0$ and any $x\in \cup_{n\geq 1}\mathbb{Z}_{h_n}^{++}\cap (0,1)$: $$\lim_{n\to\infty}\frac{\diffZn}{h_n}=\frac{1}{2}q(1+q)xe^{x(1+q)}. $$
\end{enumerate}

\end{proposition}
\begin{proof} 
With respect to~\ref{diffusion_zero_finite_variation:Zqs:i}, we have as follows. First, \ref{Z:a} of \eqref{eq:difference:Zqs} is of order $O(h)$ by coercivity. Second, in \ref{Z:b}, we employ the decomposition: $$1-\frac{\beta h}{1-e^{-\beta h}}=\left(1-\frac{\beta h}{1-e^{-\beta h}}+\frac{\beta h}{2}\right)-\frac{\beta h}{2}.$$ Then the first term may be estimated via \ref{lemma:technical:coercivity:b} of Lemma~\ref{lemma:technical:estimates} (for the denominator), a Taylor expansions into absolutely and locally uniformly convergent series (for the numerator) and coercivity \ref{(C)}; while to the second term we apply \eqref{lemma:FV:one} of Lemma~\ref{lemma:FV}. It follows that \ref{Z:b} of \eqref{eq:difference:Zqs} is $O(h)$. Finally, when it comes to \ref{Z:c} of \eqref{eq:difference:Zqs}, we have (beyond a non-zero multiplicative constant): 
\begin{equation*}
\int_{[-\pi/h,\pi/h]}e^{isx}\frac{(A(s)-A^h(s))}{A(s)A^h(s)\beta}ds.
\end{equation*}
This can now be seen to be $O(h)$ in the same manner as \eqref{eq:FV:Wqs:another_one} was seen to be so (indeed, one can simply follow, word-for-word, the argument pursuant to \eqref{eq:FV:Wqs:another_one}, and recognize that the substitution of $A(s)A^h(s)\beta$ in place of $A(s)^3$ results in no material change). 

Part~\ref{diffusion_zero_finite_variation:Zqs:ii} follows by a direct computation. 
\end{proof}

\subsection{$\diffusion=0$, infinite variation paths}\label{subsection:sclaes:diffusion0infvariation}\label{subsection:scales:diffusion=0:infinite_variation}
Finally we consider the case when $\diffusion=0$ and $\kappa(0)=\infty$. We assume here that Assumption~\ref{assumption:salient} is in effect. Note also that $\delta^0=1$ and we work under scheme 2. We do not establish sharpness of the rates. 

\begin{proposition}[$\diffusion=0$ \& $\kappa(0)=\infty$]\label{proposition:convergence:diffusion_zero_infinite_variation}
Assume $\diffusion=0$ and Assumption~\ref{assumption:salient}, let $q\geq 0$, $\gamma>\Phi(q)$. Then there are $\{A_0,h_0\}\subset (0,\infty)$ such that for all $h\in (0,h_0)$ and then all $x\in \Zh^{++}$: 
\begin{enumerate}[(i)]
\item\label{diffusion_zero_infinite_variation:one} $\left\vert\diffW\right\vert\leq A_0\frac{h^{2-\epsilon}}{x}e^{\gamma x}$.
\item\label{diffusion_zero_infinite_variation:two} $\left\vert \diffZ\right\vert\leq A_0 h^{2-\epsilon}e^{\gamma x}.$
\end{enumerate}
\end{proposition}
\begin{proof}
%We prepare the derivatives:
%\begin{eqnarray*}
%A'(s)&=&i\drift+i\int_{(-\infty,0)}y(e^{\beta y}-\mathbbm{1}_{[-1,0)}(y))\measure(dy),\\
%A^{h\prime}(s)&=&\frac{ic_0^h}{2h}\left(e^{\beta h}-e^{-\beta h}\right)+i(\mu-\mu^h)e^{\beta h}+i\sum_{y\in \Zh^{--}}c_y^hye^{\beta y}.
%\end{eqnarray*}
%Note that $A'(s)=O(\vert s\vert^{\epsilon-1})$ as $\vert s\vert\to\infty$ by \ref{assumption:salient:one} of Assumption~\ref{assumption:salient} (and Proposition~\ref{proposition:generally_on_measure_asymptotics_bis}). It is also easily seen that there is an $A_0\in (0,\infty)$, such that for all $h\in (0,h_\star\land 2)$ and then all $s\in [-\pi/h,\pi/h]$: $$\vert A'(s)-A^{h\prime}(s)\vert\leq A_0\left[h\kappa(h/2)+\xi(h/2)\right]\vert \beta\vert.$$ Indeed, we have the following decomposition of $A^{h\prime}(s)-A'(s)$ into three summands, each of which is then easily estimated:
%\begin{enumerate}[(1)]
%\item $\sum_{y\in \Zh^{--}}\int_{A_y^h}\left[y(e^{\beta y}-\indd(z))-z(e^{\beta z}-\indd(z))\right]\measure(dz)$;
%\item $i(\mu-\mu^h)(e^{\beta h}-1)$;
%\item $ic_0^h\left[\frac{e^{\beta h}-e^{-\beta h}}{2h}\right]-i\int_{[-h/2,0)}y(e^{\beta y}-1)\measure(dy)$. 
%\end{enumerate}
%From these observations it follows, furthermore, that $\vert A^{h\prime}(s)\vert \leq A_0\vert\beta\vert^{\epsilon-1}$ for all $s\in [-\pi/h,\pi/h]$, for all $h\in (0,h_\star\land 2)$, for some $A_0\in (0,\infty)$ (where, additionally, one takes into account $\xi(h/2)=O(h^{2-\epsilon})$ and $\zeta(h/2)=O(h^{2-\epsilon})$).

With respect to~\ref{diffusion_zero_infinite_variation:one}, we have as follows. Note that~\ref{W:a} in \eqref{eq:difference:Wqs} is $\frac{1}{x}O(h^\epsilon)$ by an integration by parts argument and \ref{delta1D} from Subsection~\ref{subsection:_the_difference_of_derivatives}. We do the same for~\ref{W:b}: 
\begin{eqnarray*}
&&\frac{d}{ds}\left(e^{isx}\left(\frac{1}{A(s)}-\frac{1}{A^h(s)}\right)\right)=ixe^{isx}\left(\frac{1}{A(s)}-\frac{1}{A^h(s)}\right)+\\
&+&e^{isx}\left(\frac{A^2(s)(A^{h\prime}(s)-A'(s))+A'(s)(A(s)-A^h(s))(A(s)+A^h(s))}{A^2(s)A^h(s)^2}\right).
\end{eqnarray*}
Upon integration, one gets on the left-hand side a contribution of $\frac{1}{x}O(h^\epsilon)$ to \eqref{eq:difference:Wqs}, by coercivity \ref{(C)}. With regard to the rightmost quotient on the right-hand side, we obtain a contribution of order $\frac{1}{x}O(h^{2-\epsilon})$, as follows again by coercivity \ref{(C)}, Remark~\ref{remark:under_assumption_salient_etc}, \ref{delta1D} and the estimates \ref{delta2} and \ref{delta2D}. Remark that $\epsilon> 2-\epsilon$. 

With respect to~\ref{diffusion_zero_infinite_variation:two}, we have as follows. \ref{Z:a} of \eqref{eq:difference:Zqs} is of order $O(h^\epsilon)$ by coercivity \ref{(C)}. Also, \ref{Z:b} is of order $O(h)$. Finally~\ref{Z:c} is of order $O(h^{2-\epsilon})$ immediately, with no need for an integration by parts. 
\end{proof}

\subsection{A convergence result for the derivatives of ${\Wq}$ ($\diffusion>0$)}\label{subsection:functionals_of_scale_fncs_conv} 

\begin{proposition}
Let $q\geq 0$, $\diffusion>0$. Note that $\Wq$ is then differentiable on $(0,\infty)$ \cite[Lemma 2.4]{kuznetsovkyprianourivero}. Moreover, for any $\gamma>\Phi(q)$, there exist $\{A_0,h_0\}\subset (0,\infty)$, such that for all $x\in \Zh^{++}\backslash \{h\}$:
\begin{equation*}
\left\vert\Wqprime(x)-\frac{\Wq_h(x)-\Wq_h(x-2h)}{2h}\right\vert\leq  A_0\frac{e^{\gamma x}}{x}\left(h+\zeta(h/2)+\xi(h/2)h\log(1/h)\right).
\end{equation*}
\end{proposition}
\begin{remark}
The case $\diffusion=0$ appears much more difficult to analyze, since the balance between coercivity and the estimates of the differences of Laplace exponents (and their derivatives) worsens. 
\end{remark}
\begin{proof}
%First, $\widehat{d\Wq}=\widehat{\Wqprime}$, since $\Wq(0)=0$ \cite[p. 222, Lemma 8.6]{kyprianou} (use, for example, the Fundamental Theorem of Calculus, a Radon-Nikodym change of measure and monotone convergence), whereas 
First, integration by parts, monotone convergence and the fact that $\Wq(0)=0$ \cite[p. 222, Lemma 8.6]{kyprianou} yield (for $\beta>\Phi(q)$): $\widehat{\Wqprime}(\beta)=\beta\widehat{\Wq}(\beta)=\beta/(\psi(\beta)-q)$. Then analytic continuation, Laplace inversion and dominated convergence allow to conclude, for any $x>0$ that: $$\Wqprime(x)=\frac{1}{2\pi}\int_{-\infty}^\infty\frac{\beta e^{\beta x}}{\psi(\beta)-q}ds.$$ On the other hand, it follows directly from Corollary~\ref{corollary:scale_fncs_laplace-transforms}, that for $h\in (0,h_\star)$ and then $x\in \Zh^{++ }$: $$\frac{\Wq_h(x)-\Wq(x-2h)}{2h}=\frac{1}{2\pi}\int_{-\pi/h}^{\pi/h}e^{\beta x}\left(\frac{e^{\beta h}-e^{-\beta h}}{2h}\right)\frac{ds}{\psi^h(\beta)-q}.$$  

%Next, we prepare the following derivatives. For $s\in \mathbb{R}$: 
%\begin{equation*}
%-iA'(s)=\diffusion \beta+\drift+\int_{(-\infty,0)}z\left(e^{\beta z}-\ind(z)\right)\measure(dz)
%\end{equation*}
%and further for $h\in (0,h_\star)$ and then $s\in [-\pi/h,\pi/h]$:
%\footnotesize
%\begin{equation*}
%-iA^{h\prime}(s)=\drift\frac{e^{\beta h}+e^{-\beta h}}{2}-\drift^h\left(\frac{e^{\beta h}+e^{-\beta h}}{2}-1\right)+(\diffusion+c_0^h)\frac{e^{\beta h}-e^{-\beta h}}{2h}+\sum_{y\in \Zh^{--}}\int_{A_y^h}y\left(e^{\beta y}-\ind(z)\right)\measure(dz).
%\end{equation*}
%\normalsize
%From these expressions it follows readily, using established techniques, that $A'$ and $A^{h\prime}$ are both bounded in linear growth, uniformly in $h\in (0,h_\star)$. Moreover there is an $A_0\in (0,\infty)$ such that for all $h\in (0,h_\star\land 2)$ and then all $s\in [-\pi/h,\pi/h]$, $\vert A^{h\prime}(s)-A'(s)\vert\leq A_0(h^2\vert \beta\vert^3+h\xi(h/2)\vert \beta\vert^2+(h+\zeta(h/2))\vert \beta\vert)$. 

Now, it will be sufficient to estimate the following integrals:
\footnotesize
\begin{equation*}
\int_{(-\infty,\infty)\backslash [-\pi/h,\pi/h]}\beta e^{\beta x}A(s)^{-1}ds;\ \int_{-\pi/h}^{\pi/h}e^{\beta x}\left(\frac{e^{\beta h}-e^{-\beta h}}{2h}-\beta\right)A^h(s)^{-1}ds; \ \int_{-\pi/h}^{\pi/h}\beta e^{\beta x}\left(A(s)^{-1}-A^h(s)^{-1}\right)ds.
\end{equation*}
\normalsize
An integration by parts, coupled with coercivity \ref{(C)} and the boundedness in linear growth of $A'$ \ref{delta1D}, establishes the first of these two integrals as being of order $\frac{1}{x}O(h)$. The same emerges as being true of the second integral, this time using the boundedness in linear growth of $A^{h\prime}$ instead, but also \ref{lemma:technical:estimates:iv} and \ref{lemma:technical:estimates:rough:II} of Lemma~\ref{lemma:technical:estimates}. Finally, with respect to the third integral, again one performs integration by parts, and then uses \ref{delta2}, coercivity \ref{(C)}, the decomposition \eqref{eq:finite_variation:zero} and \ref{delta2D}. The claim follows. 
\end{proof}

\section{Numerical illustrations and concluding remarks}\label{section:numerical_illustrations} 
\subsection{Numerical examples}
We illustrate our algorithm for computing $\sc$, described in Eq.~\eqref{eq:LinRecursion} of the Introduction, in two concrete examples, applying it to determine some relevant quantities arising in applied probability. The examples are chosen with two criteria in mind:
\begin{enumerate}
\item They are natural from the modeling perspective (computation of (Example~\ref{example:log-normal}) ruin parameters in the classical Cram\'er-Lundberg model with log-normal jumps  and (Example~\ref{example:main}) the L\'evy-Kintchine triplet of the limit law of a CBI process). 
\item They do not posses a closed form formula 
%(in terms of elementary/special functions) 
for the Laplace exponent of the spectrally negative L\'evy processes. Such examples arise often in practice, making it difficult to apply the standard algorithms for scale functions based on Laplace inversion. Our algorithm is well-suited for such applications.
\end{enumerate}
%We next present two concrete examples of the computation of the scale function $\sc$, with applications. The first is done for the log-normal Cram\'er-Lundberg surplus process, in the context of Gerber-Shiu risk theory. The second example, in the context of CBI processes, is for an instance of a spectrally negative L\'evy process, whose L\'evy measure is sufficiently intricate as to convince the reader of the general applicability of our method. 

\begin{example}\label{example:log-normal}
A popular choice for the claim-size modeling in the Cram\'er-Lundberg surplus process is the log-normal distribution \cite[Paragraph I.2.b, Example 2.8]{asmussen}. Fixing the values of the various parameters,  consider the spectrally negative L\'evy process $X$ having $\diffusion=0$; $\measure(dy)=\mathbbm{1}_{(-\infty,0)}(y)\exp(-(\log(-y))^2/2)/(\sqrt{2\pi}(-y))dy$; and (with $V=0$) $\mu=5$ (this satisfies the security loading condition \cite[Section~1.2]{gerber_shiu}). Remark that the log-normal density has fat tails and is not completely monotone.
%; moreover, it does not admit an explicit Laplace transform. 

We complement the computation of $\sc$ by applying it to the calculation of the density $k$ of the deficit at ruin, on the event that $X$ goes strictly above the level $a=5$, before venturing strictly below $0$, conditioned on $X_0=x=2$: $$\EE_x[-X_{\tau_0^-}\in dy,\tau_0^-<\tau_a^+]=k(y)dy$$ ($\tau_0^-$, respectively $\tau_a^+$, being the first entrance time of $X$ to $(-\infty,0)$, respectively $(a,\infty)$). Indeed, $k(y)$ may be expressed as \cite[Theorem 5.5]{gerber_shiu} $k(y)=\int_0^af(z+y)\frac{\sc(x)\sc(a-z)-\sc(a)\sc(x-z)}{\sc(a)}dz$, where $f(y):=\exp(-(\log(y))^2/2)/(\sqrt{2\pi}y)$, $y\in (0,+\infty)$. We approximate the integral $k$ by the discrete sum $k_h$, given for $y\in (0,\infty)$ as follows: \footnotesize $$k_h(y):=h\left[f(y+a)\frac{\sc_h(x)\sc_h(0)}{2\sc_h(a)}+\sum_{k=1}^{a/h-1}f(kh+y)\sc_h(a-kh)\frac{\sc_h(x)}{\sc_h(a)}-\sum_{k=1}^{x/h-1}\sc_h(x-kh)f(kh+y)-\frac{\sc_h(0)f(x+y)}{2}\right].$$ \normalsize  Results are reported in Figure~\ref{figure:log-normal}. 

\begin{figure}
                \includegraphics[width=\textwidth]{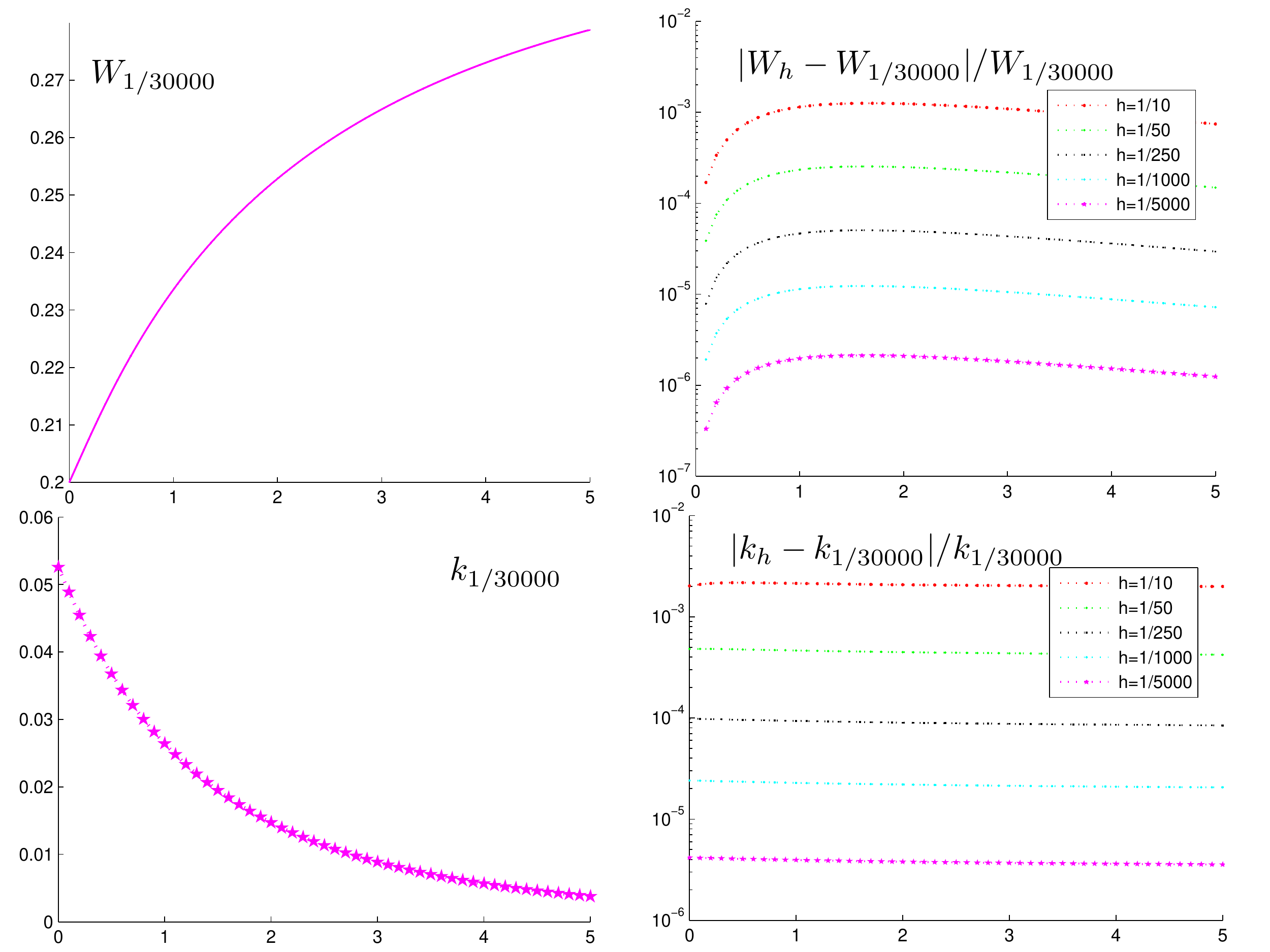}
\caption{The scale function $\sc$ and the density $k$ of the deficit at ruin on the event $\{\tau_0^-<\tau_5^+\}$, for the log-normal Cram\'er-Lundberg process. The relative errors are consistent with the linear order of convergence predicted by Theorem~\ref{theorem:rates_for_scale_functions}. See Example~\ref{example:log-normal} for details.}\label{figure:log-normal}
\end{figure}

\end{example}

\begin{example}\label{example:main}
We take $\diffusion=0$; the L\'evy measure $\measure=\measure_a+\measure_c$ has atomic part $\measure_a=\frac{1}{2}\left(\delta_{-1}+\delta_{-2}\right)$, whilst the density of its absolutely continuous part $\measure_c(dy)=l(y)dy$ is given by: \footnotesize$$l(y)=\frac{3}{2(-y)^{5/2}}\mathbbm{1}_{[-1,0)}(y)+\frac{1}{2(-y-1)^{1/2}}\mathbbm{1}_{[-2,-1)}(y)+\left(\frac{e^{\cos(y)}(3+y\sin(y))}{(-y)^4}+\frac{e}{(-y)^3}\right)\mathbbm{1}_{(-\infty,-1)}(y),\ y\in \mathbb{R};$$ \normalsize and (with $V=1$) $\mu=15$. Remark %that in this instance (i) one does not obtain an explicit expression in terms of elementary/special functions for the Laplace/Furrier transform of the L\'evy measure, and thus not the Laplace exponent, and (ii) 
the case is extreme: there are two atoms, while the density is stable-like at $0$, has a fat tail at $-\infty$, and a discontinuity (indeed, a pole; in particular, it is not completely monotone). Furthermore, there is no Gaussian component, and the sample paths of the process have infinite variation. 

We compute $\sc$ for the L\'evy process $X$ having the above characteristic triplet, and complement this with the following application. Let furthermore $X^F$ be an independent L\'evy subordinator, given by $X^F_t=t+Z_t$, $t\in [0,\infty)$, where $Z$ is a compound Poisson process with L\'evy measure $m(dy)=e^{-y}\mathbbm{1}_{(0,\infty)}(y)dy$. Denote the dual $-X$ of $X$ by $X^R$. To the pair $(X^F,X^R)$ there is associated, in a canonical way, (the law of) a (conservative) CBI process \cite{kawazu_watanabe}. The latter process converges to a limit distribution $L$, as time goes to infinity, since $\psi'(0+)>0$ \cite[Theorem~2.6(c)]{keller_mijatovic} and since further the log-moment of $m$ away from zero is finite \cite[Corollary~2.8]{keller_mijatovic}. Moreover, the limit $L$ is infinitely divisible, and: $$-\log\int_0^\infty e^{-ux}dL(x)=u\gamma-\int_{(0,\infty)}(e^{-ux}-1)\frac{k(x)}{x}dx,$$ where $\gamma=b\sc(0)$ vanishes, whilst: $$k(x)=b\sc'(x+)+\int_{(0,\infty)}[\sc(x)-\sc(x-\xi)]m(d\xi),$$ where $b=1$ and $m$ are the drift, respectively the L\'evy measure, of $X^F$ \cite[Theorem~3.1]{keller_mijatovic}. We compute $k$ via approximating, for $x\in \Zh^{++}$, $k(x)$ by $k_h(x)$: $$k_h(x):=b\frac{\sc_h(x)-\sc_h(x-h)}{h}+\sc_h(x-h)m[h/2,\infty)-\sum_{k=1}^{x/h-1}\sc_h(x-kh-h)m[kh-h/2,kh+h/2).$$ Results are reported in Figure~\ref{figure:example:main}.
\end{example}

\begin{figure}
                \includegraphics[width=\textwidth]{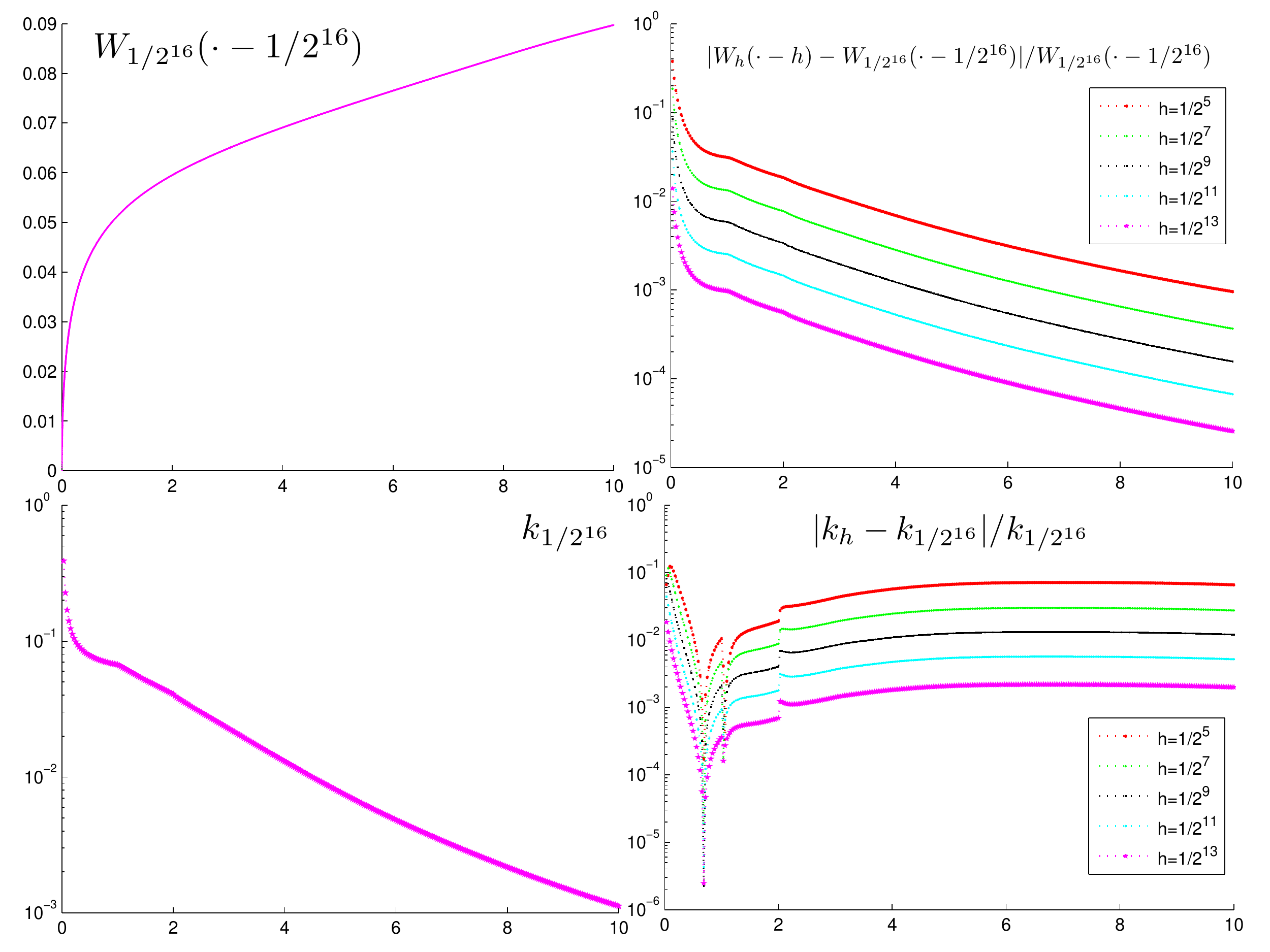}
        \caption{The scale function $\sc$ for the spectrally negative L\'evy process $X$, as described in Example~\ref{example:main}; followed by the $k$-function $k$ of the weak limit (as time goes to infinity) of the CBI process, whose spectrally positive component is the dual $X^R=-X$ of $X$, whilst the L\'evy subordinator part is the sum of a unit drift and a compound Poisson process of unit intensity and mean one exponential jumps.  The relative errors are consistent with the $O(\sqrt{h})$ order of convergence predicted by Theorem~\ref{theorem:rates_for_scale_functions}. See Example~\ref{example:main} for details.
%In Example~\ref{example:main}, convergence is considered on the decreasing sequence $(h_n=1/2^n)_{n\in \{5,7,9,11,13\}}$ and on the interval $[0,3]$. The corresponding number of computed values of $W_{h_n}$ is $N_n=3\cdot 2^n$ in each case. Those got using $h_{16}$ are taken as a benchmark against which the others are compared. Plotted are the successive approximations to $\sc$ and the relative error $\frac{\vert W_{h}(\cdot-h)-W_{h_{16}}(\cdot-h_{16})\vert}{W_{h_{16}}(\cdot-h_{16})}$ on the grid $x_i=i/32$, $i\in [96]$. Note that the latter is well-behaved, and falls by about a factor of $1/2=(1/4)^{2-3/2}$ in each consecutive computation, as it should.
}\label{figure:example:main}
\end{figure}

%\begin{figure}[!htb]
%\includegraphics[width=0.76\textwidth]{Convergence_W_furtherr_relative.pdf}
%\caption{In Example~\ref{example:main}, convergence is considered on the decreasing sequence $(h_n=1/2^n)_{n\in \{5,7,9,11,13\}}$ and on the interval $[0,3]$. The corresponding number of computed values of $W_{h_n}$ is $N_n=3\cdot 2^n$ in each case. Those got using $h_{16}$ are taken as a benchmark against which the others are compared, and plotted is the relative error $\frac{\vert W_{h}(\cdot-h)-W_{h_{16}}(\cdot-h_{16})\vert}{W_{h_{16}}(\cdot-h_{16})}$ on the grid $x_i=i/32$, $i\in [96]$. Note that it is well-behaved, and falls by about a factor of $1/2=(1/4)^{2-3/2}$ in each consecutive computation, as it should.}
%\label{figure:example:main}
%\end{figure}

Let us also mention that we have tested our algorithm on %various other examples,
%(see the arXiv version of this paper \cite[Section~6]{vidmar:scales}),
%in particular those from  \cite[p. 178]{kuznetsovkyprianourivero}; as well as  on 
very simple processes with completely monotone L\'evy densities \cite{albrecher} (Brownian motion with drift; positive drift minus a compound Poisson subordinator with exponential jumps; spectrally negative stable L\'evy process --- see Appendix~\ref{appendix:further_examples}), and the results were in nice agreement with the explicit formulae which are available for the scale functions in the latter cases.

%As already noted in the Introduction, 
%Corollary~\ref{proposition:calculating_scale_functions} %, Equations~\eqref{equation:recursion:Wq} and~\eqref{equation:recursion:Zq} (applied to $Y=X^h/h$, $h\in (0,h_\star)$),
\subsection{Concluding remarks}
%Some remarks on the computational nature of our algorithm, that have not already been made explicit in the Introduction, now follow. 

 %Recall that our algorithm for computing the scale functions of the approximating chains $(X^h)_{h\in (0,h_\star)}$ is provided by Corollary~\ref{proposition:calculating_scale_functions}. To compute $\Wq_h(x)$ or $\Zq_h(x)$ for some $x\in \Zh$ one effects recursions \eqref{equation:recursion:Wq} and \eqref{equation:recursion:Zq} (as applied to $Y=X^h/h$, $h\in (0,h_\star)$), at a cost of $O((x/h)^2)$ operations (assuming given the parameters of $X^h$). 

%Of course, parameter $h$ is chosen so small as to attain a desired level of accuracy, whilst $x/h$ remains an integer.

%As a by-product of such an evaluation, the values of $\Wq_h$ or $\Zq_h$ on the entire lattice
%$[0,x]\cap \Zh$, rather than just in the single point $x$, necessarily and automatically, obtain. Matlab code implementing this algorithm is given in~\cite{Scale_Function_Code}. 

%\begin{remark}\label{remark:numerics}
%\leavevmode
%\begin{enumerate}
%\item 
(1) Computational cost. To compute $\Wq_h(x)$ or $\Zq_h(x)$ for some $x\in \Zh$ one effects recursions \eqref{equation:recursion:Wq} and \eqref{equation:recursion:Zq} (as applied to $Y=X^h/h$, $h\in (0,h_\star)$), at a cost of $O((x/h)^2)$ operations (assuming given the parameters of $X^h$). 

\noindent 
%\item Recursions \eqref{equation:recursion:Wq} and \eqref{equation:recursion:Zq} are stable as they consists of summing only nonnegative terms \cite[Theorem~7]{panjer}.
%\item 
(2) Quantity $\Zq_h(x)$ may be obtained from the values of $\Wq_h$ on
$[0,x]\cap \Zh$ at a cost of order $O(x/h)$ operations by a nonnegative summation
(see Proposition~\ref{proposition:scale_functions_for_USF}). 

\noindent 
 %Thus, if the values of $\Wq_h$ on said range have already been precomputed, such a summation is preferable to employing~\eqref{equation:recursion:Zq} directly.   
%\item 
(3) The computation of the functions $\Wq$, $q\geq 0$, can be reduced, under an exponential change of measure, to the computation of $W$  \cite[p. 222, Lemma~8.4]{kyprianou} for a process having $\Phi(0)=0$ \cite{surya}. Under such an exponential tilting $W_h(x)$ will have a temperate growth \cite[Proposition 4.8(ii)]{vidmar:fluctuation_theory}, since $\Phi^h(0)\to \Phi(0)$. \noindent
%See also the discussion on this in \cite{surya}. 
%\item For large $x$ and sufficiently rapidly decaying tail of the L\'evy measure at $-\infty$, a truncation of the sum in \eqref{eq:LinRecursion} could be called for, neglecting very small terms (and thus improving the numerical complexity of the algorithm by reducing the number of operations). For, using $\lim_{\delta\downarrow 0}\gamma(\delta)+\zeta(\delta)+\xi(\delta)=0$ (and
%$\lim_{\delta \downarrow 0}\delta\measure[-1,-\delta)=0$, when
%$\kappa(0)<\infty$), we see that always (with $p(h):=\measure^h(\{h\})/\measure^h(\mathbb{R})$): $$\liminf_{h\downarrow 0}p(h)\geq 1/2$$ (in fact $\lim_{h\downarrow
%0}p(h)=1/2$, if $\diffusion>0$; and $\lim_{h\downarrow 0}p(h)=1$, if
%$\diffusion=0$ and $\kappa(0)<\infty$). Also $\measure^h(\mathbb{R})\to\infty$ as $h\downarrow 0$. Thus, qualitatively, the terms in the summation \eqref{eq:LinRecursion} will indeed decay, the faster as the tail of the L\'evy measure does so, whilst the growth of the scale function itself is temperate. A further (in particular, quantitative) analysis of such a truncation goes beyond the desired scope of this paper.
%\end{enumerate}
%\end{remark}

Finally, in comparison to the Laplace inversion methods discussed in \cite[Section~5.6]{kuznetsovkyprianourivero}, we note: % (again, further to what we have already put forth in the Introduction), as follows.
%\begin{enumerate}
%\item
\noindent (1)  Regarding \emph{only} the efficiency of our algorithm (i.e. how costly it is, to achieve a given precision): Firstly, that Filon's method (with Fast Fourier Transform) appears to outperform ours when an explicit formula for the Laplace exponent $\psi$ is known. Secondly, that our method is largely insensitive to the degree of smoothness of the target scale function -- and can match or outperform Euler's, the Gaver-Stehfest and the fixed Talbot's method in regimes when the scale function is less smooth, even as $\psi$ remains readily available (such, at least, was the case for Sets~3 and~4 of \cite[pp. 177-178]{kuznetsovkyprianourivero} --- see Appendix~\ref{appendix:further_examples}). Thirdly, that when $\psi$  is not given in terms of elementary/special function, any Laplace inversion algorithm, by its very nature, must resort to further numerical evaluations of $\psi$ (at complex values of its argument), which hinders its efficiency and makes it hard to control the error. Indeed, such evaluations of $\psi$ appear disadvantageous, as compared to the more innocuous operations required to compute the coefficients present in our recursion. 

\noindent (2) In our method there is only one spatial discretization parameter $h$ to vary. On the other hand, Filon's method (which, when coupled with Fast Fourier Transform, appears the most efficient of the Laplace inversion techniques), has additionally a cutoff parameter in the (complex) Bromwich integral.

\bibliographystyle{plain}
\bibliography{Biblio_scale_functions}

\appendix

\section{Proofs of lemmas from Paragraph~\ref{subsection:some_technical_estimates_and_bounds}}\label{appendix:technical_lemmas}
\begin{proof}(Of Lemma~\ref{lemma:technical:estimates}.)
Assertions \ref{lemma:technical:estimates:i}-\ref{lemma:technical:estimates:rough:II} obtain at once by expansion into Taylor series which converge absolutely and locally uniformly. \ref{lemma:technical:coercivity:xx} follows from the equality $1-\cos(u)=2\sin^2(u/2)$ and the concavity of $\sin\vert_{[0,\xi/2]}$. Then \ref{lemma:technical:coercivity:a} is got from (letting $u=\gamma_0+is_0$, $\{\gamma_0, s_0\}\subset \mathbb{R}$): $$\left\vert \frac{1}{u^2}\left(\cosh u-1\right)\right\vert^2=\frac{\left[(\cosh\gamma_0-1)+(1-\cos s_0)\right]^2}{(\gamma_0^2+s_0^2)^2}\geq \frac{(\gamma_0^2/2+1-\cos s_0)^2}{(\gamma_0^2+s_0^2)^2},$$ since $\cosh\vert_{\mathbb{R}}$, upon expansion into a power series, is a nondecreasing limit of its partial sums. Finally, \ref{lemma:technical:coercivity:b} obtains from (again $u=\gamma_0+is_0$, $\{\gamma_0, s_0\}\subset \mathbb{R}$): $$\left\vert \frac{1}{u}\left( e^u-1\right)\right\vert^2=\frac{(e^{\gamma_0}-1)^2+2e^{\gamma_0}(1-\cos s_0)}{\gamma_0^2+s_0^2},$$ noting that $\vert e^{\gamma_0}-1\vert\geq  B\vert\gamma_0\vert$ for all $\gamma_0\geq L$, for some $B>0$, whereas $\exp\vert_{[L,\infty)}$ is itself bounded away from zero.
\end{proof}

\begin{proof}(Of Lemma~\ref{lemma:convergence:a_uniformly_bounded_family}.)
Expressing $f_{\gamma_0}(u):=\frac{1-e^{-\gamma_0-is_0}-(\gamma_0+is_0)}{(1-e^{-\gamma_0-is_0})(\gamma_0+is_0)}$, use \ref{lemma:technical:coercivity:b} of Lemma~\ref{lemma:technical:estimates} in the denominator, and expansion into Taylor series which converge absolutely and locally uniformly in the numerator. 
\end{proof}

\begin{proof}(Of Lemma~\ref{lemma:fundamental:inequalities}.)
The first inequality follows from (writing $z=\gamma_0+is_0$, $\{\gamma_0,s_0\}\subset \mathbb{R}$): $$\vert e^{z}-1\vert^2=(e^{\gamma_0}-1)^2+2e^{\gamma_0}(1-\cos s_0).$$ Then, since $\gamma_0\leq 0$, $e^{\gamma_0}\leq 1$ and $1-e^{\gamma_0}\leq -\gamma_0$ (by comparing derivatives). Finally, use $1-\cos s_0\leq s_0^2/2$.

The second inequality obtains from the relation (for $\{\gamma_0,s_0\}\subset \mathbb{R}$): $$e^{\gamma_0+is_0}-(\gamma_0+is_0)-1=e^{\gamma_0}(\cos(s_0)-1)+(e^{\gamma_0}-\gamma_0-1)+i(e^{\gamma_0}-1)\sin(s_0)+i(\sin(s_0)-s_0),$$ noting in addition, that $e^{\gamma_0}-\gamma_0-1\leq \gamma_0^2/2$ for $\gamma_0\leq 0$ (compare derivatives), $\vert \sin(s_0)\vert\leq \vert s_0\vert$ and finally $\mathrm{sgn}(s_0)(s_0-\sin(s_0))\leq (\vert s_0\vert^3/6)\land (2\vert s_0\vert)\leq 2s_0^2$.
\end{proof}

\begin{proof}(Of Lemma~\ref{lemma:technical:2ndorder}.)
Apply the complex Mean Value Theorem \cite[p. 859, Theorem 2.2]{evard}.
\end{proof}

\section{Proofs for Paragraph~\ref{subsection:some_asymptotic_properties_at_zero}}\label{appendix:some_asymptotic_properties_of_measures_on_R}

\begin{proof}(Of Proposition~\ref{appendix:proposition:integrability:main}.)
It is Fubini's Theorem that: $$I:=\int_1^\infty \frac{ds}{s^2}\int \nu(dx)(1-\cos(sx))=\int \nu(dx)\int_1^\infty \frac{ds}{s^2}(1-\cos(sx)).$$ Next, for each $x\in \mathbb{R}$, do integration by parts for the integral $\int_1^\infty \frac{ds}{s^2}(1-\cos(sx))$; first, $$\frac{d}{ds}\left(\frac{1- \cos(sx)}{s}\right)=-\frac{1-\cos(sx)}{s^2}+x\frac{\sin(sx)}{s};$$ second, integrate from $1$ to $N$ against $ds$; third, let $N\to\infty$ and use the Monotone Convergence Theorem. We obtain: $$I=\int\nu(dx)\left(x\int_1^\infty\frac{\sin(sx)}{s}ds+(1-\cos(x))\right).$$ The inner integral is, of course, in the improper Riemann sense; now change variables in the latter to get: $$I=\int\nu(dx)\left(\vert x\vert\int_{\vert x\vert}^\infty\frac{\sin(u)}{u}du+(1-\cos(x))\right).$$ Since the sine integral is bounded, $1-\cos(x)\leq x^2/2$ and $\nu$ is compactly supported, it follows that $I<\infty$ whenever $\int\vert x\vert \nu(dx)<\infty$. Conversely, if $\int \vert x\vert \nu(dx)=\infty$, since $\int_{\vert x\vert}^\infty\frac{\sin(u)}{u}du\to \int_0^\infty \frac{\sin(u)}{u}du\in (0,\infty)$ as $x\to 0$, we deduce $I=\infty$ by the local finiteness of $\nu$ in $\mathbb{R}\backslash \{0\}$. 
\end{proof}

\begin{proof}(Of Lemma~\ref{lemma:fubini}.)
Fubini's Theorem.
\end{proof}

\begin{proof}(Of Proposition~\ref{proposition:generally_on_measure_asymptotics}.)
Apply Lemma~\ref{lemma:fubini}.
\end{proof}

\begin{proof}(Of Proposition~\ref{proposition:generally_on_measure_asymptotics_bis}.)
The first assertion follows at once from \ref{fubini:three} of Lemma~\ref{lemma:fubini} (applied to the measure $\nu$). The second follows from the first and \ref{fubini:two} of Lemma~\ref{lemma:fubini} applied to the measure $\lambda$. Now we consider $\int (e^{isy}-1)\lambda(dy)$. It is assumed without loss of generality that $\nu$ is supported by the positive half-line and $s>0$. Furthermore, there is an $A\in (0,\infty)$, with $\hat{\lambda}(u):=\lambda(u,1)\leq A/u^\alpha$ for all $u\in (0,1)$.
% and it will be sufficient to maintain $\sup_{s\geq 3\pi/2}\frac{1}{s^\alpha}\vert \int (e^{isy}-1)\lambda(dy)\vert<\infty$. 

Consider the imaginary part of $\int (e^{isy}-1)\lambda(dy)$ first. Fubini and dominated convergence yield: 
\footnotesize 
$$\int \lambda(dy)\sin(sy)=s\int \lambda(dy)\int_0^y\cos(su)du=s\int_0^1du\cos(su)\hat{\lambda}(u)=s\sum_{k=0}^\infty \int_{(0,1)\cap [k\pi/2,k\pi+\pi/2)}du\cos(su)\hat{\lambda}(u)=:I(s).$$ 
\normalsize
%Clearly: $$I_1(s):=s\left[\int_\mu^1du\cos(su)\hat{\lambda}(u)+\int_0^\mu du \cos(su)\hat{\lambda}(\mu)\right]=\int (\mathbbm{1}_{(\mu,1)}\cdot \lambda)(dy)\sin(sy)$$ satisfies $\sup_{s\geq 3\pi/(2\mu)}\vert I_1(s)\vert /s^\alpha<\infty$. Similarly this holds true of $I_2(s):=-s\int_0 ^\mu du\cos(su)\hat{\lambda}(\mu)$ in isolation. It remains to argue then, that $I_3(s):=s\int_0^\mu du\cos(su)\hat{\lambda}(u)$ satisfies this condition. 
We thus obtain an alternating series, whence: $\vert I(s)\vert\leq s \int_0^{3\pi/(2s)}du\hat{\lambda}(u)$. This is crucial. Indeed, taking into account that $\hat{\lambda}(u)\leq A/u^\alpha$ throughout the integration region, we now obtain immediately  $\sup_{s>0}\vert I(s)\vert /s^\alpha<\infty$.

The real part is treated in a similar vein. First, Fubini yields: $$\int \lambda(dy)(1-\cos(sy))=s\int \lambda(dy)\int_0^y\sin(su)du=s\int_0^1du\sin(su)\hat{\lambda}(u)$$ and the remaining steps are very similar and hence omitted. 
\end{proof}

\section{More numerical examples}\label{appendix:further_examples}

\begin{figure}[!htb]
%\vspace{-0.5cm}
\includegraphics[width=1\textwidth]{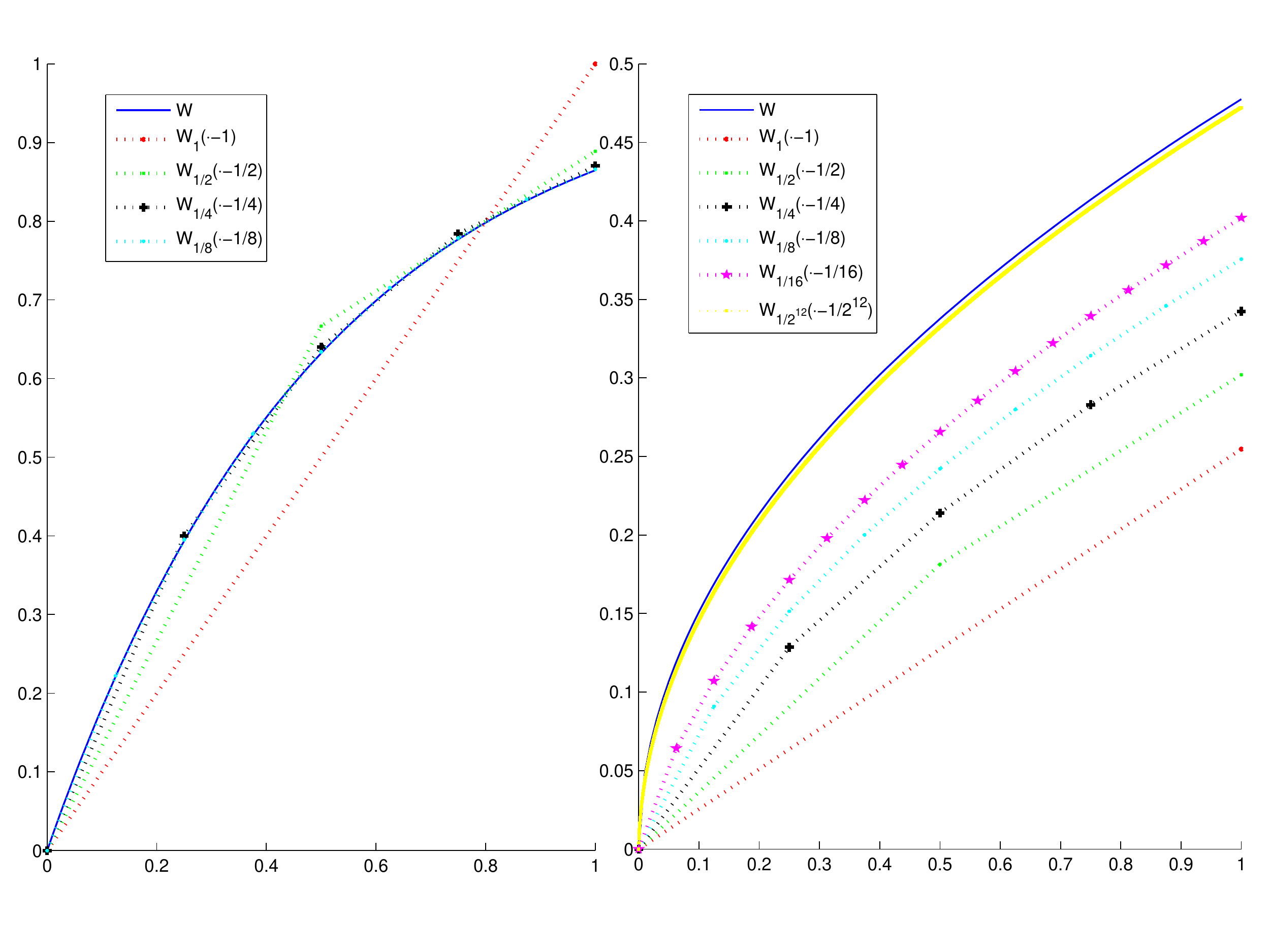}
%\vspace{-1cm}
\caption{Left: Brownian motion with drift ($\diffusion=1$, (with $V=0$) $\mu=1$, $\measure=0$). Right: a stable process ($\diffusion=0$, (with $V=1$) $\mu=1/(\beta-1)$, $\measure(dy)=\frac{1}{(-y)^{1+\beta }}\mathbbm{1}_{(-\infty,0)}(y)dy$, $\beta=3/2$). Note the quadratic, respectively `square-root', order of convergence.}
\label{fig:W_convergence_BM_drift}
\end{figure}

\begin{example}\label{example:compare}
This example is taken from \cite[pp. 177-178; Eq.~(137), Table~4 (Sets~3 and~4)]{kuznetsovkyprianourivero}; with two \emph{caveats}: (i) parameter $c_4$ is set equal to $e$ in Set~3, since there appears to be a typo in the last term of the expression for the Laplace exponent \cite[p. 178, Equation~(138)]{kuznetsovkyprianourivero} and presumably the form of the Laplace exponent as given (\emph{sic}), is the one which was used in the actual computations of \cite{kuznetsovkyprianourivero}; (ii) our parameter $\mu$ (with $V=1$) takes on such values as to ensure equality of Laplace exponents (i.e. laws of the processes) between us and \cite{kuznetsovkyprianourivero} for each of the two sets. We perform the computation for the function $W^{(1/2)}$ on the decreasing sequence $(h_n=1/(20\cdot n))_{n\in \{5,20,80,320\}}$.  $\sc_{h_{2000}}^{(1/2)}(\cdot-\delta_0h_{2000})$ is taken as a benchmark, and this gives $\text{max rel}_n:=\max_{i\in [100]}\frac{\vert W_{h_n}^{(1/2)}(x_i-\delta_0h_n)-W_{h_{2000}}^{(1/2)}(x_i-\delta_0h_{2000})\vert}{W_{h_{2000}}^{(1/2)}(x_i-\delta_0h_{2000})}$, where $x_i=i/20$, $i\in [100]$ . Note the linear order of convergence.
\vspace{0.25cm}
%Note the $O(h^{1/2})$ (Sets~1 and~2) and linear (Sets~3 and~4) order of convergence, respectively.

\begin{center}
%\begin{table}[!htb]
\begin{tabular}{|c|c|c|c|c|}\hline
$n$ & $5$ & $20$ & $80$ & $320$  \\\hline
%$\text{max rel}_n$ (Set 1) & $ 0.0011$ & $3.4086\cdot 10^{-5}$ & $6.6513\cdot 10^{-6}$ & $1.2391\cdot 10^{-6}$  \\\hline
%$\text{max rel}_n$ (Set 2) & $0.0121$ & $0.0031$ & $7.5203\cdot 10^{-4}$ & $1.6512\cdot 10^{-4}$ \\\hline 
$\text{max rel}_n$ (Set 3) & $ 0.0011$ & $3.4086\cdot 10^{-5}$ & $6.6513\cdot 10^{-6}$ & $1.2391\cdot 10^{-6}$  \\\hline
$\text{max rel}_n$ (Set 4) & $0.0121$ & $0.0031$ & $7.5203\cdot 10^{-4}$ & $1.6512\cdot 10^{-4}$ \\\hline 
\end{tabular}
\label{table:sets}
\end{center}

%\caption{In Example~\ref{example:compare} convergence is considered on the decreasing sequence $(h_n=1/(20\cdot n))_{n\in \{5,20,80,320\}}$ and on the interval $[0,5]$. The corresponding number of computed values of $W_{h_n}^{(1/2)}$ is $N_n=100\cdot n$ in each case. Those got using $h_{2000}$ are taken as a benchmark against which the others are compared, this gives $\text{max rel}_n:=\max_{i\in [100]}\frac{\vert W_{h_n}^{(1/2)}(x_i-\delta_0h_n)-W_{h_{2000}}^{(1/2)}(x_i-\delta_0h_{2000})\vert}{W_{h_{2000}}^{(1/2)}(x_i-\delta_0h_{2000})}$, where $x_i=i/20$, $i\in [100]$. Note the linear convergence. Remark also that it is difficult to ensure comparability of computational times, so we choose not to report these, but focus instead on the relative errors. In any event, the required number of operations for our algorithm has been identified as being of order $O(N_n^2)$ and suffice it to say the computational times were consistent with this bound.}\label{tbl:W_convergence_compare}
%\end{table}
\end{example}

\begin{figure}[!htb]
\includegraphics[width=\textwidth]{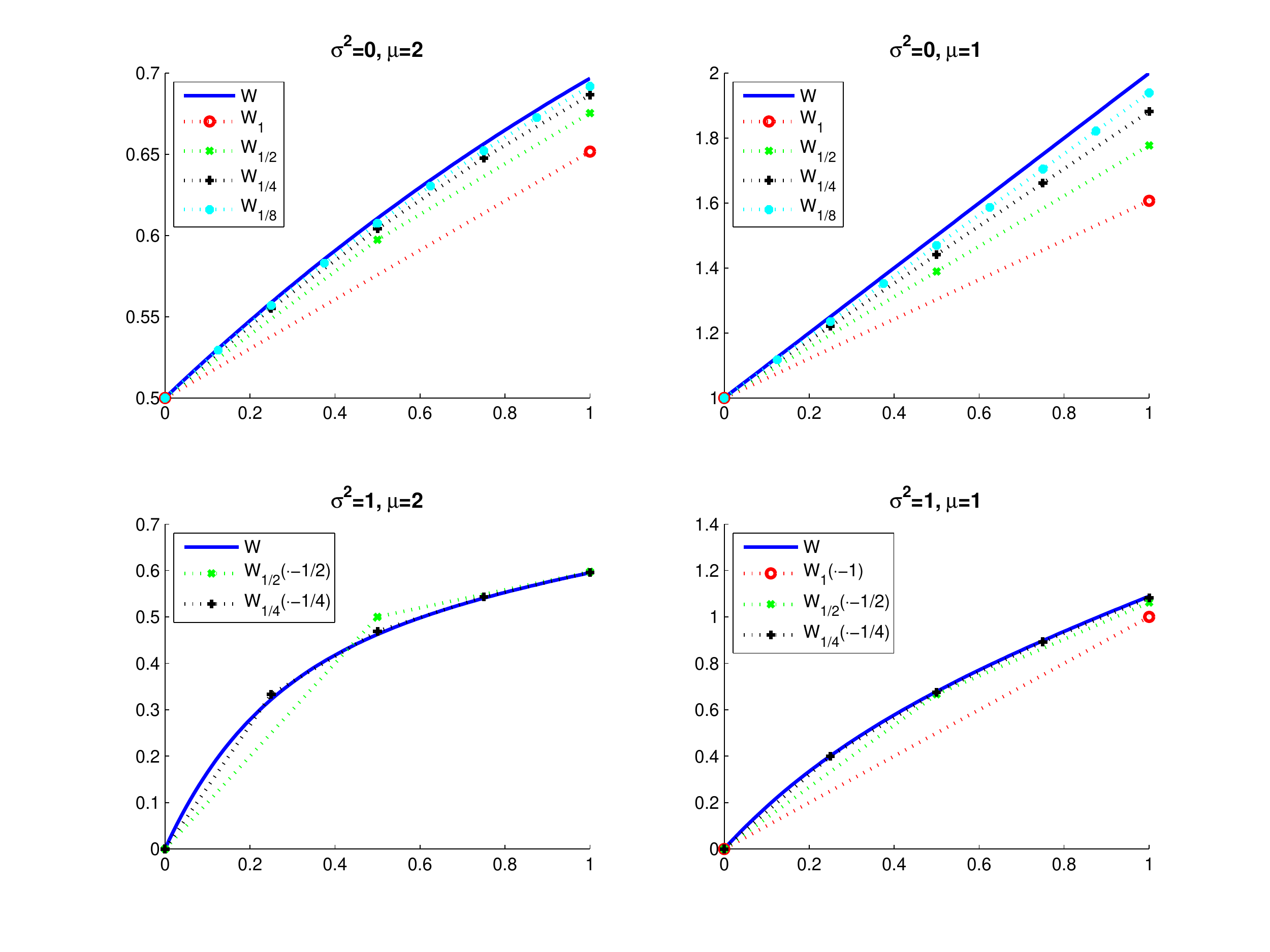}
%\vspace{-1cm}
\caption{Exponential jumps model ($\measure(dy)=a\rho e^{\rho y}\mathbbm{1}_{(-\infty,0)}(y)dy$, $a=\rho=1$, $\diffusion=0$ and (with $V=0$) $\mu\in \{2,1\}$). Note the linear order of convergence. }
%\vspace{-0.5cm}
\label{fig:W_convergence_Exp_jumps}
\end{figure}

%\begin{figure}[!htb]
%\includegraphics[width=\textwidth]{Convergence_W_stable.pdf}
%%\vspace{-1cm}
%\caption{A stable process ($\diffusion=0$, (with $V=1$) $\mu=1/(\beta-1)$, $\measure(dy)=\frac{1}{(-y)^{1+\beta }}\mathbbm{1}_{(-\infty,0)}(y)dy$, $\beta=3/2$). Note the $O(h^{1/2})$ order of convergence.}
%\label{fig:W_convergence_stable_jumps}
%\end{figure}

\end{document}